\documentclass[11pt]{article}
\usepackage[utf8]{inputenc} 
\usepackage[T1]{fontenc}
\usepackage{lmodern}
\usepackage[a4paper,top=3cm, bottom=3cm, left=2.5cm, right=2.5cm]{geometry} 
\usepackage{amsmath,amsthm,amsfonts,amssymb,stmaryrd,bbm,mathrsfs, dsfont}
\usepackage{mathtools} 
\usepackage{graphicx} 
\usepackage{enumitem}
\usepackage{url}
\usepackage{fancyhdr} 
\usepackage[Sonny]{fncychap} 
\usepackage{appendix}

\usepackage[french, english]{babel} 
\usepackage{upgreek}
\usepackage[usenames,dvipsnames]{xcolor}
\usepackage[labelfont={bf,sf}, textfont={sf}]{caption}
\usepackage{subfigure}
\usepackage{tikz}
\usetikzlibrary{shapes}
\usetikzlibrary{patterns}

\usepackage{cite}
\usepackage{oands}
\usepackage{changepage}
\usepackage[pagewise,mathlines]{lineno}
\usepackage{multicol}
\usepackage{microtype}
\usepackage[colorinlistoftodos]{todonotes}

\usepackage
[final]
{changes}
\makeatletter
\@namedef{Changes@AuthorColor}{magenta}
\colorlet{Changes@Color}{magenta}
\makeatother

\usepackage{float}
\usepackage{chngcntr}

\definecolor{darkblue}{rgb}{0.0, 0.2, 0.6}
\usepackage[pdftitle={},
  pdfauthor={},
colorlinks=true,linkcolor=RoyalBlue,urlcolor=MidnightBlue,citecolor=darkblue,bookmarks=true,bookmarksopen=true,bookmarksopenlevel=2,unicode=true,linktocpage]{hyperref}

\numberwithin{equation}{section}

\usepackage{fix-cm}
\rmfamily
\DeclareFontShape{T1}{lmr}{b}{sc}{<->ssub*cmr/bx/sc}{}
\DeclareFontShape{T1}{lmr}{bx}{sc}{<->ssub*cmr/bx/sc}{}

\fancypagestyle{plain}{%
	\fancyhf{}%
	\fancyfoot[C]{}%
}
\ChNameVar{\LARGE\bf\flushleft}
\ChRuleWidth{1pt}
\ChNumVar{\LARGE\bf}
\ChNameAsIs
\ChTitleVar{\Large\bf\centering}

\makeatletter
\renewcommand\paragraph{%
	\@startsection{paragraph}
	{4}
	{\z@}
	{3.25ex \@plus1ex \@minus.2ex}
	{-1em}
	{\normalfont\normalsize\bfseries\maybe@addperiod}%
}
\newcommand{\maybe@addperiod}[1]{%
	#1\@addpunct{.}%
}
\makeatother

\setlist[enumerate,1]{label=(\roman*), font = \normalfont} 

\let\originalleft\left
\let\originalright\right
\renewcommand{\left}{\mathopen{}\mathclose\bgroup\originalleft}
\renewcommand{\right}{\aftergroup\egroup\originalright}

\newlength{\bibitemsep}
\newlength{\bibparskip}
\let\oldthebibliography\thebibliography
\renewcommand\thebibliography[1]{\oldthebibliography{#1}
  \setlength{\parskip}{\bibitemsep}
  \setlength{\itemsep}{\bibparskip}}

\usepackage[capitalise,noabbrev]{cleveref}

\theoremstyle{plain}  

\theoremstyle{definition}

\theoremstyle{remark}



\newcommand\N{\mathbbm{N}}

\newcommand\R{\mathbbm{R}}
\renewcommand\C{\mathbbm{C}}
\newcommand\Ub{\mathbbm{U}}

\newcommand\Hb{\mathbbm{H}}

\newcommand\xihat{\widehat{\xi}}
\newcommand\psihat{\widehat{\psi}}
\newcommand\qhat{\widehat{q}}
\newcommand\Uhat{\widehat{U}}

\newcommand\Xcal{\mathcal{X}}
\newcommand\Xbf{\mathbf{X}}
\newcommand\Mcal{\mathcal{M}}
\newcommand\Lcal{\mathcal{L}}
\newcommand\Xhat{\widehat{\mathcal{X}}}
\newcommand\Zbf{\mathbf{Z}}
\newcommand\Kcal{\mathcal{K}}
\newcommand\Fcal{\mathcal{F}}
\newcommand\Ehat{\widehat{E}}
\newcommand\Fhat{\widehat{F}}
\newcommand\Xbfhat{\widehat{\mathbf{X}}}
\newcommand\Hhat{\widehat{H}}

\newcommand\Pb{\mathbbm{P}}
\newcommand\Eb{\mathbbm{E}}
\newcommand\Pcal{\mathcal{P}}
\newcommand\Ecal{\mathcal{E}}
\newcommand\Phat{\widehat{\mathcal{P}}}

\newcommand\F{\mathscr{F}}
\newcommand\Pbf{\mathbf{P}}
\newcommand\Ebf{\mathbf{E}}
\newcommand\gtilde{\gamma}

\newcommand\n{\mathfrak{n}}
\newcommand\tb{t^{\bullet}}
\newcommand\zb{z^{\bullet}}
\newcommand\Xscr{\mathscr{X}}
\newcommand{\efrak}{\mathfrak{e}}

\newcommand\Gscr{\mathscr{G}}
\newcommand\Gcal{\mathcal{G}}

\newcommand{\hfrak}{\mathfrak{h}}

\newcommand{\sgn}[1]{sgn(#1)}

\newcommand{\Pbb}[3]{\Pb_{#1}^{#2\rightarrow #3}}
\newcommand{\Ebb}[3]{\Eb_{#1}^{#2\rightarrow #3}}

\newcommand\diag{\text{diag}}

\theoremstyle{plain}
\newtheorem{Thm}{Theorem}[section]
\newtheorem{Thm-fr}{Théorème}[section]
\newtheorem*{Thm*}{Theorem}
\newtheorem{Prop}[Thm]{Proposition}
\newtheorem*{Prop*}{Proposition}

\newtheorem*{Def*}{Definition}
\newtheorem{Lem}[Thm]{Lemma}
\newtheorem*{Cor*}{Corollary}
\newtheorem{Cor}[Thm]{Corollary}

\theoremstyle{definition}

\newtheorem*{Ex*}{Example}
\newtheorem{Rk}[Thm]{Remark}
\newtheorem*{Rmq*}{Remarques}

\newtheorem{Not}[Thm]{Notation}

\begin{document}

\vglue30pt

\centerline{\Large\bf  \textsc{Self-similar signed growth-fragmentations} }

\bigskip
\medskip

 \centerline{William Da Silva\footnote{\scriptsize LPSM, Sorbonne Universit\'e Paris VI, {\tt william.da-silva@lpsm.paris}}}

\bigskip

\bigskip

{\leftskip=2truecm \rightskip=2truecm \baselineskip=15pt \small


\noindent{\slshape\bfseries Summary.} 
Growth-fragmentation processes model the evolution of positive masses which undergo binary divisions. The aim of this paper is twofold. First, we extend the theory of growth-fragmentation processes to allow signed mass. Among other things, we introduce genealogical martingales and establish a spinal decomposition for the associated cell system, following \cite{BBCK}. Then, we study a particular family of such self-similar signed growth-fragmentation processes which arise when cutting half-planar excursions at horizontal levels. When restricting this process to the positive masses in a special case, we recover part of the family introduced by Bertoin, Budd, Curien and Kortchemski in \cite{BBCK}.

\medskip


\medskip
 

} 

\bigskip
\bigskip

%
%
\section{Introduction}
Markovian growth-fragmentation processes were first introduced in \cite{Ber-GF}. They describe a system of positive masses, starting from a unique cell called the Eve cell, which may evolve over time, and suddenly split into a mother cell and a daughter cell. This happens with conservation of mass at dislocations: the sum of the mother and daughter masses after dislocation is equal to the mass of the mother cell right before. The latter daughter cells then evolve independently of each other, with the same stochastic evolution as the mother cell, thereafter dividing in the same way, and giving birth to granddaughter cells, and so on. Thus, newborn particles arise according to the \emph{negative} jumps of the mass of the mother cell. 

Self-similar growth-fragmentation processes form a rich subclass of such models and have been extensively studied in the seminal article \cite{BBCK}. Natural genealogical martingales in particular arise from the so-called additive martingales in the branching random walk setting (see the Lecture notes \cite{ZShi}). These martingales depend on exponents which can be found as the roots of the growth-fragmentation \emph{cumulant function} $\kappa$. Performing the corresponding change of measure, \cite{BBCK} then completely describes the \emph{spinal decomposition} of the growth-fragmentation cell system. Under the new tilted probability measure, all the cells roughly behave as in the original cell system, except for the Eve cell, which behaves as some tilted version of it.

The article \cite{BBCK} further introduces a remarkable family of such processes that are closely linked to $\theta$--stable Lévy processes for $\frac12<\theta\le \frac32$, and relates them to the scaling limit of the exploration of a Boltzmann planar map. The case $\theta=\frac32$ was later recovered up to a time-change in \cite{LGR} when studying, among other things, the boundary sizes of superlevel sets of Brownian motion indexed by the Brownian tree, whereas the critical case $\theta=1$ appears in \cite{AD} when cutting a half-planar Brownian excursion at horizontal levels. The growth-fragmentation processes associated with parameters $\theta\in(1,\frac32)$ were also retrieved directly in the continuum in \cite{MSW} by exploring a conformal loop ensemble on an independent quantum surface.

In \cite{AD}, the masses of the cells correspond to the sizes of the excursions above horizontal levels, defined as the difference between the endpoint and the starting point. Note that, for these to fall into the growth-fragmentation framework, one has to remove all the negative excursions of the system. Moreover, a Boltzmann planar map can be seen as the gasket of a loop $O(n)$ model (see \cite{LGM}), and from this standpoint, a positive jump in the growth-fragmentation represent the discovery of a loop which has not yet been explored.

\bigskip

The present work extends the study in \cite{BBCK} to the case when the masses may be \emph{signed}. We therefore allow positive jumps to be birth events and give rise to negative cells, so that the conservation rule still holds at dislocations. Markov additive processes and the Lamperti-Kiu representation provide a very natural framework for this. 

We illustrate this with the following examples, which are a slight generalisation of the growth-fragmentation embedded in the half-planar Brownian excursions studied in \cite{AD}. For $z_0>0$, we consider an excursion from $0$ to $z_0$ in the upper half-plane $\Hb$, where the imaginary part is a one-dimensional Brownian excursion, but the real part is some instance of an $\alpha$--stable process, with $\alpha\in(1,2)$. For $a>0$, if the excursion hits the horizontal level $\{z\in \C, \; \Im(z)=a\}$, it will make a countable number of excursions $(e_i^{a,+}, \; i\ge 1)$ above this level. We record the \emph{sizes} $(\Delta e_i^{a,+},\; i\ge 1)$ of these excursions, defined as the difference between the endpoint and the starting point. Because both points have the same imaginary part, this yields a collection of real numbers. Our main result investigates the branching structure of this collection of sizes: we show that it behaves as a signed self-similar growth-fragmentation process. Furthermore, when removing the negative sizes in the genealogy, we prove that this model gives back the family of growth-fragmentation processes introduced in \cite{BBCK} for $\frac12<\theta<1$. 

\bigskip

\added{\textbf{Related work.} In the pure fragmentation framework, multitype self-similar fragmentation processes have been introduced and their structure described, in terms of the underlying Markov additive process, in \cite{Ste}.}

\bigskip

The paper is organised as follows. In Section \ref{sec:real-valued SSMP}, we provide some background on real-valued self-similar Markov processes and their Lamperti-Kiu representation. In Section \ref{sec:martingales}, we make use of a connection with multitype branching random walks to introduce genealogical martingales similar to the ones in \cite{BBCK}. Section \ref{sec:universality} is devoted to proving that the form of these martingales only depend on the growth-fragmentation itself. Along the way, we will define \emph{signed cumulant functions} that are the analogs of the cumulant function in the positive case. The spinal decomposition will be described in Section \ref{sec:spinal decomposition}. Finally, we will investigate in Section \ref{sec:excursion} a distinguished family of self-similar signed growth-fragmentations constructed by cutting half-planar excursions at horizontal levels.

\bigskip

\textbf{Acknowledgements:} I am grateful to \'{E}lie Aïdékon for suggesting me this project, and for his involvement and guidance all the way. I also thank Juan Carlos Pardo for stimulating discussions and for correcting some mistakes in a preliminary version of this paper, as well as Alex Watson for discussions about the change of measures in \cref{sec: martingale absorption excursions}. I am grateful to Grégory Miermont for suggesting some important changes. The last details of this work were completed at the University of Vienna, supported by the FWF grant P33083 on ``Scaling limits in random conformal geometry''. \added{After completing this work, I was informed that the link between growth-fragmentation processes and half-planar excursions was already predicted by Timothy Budd in an unpublished note. Finally, I want to thank an anonymous referee for their valuable comments and suggestions.}

%
%
\section{Real-valued self-similar Markov processes}
\label{sec:real-valued SSMP}
We first recall some aspects of the Lamperti-Kiu theory for real-valued self-similar Markov processes. The Lamperti representation \cite{Lam} reveals a correspondence between positive self-similar Markov processes and Lévy processes. In the real-valued case, there is a more general correspondence, called the Lamperti-Kiu representation, between self-similar Markov processes and Markov additive processes, which are needed to take into account the sign changes.   

\subsection{Markov additive processes} \label{sec:MAP}
Let $E$ be a finite set, endowed with the discrete topology, and $(\Gcal_t)_{t\ge 0}$ a standard filtration. A \emph{Markov additive process} (MAP) with respect to $(\Gcal_t)_{t\ge 0}$ is a càdlàg process $(\xi,J)$ in $\R\times E$ with law $\Pb$, such that $(J(t), t\ge 0)$ is a continuous--time Markov chain, and the following property holds: for all $i\in E$, $t\ge 0$, 
\begin{center}\emph{Conditionally on $J(t)=i$, the process $(\xi(t+s)-\xi(t),J(t+s))_{s\ge 0}$ is independent of $\Gcal_t$ and is distributed as $(\xi(s)-\xi(0),J(s))_{s\ge 0}$ given $J(0)=i$.}
\end{center}
We shall write $\Pb_i:=\Pb(\;\cdot\; |\,\xi(0)=0 \,\text{and}\, J(0)=i)$ for $i\in E$. Details on MAPs can be found in \cite{Asm}. In particular, their structure is known to be given by the following proposition.
\begin{Prop} \label{prop:MAP structure}
\added{The process $(\xi,J)$ is a Markov additive process if, and only if, there exist independent sequences $(\xi_i^n, n\ge 0)_{i\in E}$ and $(U_{i,j}^n, n\ge 0)_{i,j\in E}$, all independent of $J$, such that:}
\begin{itemize}
    \item for $i\in E$, $(\xi_i^n, n\ge 0)$ is a sequence of i.i.d. Lévy processes,
    \item for $i,j\in E$, $(U_{i,j}^n, n\ge 0)$ are i.i.d.,
    \item if $(T_n)_{n\ge 0}$ denotes the sequence of jump times of the chain $J$ (with the convention $T_0=0$), then for all $n\ge 0$,
\begin{equation} \label{eq:piecewise}
\xi(t) = \left(\xi(T_n^-)+ U_{J(T_n^-),J(T_n)}^n\right)\mathds{1}_{n\ge 1} + \xi_{J(T_n)}^n(t-T_n), \quad T_n\le t<T_{n+1}.
\end{equation}
\end{itemize}
\end{Prop}
\noindent Proposition \ref{prop:MAP structure} describes $(\xi(t),t\ge 0)$ as a concatenation of independent Lévy processes with law depending on the current state of $J$, with additional random jumps occurring whenever the chain $J$ has a jump.

We now turn to defining the \emph{matrix exponent} of a MAP, which replaces the Laplace exponent in the setting of Lévy processes. For simplicity, we assume that $E=\{1,\ldots, N\}$ and that $J$ is irreducible. We write $Q=(q_{i,j})_{1\le i,j\le N}$ for its intensity matrix. Also, we denote for all $i,j\in E$, all on the same probability space, by $\xi_i$ a Lévy process distributed as the $\xi_i^n$'s, and by $U_{i,j}$ a random variable distributed as the $U
^n_{i,j}$'s, with the convention $U_{i,i}=0$ and $U_{i,j}=0$ if $q_{i,j}=0$. Finally, we introduce the Laplace exponent $\psi_i$ of $\xi_i$ and the Laplace transform $G_{i,j}(z):= \Eb\left[ \mathrm{e}^{zU_{i,j}}\right]$ of $U_{i,j}$ (this defines a matrix $G(z)$ with entries $G_{i,j}(z)$). Then the matrix exponent $F$ of $(\xi,J)$ is defined as 
\begin{equation} \label{eq: F matrix}
F(z) := \diag(\psi_1(z),\ldots,\psi_N(z))+Q\circ G(z),    
\end{equation}
where $\circ$ denotes pointwise multiplication of the entries. Then the following equality holds for all $i,j\in E$, $z\in \C$, $t\ge 0$, whenever one side of the equality is defined:
\[
\Eb_i\left[ \mathrm{e}^{z\xi(t)} \mathds{1}_{J(t)=j}\right] = (\mathrm{e}^{F(z)t})_{i,j}.
\]
\subsection{The Lamperti-Kiu representation} \label{sec:LK}
In \cite{Lam}, Lamperti proved that \added{positive self-similar Markov processes} can be expressed as the exponential of a time-changed Lévy process. In the real-valued case, one has to track the sign changes, but the same kind of representation holds. Let $X$ be a real-valued Markov process, which under $\Pb_z$ starts from $z\ne 0$, and denote by $T_0$ its first hitting time of $0$. We assume that $X$ is \emph{self-similar with index $\alpha$} in the following sense: for any $c>0$ and for all $z\ne 0$, the law of $(cX(c^{-\alpha}t),t\ge 0)$ under $\Pb_z$ is $\Pb_{cz}$. The next theorem summarises the main result of \cite{CPR}. Though it may appear intricate at first glance, we insist that the gist of it is intrinsically simple. It should be streamlined as follows. As long as $X$ remains positive (\emph{resp.} negative), it evolves as the exponential (\emph{resp.} minus the exponential) of a time-changed Lévy process. The Lévy processes keeping track of the positive and negative parts must not necessarily be equal. In addition, an exponential clock (modulo time-change) rings every time the sign of $X$ changes, and at these times a special jump occurs (again, the two exponential clocks and the law of the jumps may be different depending on the current sign of $X$).
\begin{Thm}(Lamperti-Kiu representation, \cite{CPR}) \label{thm:LK}

There exist independent sequences $(\xi^{\pm,k})_{k\ge 0}, (\zeta^{\pm,k})_{k\ge 0}, (U^{\pm,k})_{k\ge 0}$ of i.i.d. variables fulfilling the following properties:
\begin{enumerate}
    \item The $\xi^{\pm,k}$ are Lévy processes, the $\zeta^{\pm,k}$ are exponential random variables with parameter $q_{\pm}$, and the $U^{\pm,k}$ are real-valued random variables.
    \item For $z\ne 0$ and $k\ge 0$, if we define:
        \begin{itemize}
            \item[$\bullet$] 
                    $
                    (\xi^{z,k}, \zeta^{z,k}, U^{z,k}) 
                    :=
                    \begin{cases}
                    (\xi^{+,k}, \zeta^{+,k}, U^{+,k}) & \text{if} \; \sgn{z}(-1)^k = 1 \\
                    (\xi^{-,k}, \zeta^{-,k}, U^{-,k}) & \text{if} \; \sgn{z}(-1)^k=-1
                    \end{cases}
                    $
                    \item[$\bullet$] $\mathcal{T}^{(z)}_0:=0$ and $\mathcal{T}^{(z)}_n:=\sum_{k=0}^{n-1} \zeta^{(z,k)}$ for $n\ge 1$
                    \item[$\bullet$] $N^{(z)}_t:=\max\{n\ge 0, \; \mathcal{T}^{(z)}_n\le t\}$ and $\sigma^{(z)}_t:= t - \mathcal{T}^{(z)}_{N^{(z)}_t}$,
        \end{itemize}
    then, under $\Pb_z$, $X$ \textbf{has the representation}:
    \[
    X(t) = z\exp(\mathcal{E}^{(z)}_{\tau(t)}), \quad 0\le t< T_0,
    \]
    where 
    \[
    \mathcal{E}^{(z)}_t 
    :=
    \xi^{N^{(z)}_t}_{\sigma^{(z)}_t} + \sum_{k=0}^{N^{(z)}_t-1} \left(\xi^{(z,k)}_{\zeta^{(z,k)}}+ U^{(z,k)} \right) + i\pi N^{(z)}_t, \quad t\ge 0,
    \]
    and 
    \[
    \tau(t):= \inf\left\{s>0, \; \int_0^s |\exp(\alpha \mathcal{E}_u^{(z)})| \mathrm{d}u > t|z|^{-\alpha}\right\}, \quad t<T_0.
    \]
\end{enumerate}
Conversely, any process of this form is a self-similar Markov process with index $\alpha$.
\end{Thm}

This can be rephrased in the language of Markov additive processes, as was pointed out in \cite{KKPW}.

\begin{Prop} \label{prop:LK}
Let $X$ be a real-valued self-similar Markov process, with Lamperti-Kiu exponent $\mathcal{E}$. Introduce for $z\ne 0$,
\[
(\xi^{(z)}(t), J^{(z)}(t)):= \left(\Re(\mathcal{E}^{(z)}_t), \left[\frac{\Im(\mathcal{E}^{(z)}_t)}{\pi} + \mathds{1}_{z>0}\right]\right), \quad t\ge 0,
\]
where $[\cdot]$ denotes reduction \emph{modulo} $2$.

\noindent Then $(\xi^{(z)}(t), J^{(z)}(t))$ is a MAP with state space $\{0,1\}$ and under $\Pb_z$ for all $z\ne 0$,
\[
X(t) = z\exp\left(\xi^{(z)}(\tau(t)) + i\pi (J^{(z)}(\tau(t))+1)\right), \quad 0\le t < T_0,
\]
where, in terms of $\xi^{(z)}$, 
\[
\tau(t):= \inf\left\{s>0, \; \int_0^s \exp(\alpha \xi^{(z)}(u)) \mathrm{d}u > t|z|^{-\alpha}\right\}, \quad t<T_0.
\]
\end{Prop}

%
%
\section{Martingales in self-similar growth-fragmentation with signs}
\label{sec:martingales}
We follow closely \cite{Ber-GF} and \cite{BBCK} to extend the construction to real-valued driving processes. Let $X$ be a càdlàg Feller process which is self-similar in the sense of Section \ref{sec:LK}, with values in $\R^*$. Denote by $\Pb_z$ the law of $X$ started from $z\neq 0$, and assume that $X$ is either absorbed at a cemetery point $\partial$ after a finite time $\zeta$ or converges \added{(to 0)} as $t\rightarrow\infty$ under $\Pb_z$ for all $z$. Introduce the MAP $(\xi,J)$ associated to $X$ via the Lamperti-Kiu representation in Proposition \ref{prop:LK}, and denote by $F$ its matrix exponent. Recall that this matrix exponent is determined by the law of the Lévy processes $\xi_{\pm}$, special jumps $U_{\pm}$, and random clocks $\zeta_{\pm}$ (which are exponential with parameter $q_{\pm}$) dealing with the parts of the trajectory where $X$ is positive or negative. Recall also the notation $\Pb_{\pm}$ to denote the law of $X$ starting from $\pm 1$ (and \added{$\Eb_{\pm}$ for the corresponding expectation}). We further write $\Delta X(r) = X(r)-X(r^-)$ for the jump of $X$ at time $r$.

\subsection{Self-similar signed growth-fragmentation processes} \label{sec:growth-frag}
We now explain how to construct the \emph{cell system} driven by $X$. As usual, we label the cells using the Ulam tree $\Ub = \bigcup_{i=0}^{\infty} \N^i$, where in our notation $\N = \left\{1, 2\ldots \right\}$, and $\N^0:=\{\varnothing\}$ refers to the \emph{Eve cell}. A node $u\in \Ub$ is a list $(u_1,\ldots,u_i)$ of positive integers where $|u|=i$ is the \emph{generation} of $u$. Then the offspring of $u$ is labelled by the lists $(u_1,\ldots,u_i,k)$, with $k\in \N$.

We then define the cell system $(\Xcal_u(t),u\in \Ub)$ driven by $X$ by recursion. Again, we repeat the procedure in \cite{BBCK}, except that we include the positive jumps in the genealogy. First, set $b_{\varnothing}=0$ and $\Xcal_{\varnothing}$ to be distributed as $X$ started from some mass $z\ne 0$. Then at each jump of $\Xcal_{\varnothing}$, we will create a new particle with mass given by the opposite of this jump (so that the mass is conserved at each splitting). Since $X$ converges at infinity, one can rank the sequence of jump sizes and times $(x_1,\beta_1),(x_2,\beta_2),\ldots$ of $-\Xcal_{\varnothing}$ by descending lexicographical order for the absolute value of the $x_i$. Given this sequence of jumps, we define the first generation $\Xcal_i, i\in \N,$ of our cell system as independent processes with respective law $\Pb_{x_i}$, and we set $b_i = b_{\varnothing}+\beta_i$ for the \emph{birth time} of $i$ and $\zeta_i$ for its lifetime. The law of the $n$-th generation is constructed given generations $1,\ldots, n-1$ following the same procedure. Therefore, a cell $u'=(u_1,\ldots,u_{n-1})\in \N^{n-1}$ gives birth to the cell $u=(u_1,\ldots,u_{n-1},i)$, with lifetime $\zeta_u$, at time $b_u = b_{u'}+\beta_{i}$ where $\beta_{i}$ is the $i$-th jump of $\Xcal_{u'}$ (in terms of the same ranking as before). Moreover, conditionally on the jump sizes and times of $\Xcal_{u'}$, $\Xcal_u$ has law $\Pb_y$ with $-y=\Delta \Xcal_{u'}(\beta_i)$ and is independent of the other daughter cells at generation $n$. See Figure \ref{fig:GF_construction}.

Beware that, in this construction, the cells are not labelled chronologically. Nonetheless, exactly as in \cite{BBCK}, this uniquely defines the law $\Pcal_z$ of the cell system driven by $X$ and started from $z$. The cell system $(\Xcal_u(t),u\in \Ub)$ is meant to describe the evolution of a population of atoms $u$ with size $\Xcal_u(t)$ evolving with its age $t$ and fragmenting in a binary way. We stress once more that the difference with \cite{BBCK} is that the masses can be negative and that the genealogy also carries the positive jumps (which corresponds to negative newborn masses).

Then we may define for $t\ge 0$
\[
\Xbf(t) := \left\{\left\{ \Xcal_u(t-b_u), \; u\in\Ub \; \text{and} \; b_u\le t<b_u+\zeta_u \right\}\right\},
\]
where the double brackets denote multisets. In other words, the \emph{signed growth-fragmentation process} $\Xbf(t)$ is the collection of the sizes of all the cells in the system alive at time $t$. We denote by $\Pbf_z$ the law of $\Xbf$ started from $z$.

\begin{Rk}
This construction does not require $X$ to be self-similar.
\end{Rk}

\begin{figure}[h] 
\begin{center}
\includegraphics[scale=0.7]{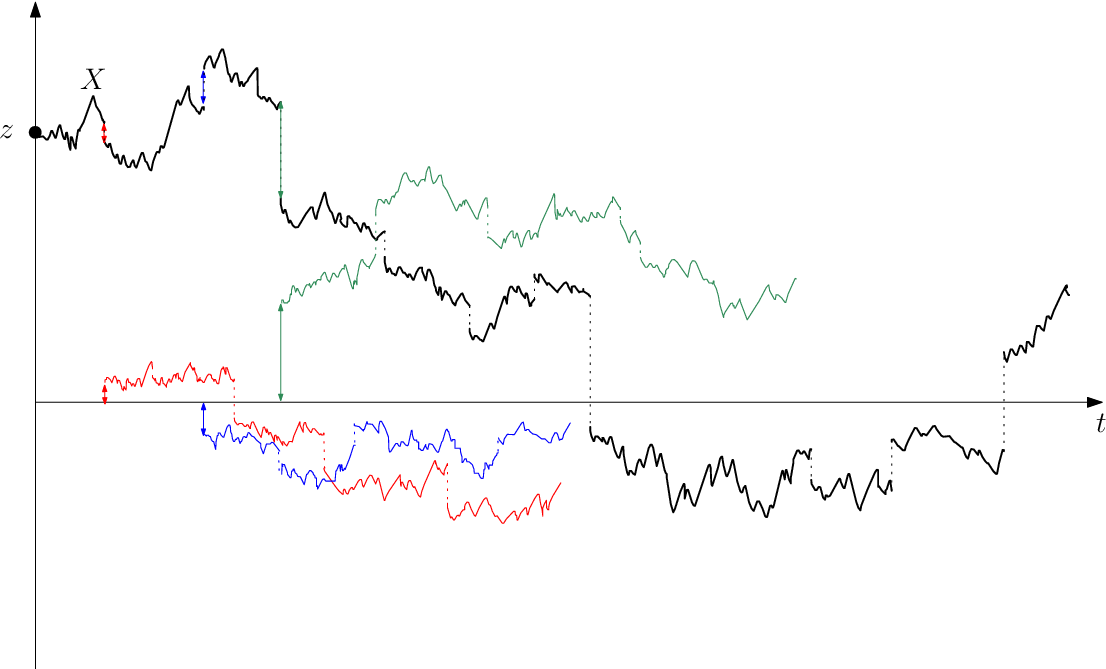}
\end{center}
\caption{Construction of the cell system from the driving process $X$. Each jump $\Delta X(s)$ of $X$ gives birth to a new particle (in colours), with size given by the intensity $-\Delta X(s)$ of the jump. These particles, in turn, give rise to the second generation (not shown in this figure).}
\label{fig:GF_construction}
\end{figure}

We now state a temporal version of the branching property. Introduce the natural filtration $(\Fcal_t)_{t\ge 0}$ of $(\Xbf(t),t\ge 0)$. As we shall need a stronger version of the branching property, we also record the generations by setting 
\[ 
\overline{\Xbf} (t) :=  \left\{\left\{ (\Xcal_u(t-b_u),|u|), \; u\in\Ub \; \text{and} \; b_u\le t<b_u+\zeta_u \right\}\right\}, \quad t\ge 0,
\]
and we denote by $(\overline{\Fcal}_t)_{t\ge 0}$ the filtration associated to $\overline{\Xbf}$. We assume that $\Xbf$ admits an \emph{excessive function}, that is a measurable function $f:\R^* \rightarrow [0,+\infty)$ such that $\underset{|x|>a}\inf f(x)>0$ for all $a>0$, and 
\begin{equation}\label{eq: excessiveness}
\forall z\in \R^*, \forall t\ge 0, \quad \Eb_z\left( \sum_{x\in \Xbf(t)} f(x)\right) \le f(z).
\end{equation}
In this case, we may rank the elements $X_1(t), X_2(t), \ldots$ of $\Xbf(t)$ in non-increasing order of their absolute value. For self-similar processes, the existence of such excessive functions will result from the assumptions that we will make later on. Then we have the following branching property, analogous to Lemma 3.2 in \cite{BBCK}. 
\begin{Thm} \label{thm:branching}
For every $t\ge 0$, conditionally on $\overline{\Xbf}(t) = \{\{(x_i,n_i), i\ge 1\}\}$, the process $(\overline{\Xbf}(t+s), s\ge 0)$ is independent of $\overline{\Fcal}_t$ and is distributed as 
\[
\bigsqcup_{i\ge 1} \overline{\Xbf}_i(s) \circ \theta_{n_i}, \quad s\ge 0,
\]
where the $\overline{\Xbf}_i$ are independent processes distributed as $\overline{\Xbf}$ under $\Pbf_{x_i}$, and $\theta_n$ is the shift operator $\{\{(y_j, k_j), j\ge 1\}\} \circ \theta_n := \{\{(y_j, k_j+n), j\ge 1\}\}$.
\end{Thm}
\noindent This theorem follows from the arguments given in Proposition 2 in \cite{Ber-GF} which holds as soon as $\Xbf$ has an excessive function. Under the same condition, the self-similarity of the driving cell process $X$ can be transferred to the whole growth-fragmentation process (see \cite[Theorem 2]{Ber-GF} for the nonnegative case, which extends easily).

\begin{Prop} For all $c>0$ and for all $z\ne 0$, the law of $(c\Xbf(c^{-\alpha}t),t\ge 0)$ under $\Pbf_z$ is $\Pbf_{cz}$.
\end{Prop}

\subsection{Multitype branching random walks and a genealogical martingale}
\label{sec:genealogical martingale}
We use a connection with branching random walks to exhibit a genealogical martingale as in \cite{BBCK}. We emphasize that the main difference with \cite{BBCK} is that in the signed case one has to deal with types (signs) in the branching structure, so that the relevant framework is provided by multitype branching random walks.
\\ \\\noindent \textbf{Multitype branching random walks.} We start by recalling the main features of multitype branching random walks. Let $I = \{1,\ldots, K\}$ be a set of \emph{types} with $K\ge 1$. The branching mechanism is then governed by $K$ random sequences of displacements $\Xi^{(k)} = (\xi^{(k)}_1,\ldots,\xi^{(k)}_{\nu})$, and types $\eta^{(k)}= (\eta^{(k)}_1,\ldots,\eta^{(k)}_{\nu})$, $k\in I$, where $\nu$ is also random and can be infinite. Denote by $\Pb_k$ the law of $(\Xi^{(k)}, \eta^{(k)})$.  Start from a particle $\varnothing$ at position $X_{\varnothing}=x\in \R$ and with initial type $k\in I$. At time $n=1$ this particle dies, giving birth to a random cloud of particles whose displacements from their parent and types are distributed as $\Xi^{(k)}$ and $\eta^{(k)}$ respectively. At time $n=2$, all these particles die and give birth in the same fashion to children of their own, independently of one another and independently of the past. This construction is repeated indefinitely as long as there are particles in the system. Therefore, if $u\in \Ub$ has type $i\in I$, it gives birth at the next generation to $\nu(u)$ particles with displacements $(X_{u1},\ldots, X_{u\nu(u)})$ and types $(i_{u1},\ldots,i_{u\nu(u)})$ according to $\Pb_i$. Also, we write $\llbracket \varnothing, u\rrbracket$ for the unique shortest path connecting the root to the node $u$, so that 
\[
V(u):= \sum_{v\in\llbracket \varnothing, u\rrbracket} X_v,
\]
is the position of particle $u$. We denote by $(\mathscr{F}_n)_{n\ge 0}$ the natural filtration of the multitype branching random walk, \textit{i.e.}
\[ 
\mathscr{F}_n := \sigma \left((X_u,i_u), \; |u|\le n \right).
\]
For $q\in\R$ define $m(q)$ to be the $K\times K$--matrix with entries
\[
m_{ij}(q):= \Eb_i\bigg[ \sum_{|v|=1} \mathrm{e}^{-qX_v} \mathds{1}_{i_v=j}\bigg].
\]
We then make the following assumption.
\medskip
\begin{center} \textit{Assumption (A0):} $\quad \forall i,j, \quad \Pb_i(\exists 1\le l \le\nu, \; \eta^{(i)}_l=j)>0.$ \end{center}
\medskip
Then the matrix $m(q)$ is positive and we may apply Perron-Frobenius theory. Let $\mathrm{e}^{\lambda(q)}$ be its largest eigenvalue and $v(q)=(v_1(q),\ldots,v_K(q))$ an associated positive eigenvector. Then we have the following result.

\begin{Thm} \label{thm:BRW martingale} For any $1\le i\le K$, under $\Pb_i$, the process
\[
M_n:= \sum_{|u|=n} v_{i_u}(q) \mathrm{e}^{-qV(u)-n\lambda(q)}, \quad n\ge 0,
\]
is a martingale with respect to the filtration $(\mathscr{F}_n)_{n\ge 0}$.
\end{Thm}
\begin{proof}
$(M_n)_{n\ge 0}$ is $(\mathscr{F}_n)_{n\ge 0}$--adapted, and for $n\ge 0$,
\begin{align}
\Eb_i\left[M_{n+1} \; | \; \mathscr{F}_n\right] &= \Eb_i \left[\sum_{|u|=n} \sum_{|w|=1} v_{i_{uw}}(q) \mathrm{e}^{-qV(u)-qX_{uw}-(n+1)\lambda(q)}\;\Bigg| \; \mathscr{F}_n \right] \notag \\
&= \sum_{|u|=n} \mathrm{e}^{-qV(u)-(n+1)\lambda(q)}\Eb_i\left[\sum_{|w|=1} v_{i_{uw}}(q) \mathrm{e}^{-qX_{uw}}\;\Bigg| \; \mathscr{F}_n \right]. \label{eq:BRW martingale}
\end{align}
By the branching property, for all $|u|=n$,
\[
\Eb_i\left[\sum_{|w|=1} v_{i_{uw}}(q) \mathrm{e}^{-qX_{uw}}\;\Bigg| \; \mathscr{F}_n \right]
=
\Eb_{i_u}\left[\sum_{|w|=1} v_{i_{w}}(q) \mathrm{e}^{-qX_{w}} \right].
\]
Since $v(q)$ is an eigenvector for $m(q)$, this is 
\[
\Eb_i\left[\sum_{|w|=1} v_{i_{uw}}(q) \mathrm{e}^{-qX_{uw}}\;\Bigg| \; \mathscr{F}_n \right]
=
\mathrm{e}^{\lambda(q)} v_{i_u}(q).
\]
Coming back to (\ref{eq:BRW martingale}), we get
\[
\Eb_i\left[M_{n+1} \; | \; \mathscr{F}_n\right] 
=
\sum_{|u|=n} v_{i_u}(q) \mathrm{e}^{-qV(u)-n\lambda(q)}
= M_n.
\] 
\end{proof}

\medskip
\noindent \textbf{The genealogical martingale of self-similar signed growth-fragmentations.}
It is easily seen from the self-similarity of the cell processes and the branching structure of growth-fragmentations that if $\sgn{x}$ denotes the sign function, the process $(-\log |\Xcal_u(0)|, \sgn{\Xcal_u(0)})_{u\in \Ub}$ is a multitype branching random walk, where the set of types is just $I=\{+,-\}$. Define
\[
\mathscr{G}_n := \sigma \left(\Xcal_u, \; |u|\le n \right), \quad n\ge 0.
\]
Note that in this setting, for any $u\in \Ub$ with $|u|=n\ge 1$, $\Xcal_u(0)$ is $\mathscr{G}_{n-1}$--measurable.  For $q\in \R$, $m(q)$ is now a $2\times 2$--matrix with entries
\begin{equation} \label{eq:matrix}
m_{ij}(q) = \Eb_i \bigg[ \sum_{s>0} |\Delta X(s)|^q \mathds{1}_{\sgn{-\Delta X(s)}=j} \bigg].
\end{equation}
We work under the following assumption, analogous to Assumption (A0). 
\medskip
\begin{center} \textbf{Assumption (A):} \emph{The process $X$ admits positive and negative jumps.} \end{center}
\medskip
\noindent Again, \added{for $q$ such that $m(q)$ is finite}, we denote by $e^{\lambda(q)}$ the Perron-Frobenius eigenvalue of $m$ and by $v(q)$ an associated positive eigenvector. We make the following additional assumption. 
\medskip
\begin{center}\textbf{Assumption (B):} \emph{There exists $\omega \in \R$ such that $\lambda(\omega) = 0$.} \label{assumption B} \end{center}
\medskip

\noindent We say that $(v_+,v_-,\omega)$ is \emph{admissible} for $X$ if $m(\omega)$ has Perron-Frobenius eigenvalue $1$ (\textit{i.e.} assumption (B) is satisfied) and $(v_+,v_-)$ is an associated positive eigenvector. This assumption is of \emph{Malthusian} nature: it ensures that some mass is preserved in the growth-fragmentation cell system (see \cref{thm:frag martingale} below). It also in particular implies that there is no local explosion \textit{via} excessiveness, cf. \eqref{eq: excessiveness}. Let us finally mention that Assumption (B) is the analogue of the existence of a root of the cumulant function $\kappa$ in \cite{Ber-GF,BBCK}. We will actually show in section \ref{sec:universality M(n) general} that $\omega$ can be seen as a root of signed variants of $\kappa$. 

Such an assumption is crucial, in both the nonnegative case and the signed one, to obtain martingales for the growth-fragmentation cell system, as Theorem \ref{thm:BRW martingale} then translates to
\begin{Thm} \label{thm:frag martingale}
For $u\in\Ub$, write $v_u:= v_{\sgn{\Xcal_u(0)}}(\omega)$ for simplicity. For any $z\ne 0$, the process
\[
\Mcal (n) := \sum_{|u|=n+1} v_u |\Xcal_u(0)|^{\omega}, \quad n\ge 0,
\]
is a $(\mathscr{G}_n)_{n\ge 0}$--martingale under $\Pcal_z$.
\end{Thm}

We conclude this paragraph by a very simple but typical calculation leading to a first temporal martingale for $X$.
\begin{Prop} \label{prop:martingale M}
Under $\Pb_z$, the process 
\[
M(s) := v_{\sgn{X(s)}}(\omega) |X(s)|^{\omega} +\sum_{0<r\le s\wedge \zeta} v_{\sgn{-\Delta X(r)}}(\omega) |\Delta X(r)|^{\omega},
\]
is a uniformly integrable martingale for the natural filtration $(F_t^X)_{t\ge 0}$ of $X$, with terminal value $\sum_{r> 0} v_{\sgn{-\Delta X(r)}}(\omega) |\Delta X(r)|^{\omega}$.
\end{Prop}
\begin{proof}
\added{$X$ is clearly adapted to the filtration}. Let us prove that for $s\ge 0$, 
\[
\Eb_i\left[\sum_{r> 0} v_{\sgn{-\Delta X(r)}}(\omega) |\Delta X(r)|^{\omega} \bigg| F_s^X  \right]
=
v_{\sgn{X(s)}}(\omega) |X(s)|^{\omega} +\sum_{0<r\le s\wedge \zeta} v_{\sgn{-\Delta X(r)}}(\omega) |\Delta X(r)|^{\omega}.
\]
Indeed, 
\begin{multline}
\Eb_i\left[\sum_{r> 0} v_{\sgn{-\Delta X(r)}}(\omega) |\Delta X(r)|^{\omega} \bigg| F_s^X  \right]  \\
= 
\Eb_i\left[\sum_{r> s} v_{\sgn{-\Delta X(r)}}(\omega) |\Delta X(r)|^{\omega} \bigg| F_s^X  \right] + \sum_{0\le r\le s\wedge \zeta} v_{\sgn{-\Delta X(r)}}(\omega) |\Delta X(r)|^{\omega}. \label{eq:martingale Z}
\end{multline}
Then by the Markov property at time $s$ and self-similarity of $X$, the first term is
\[
\Eb_i\left[\sum_{r> s} v_{\sgn{-\Delta X(r)}}(\omega) |\Delta X(r)|^{\omega} \bigg| F_s^X  \right]
=
|X(s)|^{\omega} \Eb_{\sgn{X(s)}}\left(\sum_{r>0}v_{\sgn{-\Delta X(r)}}(\omega) |\Delta X(r)|^{\omega}\right).
\]
By definition of $(v_+(\omega), v_-(\omega))$, 
\[
\Eb_{\sgn{X(s)}}\left(\sum_{r>0}v_{\sgn{-\Delta X(r)}}(\omega) |\Delta X(r)|^{\omega}\right)
=
v_{\sgn{X(s)}}(\omega),
\]
and so
\[
\Eb_i\left[\sum_{r> s} v_{\sgn{-\Delta X(r)}}(\omega) |\Delta X(r)|^{\omega} \bigg| F_s^X  \right]
=
v_{\sgn{X(s)}}(\omega)|X(s)|^{\omega}.
\]
Finally, equation \eqref{eq:martingale Z} gives the desired result.
\end{proof}

\begin{Rk}
In particular, this implies that the quantity $f(x)=v_{\sgn{x}}(\omega) |x|^{\omega}$ defines an excessive function for the signed growth-fragmentation $\Xbf$. See \cite{Ber-GF}, Theorem 1, which can be extended in the signed case.
\end{Rk}

\subsection{A change of measure} \label{sec:change measure}
We first define a new probability measure $\Phat_z$ for $z\ne 0$ thanks to the martingale $(\Mcal(n))_{n\ge 0}$ in Theorem \ref{thm:frag martingale}. This new measure is the analogue of the one defined in Section 4.1 in \cite{BBCK}. It describes the law of a new cell system $(\Xcal_u)_{u\in\Ub}$ together with an infinite \emph{ray}, or \emph{leaf}, $\Lcal \in \partial \Ub = \N^{\N}$. \added{On $\mathscr{G}_n$, for $n\ge 0$, it has Radon-Nikodym derivative} with respect to $\Pcal_z$ given by $\Mcal(n)$, normalized so that we get a probability measure, \emph{i.e.} for all $G_n \in \mathscr{G}_n$,
\[
\Phat_z(G_n) := \frac{1}{v_{\sgn{z}}(\omega)|z|^{\omega}} \Ecal_z\left(\Mcal(n) \mathds{1}_{G_n} \right).
\]
Moreover, the law of the particular leaf $\Lcal$ under $\Phat_z$ is determined for all $n\ge 0$ and all $u\in \Ub$ such that $|u|=n+1$ by
\[
\Phat_z \left( \Lcal(n+1)=u \,\big|\, \mathscr{G}_n\right) := \frac{v_{u}|\Xcal_u(0)|^{\omega}}{\Mcal(n)},
\]
where for any $\ell\in\partial\Ub$, $\ell(n)$ denotes the ancestor of $\ell$ at generation $n$. In other words, to define the next generation of the spine, we select one of its jumps proportionally to its size to the power $\omega$ (the spine at generation $0$ being the Eve cell). One can check from the martingale property and the branching structure of the system that these definitions are compatible, and therefore this defines a probability measure by an application of the Kolmogorov extension theorem.

We now introduce the \emph{tagged cell}, that is the size of the cell associated to the leaf $\Lcal$. First, we write $b_{\ell} = \lim\uparrow b_{\ell(n)}$ for any leaf $\ell\in\partial \Ub$. Then, we define $\Xhat$ by $\Xhat(t):=\partial$ if $t\ge b_{\Lcal}$ and 
\[
\Xhat(t):= \Xcal_{\Lcal(n_t)}(t-b_{\Lcal(n_t)}), \quad t<b_{\Lcal},
\]
where $n_t$ is the unique integer $n$ such that $b_{\Lcal(n)}\le t < b_{\Lcal(n+1)}$.  

Observe from the definition of $\Phat_z$ that we have for all nonnegative measurable function $f$ and all $\Gscr_n$--measurable nonnegative random variable $B_n$, 
\[
v_{\sgn{z}}(\omega) |z|^{\omega}\Ehat_z\left(f(\Xcal_{\Lcal(n+1)}(0))B_n\right)
=
\Ecal_z\left( \sum_{|u|=n+1} v_u|\Xcal_u(0)|^{\omega}f(\Xcal_u(0))B_n\right).
\]
This extends to a temporal identity in the following way. Recall that we have enumerated $\Xbf(t)=\left\{\left\{X_i(t), i\ge 1\right\}\right\}, \, t\ge 0$ (this is possible since, according to the remark following Proposition \ref{prop:martingale M}, we know that $\Xbf$ has an excessive function).
\begin{Prop} \label{prop:spine-temporal}
For every $t\ge 0$, every nonnegative measurable function $f$ vanishing at $\partial$, and every $\overline{\Fcal}_t$--measurable nonnegative random variable $B_t$, we have
\[
v_{\sgn{z}}(\omega) |z|^{\omega}\Ehat_z\left(f(\Xhat(t))B_t\right)
=
\Ecal_z\left( \sum_{i\ge 1} v_{\sgn{X_i(t)}}(\omega)|X_i(t)|^{\omega}f(X_i(t))B_t\right).
\]
\end{Prop}
\begin{proof}
We reproduce the proof of \cite{BBCK}, though the ideas are the same, because the martingale used in the change of measure is slightly different.

Let $t\ge 0$. We restrict to proving the result for $B_t$ which is $\overline{\Fcal}_t \cap \Gscr_k$--measurable for some $k\in \N$. First, observe that since $f(\partial)=0$, almost surely,
\[
f(\Xhat(t)) B_t\mathds{1}_{b_{\Lcal(n+1)}>t} \underset{n\rightarrow \infty}{\longrightarrow} f(\Xhat(t)) B_t. 
\]
Therefore, by monotone convergence, 
\[
\Ehat_z(f(\Xhat(t)) B_t) = \underset{n\rightarrow \infty}{\lim} \Ehat_z \left(f(\Xhat(t)) B_t\mathds{1}_{b_{\Lcal(n+1)}>t}\right). 
\]
Now if we condition on $\Gscr_n$ and decompose $\Lcal(n+1)$ over the cells at generation $n+1$, provided $n>k$ so that $B_t$ is $\Gscr_n$--measurable, we get
\[
\Ehat_z \left(f(\Xhat(t)) B_t\mathds{1}_{b_{\Lcal(n+1)}>t}\right)
=
\frac{1}{v_{\sgn{z}}(\omega)|z|^{\omega}} \Ecal_z \left(\sum_{|u|=n+1} v_{u} |\Xcal_u(0)|^{\omega} \mathds{1}_{b_u>t} f(\Xcal_{u(t)}(t-b_{u(t)}))B_t\right).
\]
Here we wrote $u(t)$ for the most recent ancestor of $u$ at time $t$. We now decompose the sum over the ancestor $u(t)$ at time $t$. This gives 
\begin{multline} \label{eq:decompose u(t)}
\Ecal_z \left(\sum_{|u|=n+1} v_{u} |\Xcal_u(0)|^{\omega} \mathds{1}_{b_u>t} f(\Xcal_{u(t)}(t-b_{u(t)}))B_t\right) \\
=
\Ecal_z \left(\sum_{|u'|\le n} \sum_{|u|=n+1} v_{u} |\Xcal_u(0)|^{\omega} \mathds{1}_{b_u>t} f(\Xcal_{u'}(t-b_{u'}))B_t \mathds{1}_{u(t)=u'}\right),    
\end{multline}
and by conditioning on $\overline{\Fcal}_t$ and applying the temporal branching property stated in Theorem \ref{thm:branching}, 
\begin{align*}
&\Ecal_z \left(\sum_{|u|=n+1} v_{u} |\Xcal_u(0)|^{\omega} \mathds{1}_{b_u>t} f(\Xcal_{u(t)}(t-b_{u(t)}))B_t\right) \\
&= \Ecal_z \left(\sum_{|u'|\le n} f(\Xcal_{u'}(t-b_{u'}))B_t \, \Ecal_z\left[\sum_{|u|=n+1} v_{u} |\Xcal_u(0)|^{\omega} \mathds{1}_{b_u>t} \mathds{1}_{u(t)=u'} \, \bigg| \, \overline{\Fcal}_t \right]\right)\\
&= \Ecal_z \left(\sum_{|u'|\le n} f(\Xcal_{u'}(t-b_{u'}))B_t \, \Ecal_{\Xcal_{u'}(t-b_{u'})}\left[\sum_{|u|=n+1-|u'|} v_{u'u} |\Xcal_{u'u}(0)|^{\omega} \right]\mathds{1}_{b_{u'}\le t<b_{u'}+\zeta_{u'}}\right)\\
&= \Ecal_z \left(\sum_{|u'|\le n} f(\Xcal_{u'}(t-b_{u'}))B_t \, \mathds{1}_{b_{u'}\le t<b_{u'}+\zeta_{u'}} v_{\sgn{\Xcal_{u'}(t-b_{u'})}}(\omega) |\Xcal_{u'}(t-b_{u'})|^{\omega}\right).
\end{align*}
Finally, taking $n\rightarrow\infty$ and using monotone convergence, we obtain the desired result.
\end{proof}

\begin{Cor} \label{cor:supermartingale}
The process 
\[
\Mcal_t := \sum_{i=1}^{\infty} v_{\sgn{X_i(t)}}(\omega)|X_i(t)|^{\omega}, \quad t\ge 0,
\]
is a supermartingale with respect to $(\Fcal_t, t\ge 0)$. 
\end{Cor}
\begin{proof}
Proposition \ref{prop:spine-temporal} with $f:=\mathds{1}_{x\neq \partial}$ gives that $\Ecal_z(\Mcal_t) \le v_{\sgn{z}}(\omega) |z|^{\omega}$ and the supermartingale property follows readily from the branching property.
\end{proof}

%
%
\section{Universality of $\Mcal(n)$ and the signed cumulant functions}
\label{sec:universality}
The construction in section \ref{sec:genealogical martingale} produces martingales $\Mcal(n)$ depending on $X$. We now aim at proving that actually, these do not depend on the choice of the Eve process, in the sense that any admissible triplet $(v_+,v_-,\omega)$ for $X$ will also lead to a martingale for any other cell process driving the same growth-fragmentation process. The strategy is as follows. First, we prove universality for all constant sign driving cell processes by defining \emph{signed cumulant functions} which characterise the couple $(v_+/v_-,\omega)$. Then, starting from a possibly signed Eve process $X$, we flip it every time its sign changes and reduce to the previous case. Along the way, we extend the definition of signed cumulant functions to signed processes. Using the constant sign case, this will provide universality for all cell processes driving the same growth-fragmentation.

\subsection{Signed cumulants and universality in the constant sign case} \label{sec:constant sign}
We introduce two key players in the study of self-similar signed growth-fragmentation processes. We focus on the case when $X$ \emph{has no sign change}: in this case, particles born with a positive mass will continue to have a positive mass until they die, and those with negative mass will remain negative. Then the law of $X$ under $\Pb_z$ is determined by its self-similarity index $\alpha$, and the Laplace exponents $\psi_+$ and $\psi_-$ of the Lamperti exponents $\xi_+$ and $\xi_-$ underlying $X$, depending on the sign of $z$. It is convenient to consider 
\[
F(q)
:=
\begin{pmatrix}
\psi_+(q) & 0 \\
0 & \psi_-(q)
\end{pmatrix}, \quad q\ge 0,
\]
as the matrix exponent of $X$ (this is the analog of \eqref{eq: F matrix} in the simple case when there is no sign change). In the constant sign case, because the Lamperti representation holds, it is easy to compute $\lambda(q)$ and to define \emph{signed cumulant functions} which are analogs of the cumulants in \cite{BBCK}. 

Indeed, let us compute $\Eb_+\left[ \sum_{0<r< \zeta} v_{\sgn{-\Delta X(r)}}(q) |\Delta X(r)|^{q} \right]$, for $q\ge 0$. We will assume that 
\added{$m(q)$ is finite all the way, so that in particular $(v_+(q),v_-(q))$ is well-defined.} Let $\Lambda_+$ denote the Lévy measure of $\xi_+$. Since we are summing over all times, we can omit the Lamperti time-change. In addition, note that, because $X$ is positive, the sign of any jump of $X$ is the same as the one of the corresponding jump of $\xi_+$, so that the previous expectation boils down to 
\[
\Eb_+\left[ \sum_{0<r< \zeta} v_{\sgn{-\Delta X(r)}}(q) |\Delta X(r)|^{q} \right]
=
\Eb_+\left[\sum_{r>0} v_{\sgn{-\Delta \xi_+(r)}}(q)\left|\mathrm{e}^{\xi_+(r)}-\mathrm{e}^{\xi_+(r^-)}\right|^q \right].
\]
From there, we can use the compensation formula for Lévy processes, \emph{i.e.}
\begin{align*}
\Eb_+\left[ \sum_{0<r< \zeta} v_{\sgn{-\Delta X(r)}}(q) |\Delta X(r)|^{q} \right]
&=
\Eb_+\left[ \sum_{r>0} v_{\sgn{-\Delta \xi_+(r)}}(q)\mathrm{e}^{q\xi_+(r^-)}\left|\mathrm{e}^{\Delta\xi_+(r)}-1\right|^q\right] \\
&= \int_0^{\infty} \mathrm{d}r \Eb_+[\mathrm{e}^{q\xi_+(r)}] \int_{\R} \Lambda_+(\mathrm{d}x) v_{-\sgn{x}}(q) \left|\mathrm{e}^{x}-1\right|^q.
\end{align*}
\added{Since we assumed that $m(q)$ was finite, we get that $\psi_+(q)<0$ and $\int_{\R} \Lambda_+(\mathrm{d}x) \left|\mathrm{e}^{x}-1\right|^q<\infty$}. We obtain
\[
\Eb_+\left[ \sum_{0<r< \zeta} \frac{v_{\sgn{-\Delta X(r)}}(q)}{v_+(q)} |\Delta X(r)|^{q} \right]
=
-\frac{1}{\psi_+(q)} \int_{\R}\Lambda_+(\mathrm{d}x) \frac{v_{-\sgn{x}}(q)}{v_+(q)} \left|\mathrm{e}^{x}-1\right|^q.
\]
Therefore, if we set 
\[
\Kcal_+(q) = \psi_+(q) + \int_{\R}\Lambda_+(\mathrm{d}x) \frac{v_{-\sgn{x}}(q)}{v_+(q)} \left|\mathrm{e}^{x}-1\right|^q,
\]
we see that
\begin{equation} \label{eq:sum positive cumulant}
\Eb_+\left[ \sum_{0<r< \zeta} \frac{v_{\sgn{-\Delta X(r)}}(q)}{v_+(q)} |\Delta X(r)|^{q} \right]
=
\begin{cases}
1-\frac{\Kcal_+(q)}{\psi_+(q)} & \text{if} \; \psi_+(q)<0, \\
+\infty & \text{otherwise}.
\end{cases}
\end{equation}
Equation \eqref{eq:sum positive cumulant} is reminiscent of Lemma 3 in \cite{Ber-GF}. 

Likewise, under symmetrical assumptions on $q$ (or applying the previous calculations to $-X$),
\begin{equation} \label{eq:sum negative cumulant}
\Eb_-\left[ \sum_{0<r< \zeta} \frac{v_{\sgn{-\Delta X(r)}}(q)}{v_-(q)} |\Delta X(r)|^{q} \right]
=
\begin{cases}
1-\frac{\Kcal_-(q)}{\psi_-(q)} & \text{if} \; \psi_-(q)<0, \\
+\infty & \text{otherwise},
\end{cases}
\end{equation}
where, with obvious notations, 
\[
\Kcal_-(q) = \psi_-(q) + \int_{\R}\Lambda_-(\mathrm{d}x) \frac{v_{\sgn{x}}(q)}{v_-(q)} \left|\mathrm{e}^{x}-1\right|^q.
\]
Then Assumption (B) translates to $\Kcal_+(\omega) = \Kcal_-(\omega) =0$ (owing to the fact that, by Perron-Frobenius theory, the leading eigenvalue is the only one associated with a positive eigenvector). Therefore the roots of $(\Kcal_+,\Kcal_-)$ give rise to martingales, as explained in Theorem \ref{thm:frag martingale}. We call $\Kcal_+$ and $\Kcal_-$ the \emph{signed cumulant functions}. We will also use the term \emph{cumulant functions} to refer to the one defined in \cite{BBCK}. More precisely, let $\Xbf^+$ (\emph{resp.} $\Xbf^-$) be the growth-fragmentation process obtained from $\Xbf$ by killing all the cells with negative mass (\emph{resp.} positive mass) together with their progeny, under $\Pbf_+$ (\emph{resp.} $\Pbf_-$). We define $\kappa_+$ and $\kappa_-$ as the cumulant functions, in the sense of \cite{BBCK}, of $\Xbf
^+$ and $\Xbf^-$ respectively. Recall that they are given by
\[
\kappa_+(q) := \psi_+(q) + \int_{-\infty}^0 \Lambda_+(\mathrm{d}x) |\mathrm{e}^x
-1|^q, \quad q\ge 0,
\]
and 
\[
\kappa_-(q) := \psi_-(q) + \int_{-\infty}^0 \Lambda_-(\mathrm{d}x) |\mathrm{e}^x
-1|^q, \quad q\ge 0,
\]
so that for instance
\[
\Kcal_+(q) = \kappa_+(q) + \frac{v_-(q)}{v_+(q)}\int_0^{\infty} \Lambda_+(\mathrm{d}x) |\mathrm{e}^x-1|^q, \quad q\ge 0.
\]
Moreover, it is well-known that $\kappa_+$ and $\kappa_-$ are invariants of $\Xbf^+$ and $\Xbf^-$, and characterise them respectively (see \cite{QShi} for more details). We now prove the universality of $(\Mcal(n))_{n\ge 0}$ in the constant sign case. Let $X$ and $X'$ be two driving Markov processes with constant sign defining the growth-fragmentation processes $\Xbf$ and $\Xbf'$ respectively. We write $m(q)$ and $m'(q)$ for the corresponding matrices introduced in \eqref{eq:matrix}. Recall that $(v_+,v_-,\omega)$ is said \emph{admissible} for $X$ if $m(\omega)$ has Perron-Frobenius eigenvalue $1$ and $(v_+,v_-)$ is an associated positive eigenvector, so that the triplet $(v_+,v_-,\omega)$ defines a martingale as explained in section \ref{sec:genealogical martingale}. 
\begin{Prop} \label{prop:admissible constant sign}
Suppose that $\Xbf \overset{\Lcal}{=} \Xbf'$. If $(v_+,v_-,\omega)$ is admissible for $X$, then it is also admissible for $X'$.
\end{Prop}
\begin{proof}
By definition of $(\Kcal_+,\Kcal_-)$, the triplet $(v_+,v_-,\omega)$ is admissible for $X$ if, and only if, $\Kcal_+(\omega) = \Kcal_-(\omega) =0$. This, in turn, is equivalent to the system
\[
\begin{cases}
    \frac{v_-}{v_+} = -\frac{\kappa_+(\omega)}{\displaystyle{\int_0^{\infty} \Lambda_+(\mathrm{d}x) |\mathrm{e}^x-1|^{\omega}}},  \\
     0 = \kappa_-(\omega)\kappa_+(\omega) - \displaystyle{\int_0^{\infty} \Lambda_+(\mathrm{d}x) |\mathrm{e}^x-1|^{\omega}} \cdot \displaystyle{\int_{0}^{\infty} \Lambda_-(\mathrm{d}x) |\mathrm{e}^x-1|^{\omega}}. 
\end{cases}
\]
Note that all the fractions are well defined because $(v_+,v_-)$ is a positive eigenvector. Since $\kappa_+$ and $\kappa_-$ only depend on the growth-fragmentation $\Xbf$ induced by $X$, the result follows if we show that $\Lambda_{+ \, \big|(0,+\infty)}$ and $\Lambda_{- \, \big|(0,+\infty)}$ also depend only on $\Xbf$. In fact, since $\Xbf$ and $\Xbf'$ share the same positive cumulant, for instance, we have for $q\ge 0$, with obvious notations
\[
\kappa_+(q) 
= \psi_+(q) + \int_{-\infty}^0 \Lambda_+(\mathrm{d}x) |\mathrm{e}^x
-1|^q
= \psi'_+(q) + \int_{-\infty}^0 \Lambda'_+(\mathrm{d}x) |\mathrm{e}^x
-1|^q.
\]
\added{Up to translation, $\kappa_+$ is a Laplace exponent, and therefore} uniqueness in the Lévy-Khintchine formula triggers that $\Lambda_{+ \, \big|(0,+\infty)}=\Lambda'_{+ \, \big|(0,+\infty)}$. Similarly, using invariance of $\kappa_-$, one has $\Lambda_{- \, \big|(0,+\infty)}=\Lambda'_{- \, \big|(0,+\infty)}$, and this concludes the proof of Proposition \ref{prop:admissible constant sign}.
\end{proof}
\begin{Cor}
Suppose that $\Xbf \overset{\Lcal}{=} \Xbf'$. If $(v_+,v_-,\omega)$ is admissible for $X$, then under $\Pb_z$, the process 
\[
M'(s) := v_{\sgn{X'(s)}}(\omega) |X'(s)|^{\omega} +\sum_{0<r\le s\wedge \zeta} v_{\sgn{-\Delta X'(r)}}(\omega) |\Delta X'(r)|^{\omega},
\]
is a uniformly integrable martingale for the natural filtration $(F_t^{X'})_{t\ge 0}$ of $X'$, with terminal value $\sum_{r> 0} \frac{v_{\sgn{-\Delta X'(r)}}(\omega)}{v_i(\omega)} |\Delta X'(r)|^{\omega}$.
\end{Cor}

\subsection{Universality of $\Mcal(n)$ in the general case} \label{sec:universality M(n) general}
We now move from the constant sign case to the general case. To this end, we construct from any Eve cell process $X$ a constant sign process $X^{\uparrow}$ driving the same growth-fragmentation. We then prove that the triplets $(v_+,v_-,\omega)$ are simultaneously admissible for $X$ and $X^{\uparrow}$.
\\ \\ \noindent \textbf{Constructing a constant sign driving process from a signed Eve cell.} \label{par:constant sign}
To study signed growth-fragmentation, it is reasonable to reduce to the constant sign case in \cite{BBCK}. One natural choice for this, starting from a signed Eve process $X$ with positive mass at time $0$, is to follow it until it jumps to the negatives, and then select the particle this jump creates (which has positive mass, since the jump is negative). If we continue by induction, this constructs some process $X
^{\uparrow}$ which, under $\Pb_z$ for $z>0$, remains nonnegative at all times. Note that the jump times that we select (from positive to negative) could have an accumulation point. If this happens, then by \cite[Proposition 3]{CPR}, this accumulation point is the first time that $X^{\uparrow}$ hits $0$, in which case we decide that $X^{\uparrow}$ is absorbed at $0$. Likewise, we can construct $X^{\uparrow}$ starting from a negative mass. A first observation is that the branching structure, Markov property and self-similarity of $X$ ensure that $X^{\uparrow}$ is a self-similar Markov process under $\Pb_z$ for all $z\ne 0$. Moreover, it is plain that $X^{\uparrow}$ carries the same growth-fragmentation as $X$ itself. In short, $X$ and $X
^{\uparrow}$ are two driving cell processes for the same growth-fragmentation. If we manage to make explicit the law of $X^{\uparrow}$ (or rather, its Lamperti exponent) in terms of $X$ (or its Lamperti-Kiu characteristics), then we are in good shape to reduce the study to constant sign driving cell processes. Let us focus on the case when $X$ starts from a positive mass $z>0$, say, and recall the notation $\xi_+$, $q_+$ and $U_+$ of section \ref{sec:MAP}. Then the independence and stationarity of increments of the Lévy process $\xi_+$ imply that,  up to Lamperti-Kiu time change, going from $X$ to $X^{\uparrow}$ amounts to adding jumps to $\xi_+$ at times $\zeta_+$, which are exponential clocks with parameter $q_+$. Call $H_1$ the first time when $X$ crosses $0$. Then the intensity $\delta$ of these jumps is exactly what it takes, at the exponential level, to go from $X(H_1^-)$ to $X^{\uparrow}(H_1)$, \emph{i.e.} $\delta = \log \left(\frac{-\Delta X(H_1)}{X(H_1^-)} \right) = \log(1+\mathrm{e}^{U_+})$. Therefore, the Lamperti exponent $\xi^{\uparrow}_+$ of $X^{\uparrow}$ started from $z>0$ results in the superposition of $\xi_+$ and an independent compound Poisson process with rate $q_+$ and jumps distributed as the image $\widetilde{\Lambda}_{U_+}(\mathrm{d}x)$ of the law $\Lambda_{U_+}(\mathrm{d}x)$ of $U_+$, by the mapping $x\mapsto\log(1+\mathrm{e}^x)$. Hence its Laplace exponent is 
\begin{equation} \label{eq:psi dagger +}
\psi^{\uparrow}_+(q) = \psi_+(q) + q_+\left(\int_{\R} (1+\mathrm{e}^x)^q \Lambda_{U_+}(\mathrm{d}x) - 1\right), \quad q\ge 0,
\end{equation}
and, in particular, its Lévy measure is 
\begin{equation} \label{eq:Lévy dagger +}
\Lambda^{\uparrow}_+(\mathrm{d}x) := \Lambda_+(\mathrm{d}x) + q_+\widetilde{\Lambda}_{U_+}(\mathrm{d}x).
\end{equation}
Note that these expressions only depend on the positive characteristics of $X$, and this is coherent with the construction of $X^{\uparrow}$. The same calculations can be carried out in the case when $z<0$, and finally, we obtain
\begin{equation} \label{eq:psi dagger -}
\psi^{\uparrow}_-(q) = \psi_-(q) + q_- \left(\int_{\R} (1+\mathrm{e}^x)^q \Lambda_{U_-}(\mathrm{d}x) - 1\right), \quad q\ge 0.
\end{equation}
See Figure \ref{fig:X dagger} for a drawing of $X^{\uparrow}$.
\begin{figure}[h] 
\begin{center}
\includegraphics[scale=0.75]{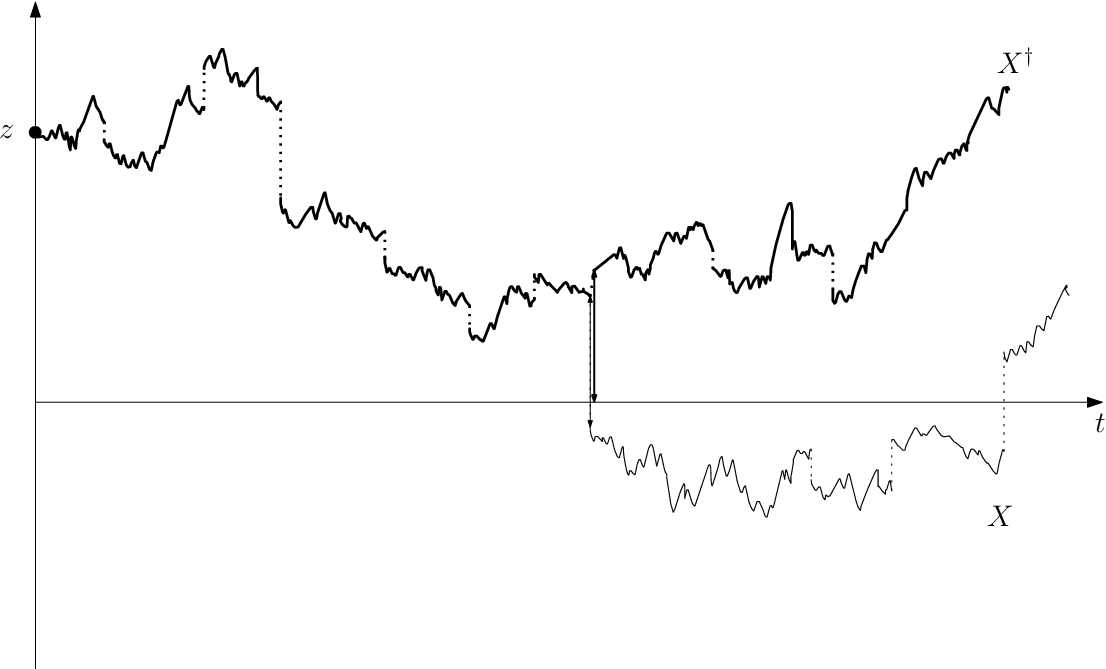}
\end{center}
\caption{Constructing a positive Eve cell process from $X$. The process $X^{\uparrow}$ (in bold) is constructed from $X$ by selecting the complementary positive cell created when $X$ jumps below $0$, and then by induction.}
\label{fig:X dagger}
\end{figure}

\noindent \textbf{Universality of $\Mcal(n)$ and the signed cumulant functions.}
We want to extend the result of Proposition \ref{prop:admissible constant sign} to general signed driving processes. To do this, we resort to $X^{\uparrow}$ and link admissible triplets $(v_+,v_-,\omega)$ for $X$ and $X^{\uparrow}$. First, we state a technical lemma, that is probably superfluous, but simplifies calculations.
\begin{Lem} \label{lem:admissibility}
The following points are equivalent:
\begin{itemize}
    \item[(i)] $(v_+,v_-,\omega)$ is admissible for $X$.
    \item[(ii)] The process $M$ defined in Proposition \ref{prop:martingale M} with parameters $(v_+,v_-,\omega)$ is a uniformly integrable martingale.
    \item[(iii)] Let $H_1$ be the first time $X$ crosses $0$. Then,
    \[
    \Eb_{\pm}\left[ M(H_1)\right] = v_{\pm}.
    \]
\end{itemize}
\end{Lem}
\begin{proof}
The implication $(i)\Rightarrow(ii)$ has already been proved in Proposition \ref{prop:martingale M}, and $(ii)\Rightarrow (iii)$ follows from an application of the optional stopping theorem and the uniform integrability of $M$. Therefore only $(iii) \Rightarrow (i)$ remains to be proved. Assume that we know $(iii)$. Denote by $H_0=0<H_1<H_2<\ldots$ the sequence of times when $X$ crosses $0$. Then 
\[
    \Eb_+\left[ \sum_{s>0} v_{-\sgn{\Delta X(s)}} |\Delta X(s)|^{\omega}\right] 
    =
    \sum_{k\ge 0} \Eb_+\left[ \sum_{H_k<s\le H_{k+1}} v_{-\sgn{\Delta X(s)}} |\Delta X(s)|^{\omega} \right].
\]
By the Markov property and the self-similarity of $X$, this entails
\begin{multline*}
\Eb_+\left[ \sum_{s>0} v_{-\sgn{\Delta X(s)}} |\Delta X(s)|^{\omega}\right] 
=
\sum_{k\ge 0} \Eb_+\left[ |X(H_{2k})|^{\omega}\right] \Eb_+\left[ \sum_{0<s\le H_1} v_{-\sgn{\Delta X(s)}} |\Delta X(s)|^{\omega} \right] \\
+
\sum_{k\ge 0} \Eb_+\left[ |X(H_{2k+1})|^{\omega}\right] \Eb_-\left[ \sum_{0<s\le H_1} v_{-\sgn{\Delta X(s)}} |\Delta X(s)|^{\omega} \right].
\end{multline*}
Making use of $(iii)$, this means
\begin{multline*} 
\Eb_+\left[ \sum_{s>0} v_{-\sgn{\Delta X(s)}} |\Delta X(s)|^{\omega}\right] \\
= \sum_{k\ge 0} \Eb_+\left[ |X(H_{2k})|^{\omega}\right] \left(v_+-v_-\Eb_+[|X(H_1)|^{\omega}]\right) + \sum_{k\ge 0} \Eb_+\left[ |X(H_{2k+1})|^{\omega}\right] \left(v_--v_+\Eb_-[|X(H_1)|^{\omega}]\right). 
\end{multline*}
Using the Markov property backwards, we have $\Eb_+\left[ |X(H_{2k})|^{\omega}\right] \Eb_+[|X(H_1)|^{\omega}] = \Eb_+\left[ |X(H_{2k+1})|^{\omega}\right]$, and likewise $\Eb_+\left[ |X(H_{2k+1})|^{\omega}\right] \added{\Eb_-[|X(H_1)|^{\omega}]} = \Eb_+\left[ |X(H_{2k+2})|^{\omega}\right]$ for all $k\ge 0$. We can therefore simplify the previous expression by telescoping series (one can first use a truncated version of the series in order to split the sums). This gives 
\[
\Eb_+\left[ \sum_{s>0} v_{-\sgn{\Delta X(s)}} |\Delta X(s)|^{\omega}\right]
=
v_+\sum_{k\ge 0} \Eb_+\left[ |X(H_{2k})|^{\omega}\right] - v_+\sum_{k\ge 0} \added{\Eb_+\left[ |X(H_{2k+2})|^{\omega}\right]}.
\]
Therefore,
\[
\Eb_+\left[ \sum_{s>0} v_{-\sgn{\Delta X(s)}} |\Delta X(s)|^{\omega}\right] = v_+.
\]
Similarly, 
\[
\Eb_-\left[ \sum_{s>0} v_{-\sgn{\Delta X(s)}} |\Delta X(s)|^{\omega}\right] = v_-,
\]
and so $(v_+,v_-,\omega)$ is admissible for $X$.
\end{proof}

We may now bridge the gap between $X$ and $X^{\uparrow}$.
\begin{Prop} \label{prop:admissibility dagger}
A triplet $(v_+,v_-,\omega)$ is admissible for $X$ if, and only if, it is admissible for $X^{\uparrow}$.
\end{Prop}
\begin{proof}
We use $(iii)$ in Lemma \ref{lem:admissibility} above. Define $M^{\uparrow}$ as the process in Proposition \ref{prop:martingale M} associated with $X^{\uparrow}$ and with parameters $(v_+,v_-,\omega)$. The key remark is that $M(H_1) = M^{\uparrow}(H_1)$ a.s. Indeed, under $\Pb_+$ say, $-\Delta X(H_1) = X^{\uparrow}(H_1)>0$, and $-\Delta X^{\uparrow}(H_1) = X(H_1)<0$, a.s. Therefore, $(v_+,v_-,\omega)$ is admissible for $X$ if and only if,
\begin{equation} \label{eq:M(H_1) dagger}
\Eb_{\pm}[M^{\uparrow}(H_1)] = v_{\pm}.
\end{equation}
It remains to prove that this is equivalent to $(v_+,v_-,\omega)$ being admissible for $X^{\uparrow}$. First, the optional stopping theorem gives that if $(v_+,v_-,\omega)$ is admissible for $X^{\uparrow}$, then \eqref{eq:M(H_1) dagger} holds. Conversely, we can more or less run the same arguments as in the proof of $(iii)\Rightarrow (i)$ in Lemma \ref{lem:admissibility}. For example, if we denote by $T_0=0<T_1<T_2<\ldots$ the times corresponding to those special jumps of $X^{\uparrow}$ that correspond to sign changes for $X$, then using the Markov property and self-similarity of $X
^{\uparrow}$
\[
\Eb_+\left[ \sum_{t>0} v_{-\sgn{\Delta X^{\uparrow}(t)}} |\Delta X^{\uparrow}(t)|^{\omega}\right] 
=
\sum_{k\ge 0} \Eb_+\left[ |X^{\uparrow}(T_{k})|^{\omega}\right] \Eb_+\left[ \sum_{0<t\le T_1} v_{-\sgn{\Delta X^{\uparrow}(t)}} |\Delta X^{\uparrow}(t)|^{\omega} \right] 
\]
By \eqref{eq:M(H_1) dagger}, this is
\[
\Eb_+\left[ \sum_{t>0} v_{-\sgn{\Delta X^{\uparrow}(t)}} |\Delta X^{\uparrow}(t)|^{\omega}\right]
= 
v_+\sum_{k\ge 0} \Eb_+\left[ |X^{\uparrow}(T_{k})|^{\omega}\right] \left(1-\Eb_+[|X^{\uparrow}(T_1)|^{\omega}]\right).
\]
Yet by applying the Markov property backwards,
\[
\sum_{k\ge 0} \Eb_+\left[ |X^{\uparrow}(T_{k})|^{\omega}\right] \left(1-\Eb_+[|X^{\uparrow}(T_1)|^{\omega}]\right) = \sum_{k\ge 0} \Eb_+\left[ |X^{\uparrow}(T_{k})|^{\omega}\right] - \sum_{k\ge 0} \Eb_+\left[ |X^{\uparrow}(T_{k+1})|^{\omega}\right] = 1.
\]
Therefore
\[
\Eb_+\left[ \sum_{t>0} v_{-\sgn{\Delta X^{\uparrow}(t)}} |\Delta X^{\uparrow}(t)|^{\omega}\right]
= 
v_+,
\]
and similarly
\[
\Eb_-\left[ \sum_{t>0} v_{-\sgn{\Delta X^{\uparrow}(t)}} |\Delta X^{\uparrow}(t)|^{\omega}\right]
= 
v_-.
\]
Thus $(v_+,v_-,\omega)$ is admissible for $X^{\uparrow}$.
\end{proof}

Proposition \ref{prop:admissibility dagger}, in turn, enables us to define general \emph{signed cumulant functions}. Recall from section \ref{sec:MAP} the notation $G_{+,-}(q):=\Eb[\mathrm{e}^{qU_+}]$ and $G_{-,+}(q):=\Eb[\mathrm{e}^{qU_-}]$ for the Laplace transforms of the special jumps.
\begin{Cor} \label{cor:cumulant}
Let 
\begin{align*}
\Kcal_+(q) 
&= \psi^{\uparrow}_+(q) + \int_{\R}\Lambda^{\uparrow}_+(\mathrm{d}x) \frac{v_{-\sgn{x}}(q)}{v_+(q)} \left|\mathrm{e}^{x}-1\right|^q \\
&= \kappa_+(q) + \frac{v_-(q)}{v_+(q)}\left(\int_0^{\infty} \Lambda_+(\mathrm{d}x) |\mathrm{e}^x-1|^q +q_+ G_{+,-}(q)\right),
\end{align*}
and 
\begin{align*}
\Kcal_-(q) 
&= \psi^{\uparrow}_-(q) + \int_{\R}\Lambda^{\uparrow}_-(\mathrm{d}x) \frac{v_{\sgn{x}}(q)}{v_-(q)} \left|\mathrm{e}^{x}-1\right|^q \\
&=\kappa_-(q) + \frac{v_+(q)}{v_-(q)}\left(\int_0^{\infty} \Lambda_-(\mathrm{d}x) |\mathrm{e}^x-1|^q +q_- G_{-,+}(q)\right),
\end{align*}
be the signed cumulant functions associated with $X^{\uparrow}$ -- which we rephrased in terms of $X$ thanks to \eqref{eq:psi dagger +} and \eqref{eq:psi dagger -}. Then the suitable martingale exponents $\omega$ for $X$ are the roots of $(\Kcal_+,\Kcal_-)$.
\end{Cor}
The final end to the universality of $\Mcal(n)$ is provided by the next theorem.
\begin{Thm} \label{thm:universality M(n)}(Universality of $\Mcal(n)$)

\noindent Let $X$ and $X'$ be two possibly signed cell processes, driving the same growth-fragmentation $\Xbf=\Xbf'$. Then $(v_+,v_-,\omega)$ is admissible for $X$ if, and only if, it is admissible for $X'$.
\end{Thm}
\begin{proof}
This is a corollary of Propositions \ref{prop:admissible constant sign} and \ref{prop:admissibility dagger}. We have the following equivalences: $(v_+,v_-,\omega)$ is admissible for $X$ if and only if it is admissible for $X^{\uparrow}$ (Proposition \ref{prop:admissibility dagger}), \textit{i.e.} if and only if it is admissible for $(X')^{\uparrow}$ (Proposition \ref{prop:admissible constant sign}), \textit{i.e.} if and only if it is admissible for $X'$ (Proposition \ref{prop:admissibility dagger}).
\end{proof}

\section{The spinal decomposition}
\label{sec:spinal decomposition}
This section is devoted to the study of self-similar signed growth-fragmentations under the change of measure given in section \ref{sec:change measure}. In particular, we aim at describing the law of the tagged cell under $\Phat_z$. Roughly speaking, we shall see that by changing the measure according to section \ref{sec:change measure}, the tagged cell $\Xhat$ evolves as an explicit self-similar Markov process $Y$, and conditionally on its evolution, the growth-fragmentations induced by the jumps of $\Xhat$ are independent with law $\Pbf_x$ where $-x$ is the jump size.

\subsection{Description and results}

\textbf{Description of the Markov process $Y$.} We first introduce a Markov process that will describe the law of the \emph{spine} in the next paragraph. Remember the couple of Lévy measures $(\Lambda_+^{\uparrow},\Lambda_-^{\uparrow})$ for the constant sign process constructed in paragraph \ref{par:constant sign}. Let us set some notation and write 
\begin{align*}
\sigma_{+}(q) 
&:= \frac{v_-(\omega)}{v_+(\omega)} \int_0^{\infty} |\mathrm{e}^x-1|^q \Lambda^{\uparrow}_{+}(\mathrm{d}x) \\
&= \frac{v_-(\omega)}{v_+(\omega)}\left(\int_0^{\infty} \Lambda_+(\mathrm{d}x) |\mathrm{e}^x-1|^q +q_+ G_{+,-}(q)\right), \quad q\ge 0,
\end{align*}
and symmetrically,
\begin{align*}
\sigma_{-}(q) 
&:= \frac{v_+(\omega)}{v_-(\omega)} \int_0^{\infty} |\mathrm{e}^x-1|^q \Lambda^{\uparrow}_{-}(\mathrm{d}x) \\
&=\frac{v_+(\omega)}{v_-(\omega)}\left(\int_0^{\infty} \Lambda_-(\mathrm{d}x) |\mathrm{e}^x-1|^q +q_- G_{-,+}(q)\right), \quad q\ge 0.
\end{align*}
Recall the notation $(\Kcal_+,\Kcal_-)$ for the signed cumulant functions and $(\kappa_+,\kappa_-)$ for the cumulant functions (see section \ref{sec:constant sign} and Corollary \ref{cor:cumulant}). Define the following matrix
\[
\Fhat(q) := 
\begin{pmatrix}
\kappa_+(\omega+q) & \sigma_+(\omega+q) \\
\sigma_-(\omega+q) & \kappa_-(\omega+q) 
\end{pmatrix}, \quad q\ge 0.
\]
\begin{Lem} \label{lem:Fhat}
Let $(\xihat_{+},\xihat_{-})$ be a pair of independent Lévy processes with Laplace exponents
\[
\psihat_{+}(q):=\kappa_+(\omega+q)-\kappa_+(\omega), \quad q\ge 0,
\]
and 
\[
\psihat_{-}(q):=\kappa_-(\omega+q)-\kappa_-(\omega), \quad q\ge 0.
\]
Furthermore, let $\qhat_{\pm}:= \sigma_{\pm}(\omega)$, and $(\Uhat_{+,-},\Uhat_{-,+})$ be a pair of random variables with respective Laplace transforms $G_{+,-}(q) := \frac{\sigma_+(\omega+q)}{\sigma_+(\omega)}$ and $G_{-,+}(q) := \frac{\sigma_-(\omega+q)}{\sigma_-(\omega)}$ for $q\ge 0$. Then the Markov additive process $(\xihat,\widehat{J})$ defined piecewise as in \eqref{eq:piecewise} with these characteristics has matrix exponent $\Fhat$. 
\end{Lem}
\begin{Rk} \label{rk:Levy-Ito}
Note that for instance 
\[
\kappa_+(\omega+q)-\kappa_+(\omega) = \psi^{\uparrow}_+(\omega+q)-\psi^{\uparrow}_+(\omega) + \int_{(-\infty,0)} \left((1-\mathrm{e}^x)^{\omega+q} - (1-\mathrm{e}^x)^{\omega} \right)\Lambda^{\uparrow}_+(\mathrm{d}x).
\]
Therefore $\xihat_{+}$ can be obtained by the Lévy-Itô decomposition as a superposition of a Lévy process $\eta_+$ with Laplace exponent $q\mapsto \psi^{\uparrow}_+(\omega+q)-\psi^{\uparrow}_+(\omega)$, and a compound Poisson process $\nu_+$ with Lévy measure $\mathrm{e}^{\omega x}\widetilde{\Lambda}_{+ \big|(-\infty,0)}(\mathrm{d}x)$, where $\widetilde{\Lambda}_+$ is the pushforward measure of $\Lambda^{\uparrow}_+$ by $x\mapsto \log |1-\mathrm{e}^x|$. In this decomposition, $\nu_+$ will in fact stand for special jumps of the spine corresponding to changes in the \emph{generation} of the spine (when we select a negative jump), whereas $\eta_+$ stems from biasing $\xi_+$ according to its exponential martingale.
\end{Rk}
\begin{Not} \label{not:Y and N}
We shall denote by $Y$ the real-valued self-similar Markov process with Lamperti-Kiu characteristics $(\alpha,\Fhat)$. 
\end{Not}
\begin{proof}
The only point is to prove that $\Fhat$ is indeed the matrix exponent of this MAP. This follows from straightforward calculations, using $\Kcal_+(\omega)=\Kcal_-(\omega)=0$. For example, the first entry of the matrix exponent should be 
\[
\psihat_+(q) - \qhat_+ = \kappa_+(\omega+q)-\kappa_+(\omega) - \frac{v_-(\omega)}{v_+(\omega)}\int_0^{\infty} |\mathrm{e}^x-1|^{\omega} \Lambda^{\uparrow}_{+}(\mathrm{d}x) = \kappa_{+}(\omega+q)  - \Kcal_+(\omega) = \kappa_+(\omega+q).
\]
\end{proof}

\medskip
\noindent \textbf{Rebuilding the growth-fragmentation from the spine.} To give a precise statement on the law of the growth-fragmentation under $\Phat_z$, we need to rebuild the growth-fragmentation from the spine. As in section \ref{sec:growth-frag}, the first step is to label the jumps of $\Xhat$. In general, we do not know if we can rank those in lexicographical order, and thus we use the following procedure. Jumps of the tagged cell $\Xhat$ will be labelled by pairs $(n,j)$, $n\ge 0$ denoting the generation of the tagged cell immediately before the jump, and $j\ge 1$ being the rank (in the usual lexicographical sense) of the jump among those of the tagged cell at generation $n$ (we also count the final jump, when the generation changes to $n+1$). For each such $(n,j)$, we can define the growth-fragmentation $\Xbfhat_{n,j}$ stemming from the corresponding jump. More precisely, if the generation stays the same during the $(n,j)$--jump, then we set
\[
\Xbfhat_{n,j}(t) := \left\{\left\{ \Xcal_{uw}(t-b_{uw}+b_u), \; w\in\Ub \; \text{and} \; b_{uw}\le t+b_u<b_{uw}+\zeta_{uw} \right\}\right\},
\]
where $u$ is the label of the cell born at the $(n,j)$--jump. On the contrary, if the $(n,j)$--jump corresponds to a jump for the generation of the tagged cell, then the tagged cell jumps from label $u$ to label $uk$ say, and we set
\[
\Xbfhat_{n,j}(t) := \left\{\left\{ \Xcal_{uw}(t-b_{uw}+b_{uk}), \; w\in\Ub\setminus \{k\} \; \text{and} \; b_{uw}\le t+b_{uk}<b_{uw}+\zeta_{uw} \right\}\right\}.
\]
Finally, we agree that $\Xbfhat_{n,j} := \partial$ when the $(n,j)$--jump does not exist, and this sets $\Xbfhat_{n,j}$ for all $n\ge 0$ and all $j\ge 1$.
\\ \\ \noindent \textbf{Description of the growth-fragmentation under $\Phat_z$.} We are now set to describe the law of $\Xbf$ under $\Phat_z$. Recall the definition of $Y$ from Notation \ref{not:Y and N}, and that $n_t$ denotes the generation of the spine at time $t$.
\begin{Thm} \label{thm:spine}
Under $\Phat_1$, $(\Xhat(t), 0\le t<b_{\Lcal})$ is distributed as $(Y(t), 0\le t<I)$. 
Moreover, conditionally on $(\Xhat(t), n_t)_{0\le t<b_{\Lcal}}$, the processes $\Xbfhat_{n,j}$, $n\ge 0$, $j\ge 1$, are independent and each $\Xbfhat_{n,j}$ has law $\Pbf_{x(n,j)}$ where $-x(n,j)$ is the size of the $(n,j)$--th jump.
\end{Thm}
Before we come to the proof, let us make some comments on this result.
\begin{Rk} \label{Rk: spine thm}
\begin{enumerate}
    \item We can give the joint law of $(\Xhat(t), n_t)_{0\le t<b_{\Lcal}}$. Note that, unlike $\Xhat$, the law of $n_t$ depends on the choice of the Eve cell. For example, in the case when the Eve cell is $X^{\uparrow}$, the joint law of $(\Xhat(t), n_t)_{0\le t<b_{\Lcal}}$ is the same as $(Y(t), N(\tau_t))_{0\le t<I}$, where $(N(t), t\ge 0)$ is the Poisson process arising from the superposition of $\nu_+$ and the compound Poisson process corresponding to the sign changes of $Y$ (modulo Lamperti time-change $\tau_t$).
    \item We can rephrase the theorem perhaps more tellingly by clarifying the characteristics $\xihat_{\pm}, \qhat_{\pm}, \Uhat_{\pm,\mp}$ describing the MAP. Let us do this for the positive part (the negative one being \added{analogous}). As explained in Remark \ref{rk:Levy-Ito}, the Lévy process $\xihat_+$ is the result of a superposition of a biased version of $\xi_+$, and a compound Poisson process. This compound Poisson process takes care of special jumps of the spine: namely, it takes care of the eventuality that the spine selects a \emph{negative} jump of the driving process, so that the spine remains positive at the next generation. The variable $\qhat_+$ is an exponential random variable which has parameter $\sigma_+(\omega)$, that is to say it corresponds to the first time the spine becomes negative. This happens either because the driving process it follows does, or because the spine jumps to a negative cell, and this is conspicuous in the two terms of $\sigma_+$. Finally, the variable $U_+$ is the intensity of the jumps of the spine when it crosses $0$ (again, both cases can happen).
    \item The signed growth-fragmentation $\Xbf$ is characterised by $(\kappa_+,\kappa_-)$. Theorem \ref{thm:spine} shows that the law of the spine also characterises $\Xbf$.
    \item One can retrieve from the first entry of $\Fhat$ the description of the spine for \emph{unsigned} growth-fragmentation presented in \cite{BBCK}, Theorem 4.2. Note, however, that the exponent $\omega$ differs, and this is because of the $h$-transform used to condition the spine to remain positive. We refer to \cite{DDK}, Appendix 8, for details on these harmonic functions for self-similar real-valued Markov processes. We will give details of this for a particular family of signed growth-fragmentation processes in the next section.
    \item The process $(\Mcal_t, t\ge 0)$ in Corollary \ref{cor:supermartingale} is a supermartingale, but when is it a martingale? Proposition \ref{prop:spine-temporal} gives that
    \[
    \forall t\ge 0, \quad \Ecal_z[\Mcal_t] = v_{\sgn{z}}(\omega) |z|^{\omega} \Phat_z(\Xhat(t)\in \R^*).
    \]
    Therefore $(\Mcal_t, t\ge 0)$ is a martingale if, and only if, for all $t\ge 0$, $\Phat_z(\Xhat(t)\in \R^*)=1$. This, in turn, is equivalent to $Y$ having infinite lifetime. In particular, if $\alpha \kappa_+'(\omega)>0$ and $\alpha \kappa_-'(\omega)>0$, then $\alpha\added{\xihat_+}$ and $\alpha\added{\xihat_-}$ both drift to $+\infty$ (\added{$\xihat_+$} and \added{$\xihat_-$} both drift \added{to} $+\infty$ or $-\infty$ depending on the sign of $\added{\psihat_{\pm}}'(0)=\kappa_{\pm}'(\omega)$), and by Lamperti time-change $Y$ has infinite lifetime and $(\Mcal_t, t\ge 0)$ is a martingale. On the other hand, if $\alpha \kappa_+'(\omega)<0$ or $\alpha \kappa_-'(\omega)<0$, then for symmetric reasons $(\Mcal_t, t\ge 0)$ is not a martingale.
\end{enumerate}
\end{Rk}

\subsection{Proof of Theorem \ref{thm:spine}}
\textbf{Proof in the constant sign case.}
We look at the specific example when the Eve cell $X$ has no sign change. In this case, the Lamperti representation holds, and so the compensation formula for Lévy processes makes it simpler to determine the law of the spine $\Xhat$. This paragraph is therefore an extension of \cite{BBCK}, when we also take into account the \emph{positive} jumps. 

Let us prove the first claim. First of all, we can restrict to the homogeneous case $\alpha=0$: for a general index $\alpha$, the result then stems from Lamperti time-change. Furthermore, the definition of $\Xhat$ and the branching property ensure that $(\Xhat(t), t\ge 0)$ is an homogeneous Markov process, and therefore can be written as the exponential of a MAP. The claim now boils down to finding its characteristics $(\Xi_{\pm}, Q_{\pm}, V_{\pm,\mp})$, and for obvious reasons of symmetry, we restrict to finding $(\Xi_+,Q_+,V_{+,-})$. 

\begin{itemize}[leftmargin=*]
\item[$\rhd$] \textsc{Determining the law of $\Xi_+$}. This is essentially done in \cite{BBCK}, but we recall the main ideas for the sake of completeness. The branching structure enables us to focus on the law of $(\Xi_+(t),0\le t\le b_{\Lcal(1)})$. Let \added{$f, g$} be two nonnegative measurable functions defined on the space of finite càdlàg paths and on $\R$ respectively. Therefore, we want to compute
\begin{align*}
&\Ehat_1 \left[ \added{f}\left(\log (\Xhat(s)), 0\le s< b_{\Lcal(1)}\right) \added{g}\left(\log \frac{\Xhat(b_{\Lcal(1)})}{\Xhat(b_{\Lcal(1)}^-)}\right) \mathds{1}_{\forall 0\le s\le b_{\Lcal(1)}, \; \Xhat(s)>0}\right] \\
&= \Eb_+\left[\sum_{t>0} |\Delta X(t)|^{\omega} \added{f}\left(\log (X(s)), 0\le s< t\right) \added{g}\left(\log \frac{-\Delta X(t)}{X(t^-)}\right) \mathds{1}_{\Delta X(t)<0}\right] \\
&= \Eb_+\left[\sum_{t>0} \mathrm{e}^{\omega \xi_+(t^-)} \added{f}\left(\xi_+(s), 0\le s< t\right) \added{g}\left(\log \left|1-\mathrm{e}^{\Delta \xi_+(t)}\right|\right)\left|1-\mathrm{e}^{\Delta \xi_+(t)}\right|^{\omega} \mathds{1}_{\Delta \xi_+(t)<0}\right] \\
&= \Eb_+\left[ \int_0^{\infty} \mathrm{d}t \mathrm{e}^{\omega \xi_+(t)} \added{f}\left(\xi_+(s), 0\le s< t\right) \right] \int_{-\infty}^0 \Lambda_+(\mathrm{d}x) \left|1-\mathrm{e}^{x}\right|^{\omega} \added{g}\left(\log \left|1-\mathrm{e}^{x}\right|\right) \\
&= \Eb_+\left[ \int_0^{\infty} \mathrm{d}t \mathrm{e}^{\omega \xi_+(t)} \added{f}\left(\xi_+(s), 0\le s< t\right) \right] \int_{-\infty}^0 \widetilde{\Lambda}_+(\mathrm{d}x) \mathrm{e}^{\omega x} \added{g}(x),
\end{align*}
where we used the compensation formula. Thus, under $\Phat_1$, on the event that $\Xhat(s)>0$ for all $s\in[0,b_{\Lcal(1)}]$, the two processes $\log \frac{\Xhat(b_{\Lcal(1)})}{\Xhat(b_{\Lcal(1)}^-)}$ and $\left(\log (\Xhat(s)), 0\le s< b_{\Lcal(1)}\right)$ are independent. The former has the law $-\psi_+(\omega)^{-1} \mathrm{e}^{\omega x} \widetilde{\Lambda}_{+\big|(-\infty,0)}(\mathrm{d}x)$, and the latter is distributed as $\xi_+$ killed according to $\exp(\omega \xi_+(t))$, so that it gives a Lévy process with Laplace exponent $q\mapsto\psi_+(\omega+q)$. Note that, in particular, $b_{\Lcal(1)}$ is an exponential random variable with parameter $-\psi_+(\omega)$. On the second hand, we can do the same for $\xihat_+$. Recall the notation $(\eta_+,\nu_+)$ from Remark \ref{rk:Levy-Ito}. Denote by $T_1$ the first time when the compound Poisson process $\nu_+$ has a jump: $T_1$ is exponential with parameter $-\psi_+(\omega)$. Since these jumps arise according to $\mathrm{e}^{\omega x}\mathrm{d}t \cdot \widetilde{\Lambda}_{+ \big|(-\infty,0)}(\mathrm{d}x)$, its first jump $\Delta \nu_+(T_1)$ is distributed according to $-\psi_+(\omega)^{-1} \mathrm{e}^{\omega x} \widetilde{\Lambda}_{+\big|(-\infty,0)}(\mathrm{d}x)$, and is independent of the process $(\eta_+(s),0\le s<T_1)$. The latter, in turn, is $\eta_+$ killed at an independent exponential time with parameter $-\psi_+(\omega)$. Since the Laplace exponent of $\eta_+$ is by definition
\[ 
\psi_{\eta_+}(q):= \psi_+(\omega+q)-\psi_+(\omega),
\]
we get that $(\eta_+(s),0\le s<T_1)$ has Laplace exponent $q\mapsto \psi_+(\omega+q)$. Therefore, we obtain the same description, and this entails that $\Xi_+$ and $\xihat_+$ have the same distribution.

\item[$\rhd$] \textsc{Determination of $Q_+$}. Call $\Hhat_1$ the first time when $\Xhat$ becomes negative. Since $X$ always remains positive when started from a positive mass, $\Hhat_1$ corresponds to the first time when the spine picks a positive jump in the change of measure \ref{sec:change measure}. Therefore, $\Hhat_1$ can be written
\[
\Hhat_1 = \sum_{i=1}^{G} \tau_i,
\]
where $G$ is a random variable corresponding to the generation of the spine at which a negative particle is selected, and \added{$\tau_i = b_{\Lcal(i)}-b_{\Lcal(i-1)}$, $i\ge 1$}. Since on the event that the spine selects a negative jump, we have seen that $b_{\Lcal(1)}$ is exponential with parameter $-\psi_+(\omega)$, we may deduce from the branching property that the $\tau_i$'s form an independent family of exponential variables with parameter $-\psi_+(\omega)$. Moreover, $G$ is a geometric variable on $\N^*$ with probability of success $p$ given by
\[
p 
:=
\Phat_1(\Xhat(b_{\Lcal(1)})<0)
=
\Eb_+\left( \sum_{t>0} \frac{v_-(\omega)}{v_+(\omega)}\left|\Delta X(t)\right|^{\omega} \mathds{1}_{\Delta X(t)>0}\right).
\]
Again, the compensation formula for $\xi_+$ yields 
\[
p 
=
-\frac{1}{\psi_+(\omega)} \cdot \frac{v_-(\omega)}{v_+(\omega)}  \int_0^{\infty} \Lambda_+(\mathrm{d}x) |\mathrm{e}^x-1|^{\omega}
=
-\frac{\sigma_+(\omega)}{\psi_+(\omega)}.
\]
As a sum of a geometric number of independent exponential variables, $\Hhat_1$ is an exponential random variable with parameter 
\[
Q_+ = -\psi_+(\omega) \cdot p = \sigma_+(\omega).
\]
Therefore $Q_+ = \qhat_+$.
\item[$\rhd$] \textsc{Determination of $V_{+,-}$}. For $q\ge 0$, we have
\[
\Ehat_1[\mathrm{e}^{qV_{+,-}}] = \sum_{i=1}^{\infty} \Ehat_1\left[\left(\frac{|\Xhat(\Hhat_1)|}{\Xhat(\Hhat_1^-)}\right)^q\mathds{1}_{G=i}\right]
\]
Let $a_i := \Ehat_1\left[\left(\frac{|\Xhat(\Hhat_1)|}{\Xhat(\Hhat_1^-)}\right)^q\mathds{1}_{G=i}\right]$, $i\ge 1$. Then, for $i\ge 2$, \added{conditioning on the spine at time $b_{\Lcal(1)}$ and using the Markov property yields}
\begin{align*}
a_i &= \Eb_+ \left( \sum_{t>0} |\Delta X(t)|^{\omega} \mathds{1}_{\Delta X(t)<0} \cdot \Ehat_{-\Delta X(t)}\left[ \left(\frac{|\Xhat(\Hhat_1)|}{\Xhat(\Hhat_1^-)}\right)^q\mathds{1}_{G=i-1}\right]\right) \\
&= \Eb_+\left( \sum_{t>0} |\Delta X(t)|^{\omega} \mathds{1}_{\Delta X(t)<0}\right) \cdot a_{i-1},
\end{align*}
by self-similarity. Hence, $(a_i)_{i\ge 1}$ is a geometric progression with common ratio 
\[
\Eb_+\left( \sum_{t>0} |\Delta X(t)|^{\omega} \mathds{1}_{\Delta X(t)<0}\right) 
=
-\frac{1}{\psi_+(\omega)} \int_{-\infty}^0 |\mathrm{e}^x-1|^{\omega} \Lambda_+(\mathrm{d}x),
\]
by an application of the compensation formula. Moreover, by another use of the compensation formula, the first term is 
\begin{align*}
a_1 &= \Ehat_1\left[\left(\frac{|\Xhat(\Hhat_1)|}{\Xhat(\Hhat_1^-)}\right)^q\mathds{1}_{G=1}\right] \\
&= \Eb_+\left[\sum_{t>0} \mathds{1}_{\Delta X(t)>0} \frac{v_-(\omega)}{v_+(\omega)} |\Delta X(t)|^{\omega} \left(\frac{|\Delta X(t)|}{X(t^-)}\right)^q\right] \\
&= \frac{v_-(\omega)}{v_+(\omega)}\Eb_+\left[ \sum_{t>0} \mathds{1}_{\Delta \xi_+(t)>0} \mathrm{e}^{\omega \xi_+(t^-)} \left|\mathrm{e}^{\Delta \xi_+(t)}-1 \right|^{q+\omega}\right] \\
&= - \frac{v_-(\omega)}{v_+(\omega)} \cdot \frac{1}{\psi_+(\omega)} \int_0^{\infty} |\mathrm{e}^{x}-1|^{q+\omega} \Lambda_+(\mathrm{d}x) \\
&= -\frac{\sigma_+(q+\omega)}{\psi_+(\omega)}.
\end{align*}
Finally, we get that 
\[
\Ehat_1[\mathrm{e}^{qV_{+,-}}] 
=
-\frac{\sigma_+(\omega+q)}{\psi_+(\omega)+\displaystyle{\int_{-\infty}^0 |\mathrm{e}^x-1|^{\omega}\Lambda_+(\mathrm{d}x)}}.
\]
Using that $\Kcal_+(\omega)=0$, we come to the conclusion that
\[
\Ehat_1[\mathrm{e}^{qV_{+,-}}] 
=
\frac{\sigma_+(\omega+q)}{\sigma_+(\omega)}.
\]
\end{itemize}
We have proved that $(\Xi_+,Q_+,V_{+,-}) \overset{\Lcal}{=} (\xihat_+,\qhat_+,\Uhat_{+,-}) $. Therefore the first claim of Theorem \ref{thm:spine} follows readily from Lemma \ref{lem:Fhat}.
\\ \\ \noindent \textbf{The spinal decomposition in the general case.}
We now prove the spinal decomposition under the tilted measure $\Phat_1$ by restricting to the previous case. More precisely, we want to prove that the law of $(\Xhat(t),n_t)_{0\le t<b_{\Lcal}}$ under $\Phat_1$ is the same as under $\Phat^{\uparrow}_1$, where $\Phat^{\uparrow}_1$ is the change of probability induced by $X^{\uparrow}$ \textit{via} section \ref{sec:change measure}. Indeed, since $X^{\uparrow}$ is nonnegative, the case of $\Phat^{\uparrow}_1$ comes under the previous paragraph, for which the spinal decomposition has just been established.

The definition of $\Phat_1$ clearly depends on the Eve cell. Note however that we have proved in Theorem \ref{thm:universality M(n)} that the exponent $\omega$ and the constants $(v_-(\omega),v_+(\omega))$ depend only on the growth-fragmentation (see also Proposition \ref{prop:admissibility dagger} for the relation between $X$ and $X^{\uparrow}$). Therefore, Proposition \ref{prop:spine-temporal} entails that the marginal law of $\Xhat$ only depends on the growth-fragmentation $\Xbf$. In order to prove that the law of $\Xhat$ itself is invariant within the same growth-fragmentation, we need to extend Proposition \ref{prop:spine-temporal} to finite-dimensional distributions. To avoid cumbersome notation, we state and prove the result for two times $s<t$. We want to show that for $z\ne 0$ and nonnegative measurable functions $f,g$ such that $f(\partial)=g(\partial)=0$,
\begin{multline} \label{eq:spine temporal extension}
v_{\sgn{z}}(\omega) |z|^{\omega}\Ehat_z\left(f(\Xhat(t)) g(\Xhat(s))\right) \\
=
\Ecal_z\left( \sum_{j\ge 1} g(X_j(s)) \Ecal_{X_j(s)}\left[\sum_{i\ge 1} v_{\sgn{X_i(t-s)}}(\omega)|X_i(t-s)|^{\omega}f(X_i(t-s))\right]\right).
\end{multline}
If one is willing to accept that $\Xhat$ is a Markov process, then this follows readily from Proposition \ref{prop:spine-temporal}. Otherwise, we can prove this directly. Let us mimic the proof of Proposition \ref{prop:spine-temporal}. Splitting over $u(t)$ as in equation \eqref{eq:decompose u(t)} and then conditioning on $\overline{\Fcal}_t$ and using the branching property, we get
\begin{multline*} 
v_{\sgn{z}}(\omega)|z|^{\omega}\Ehat_z \left(f(\Xhat(t)) g(\Xhat(s))\mathds{1}_{b_{\Lcal(n+1)}>t}\right) \\
=
\Ecal_z \left(\sum_{|w|\le n} g(\Xcal_{w(s)}(s-b_{w(s)})) v_{\sgn{\Xcal_w(t-b_w)}}(\omega) |\Xcal_w(t-b_w)|^{\omega} f(\Xcal_{w}(t-b_{w})) \mathds{1}_{b_w\le t}\right).  
\end{multline*}
We may then split this again over $w(s)=w'$. Using the branching property, this gives
\begin{multline*} 
v_{\sgn{z}}(\omega)|z|^{\omega}\Ehat_z \left(f(\Xhat(t)) g(\Xhat(s))\mathds{1}_{b_{\Lcal(n+1)}>t}\right)
= \Ecal_z\bigg(\sum_{|w'|\le n} g(\Xcal_{w'}(s-b_{w'})) \mathds{1}_{b_{w'}<s}  \\
\times  \Ecal_{\Xcal_{w'}(s-b_{w'})} \bigg[ \sum_{|w|\le n-|w'|}  v_{\sgn{\Xcal_w(t-s-b_w)}}(\omega) |\Xcal_w(t-s-b_w)|^{\omega} f(\Xcal_{w}(t-s-b_{w})) \mathds{1}_{b_w< t-s}\bigg]\bigg).  
\end{multline*}
Now taking $n\rightarrow\infty$ yields the desired identity \eqref{eq:spine temporal extension}.
\\ \\ \noindent \textbf{Proof of the second assertion.} We finally prove the second assertion of Theorem \ref{thm:spine} directly in the general setting. We will limit ourselves to proving the statement for the first generation (this is easily extended using the branching property). Let $f$ be a nonnegative measurable functional on the space of càdlàg trajectories, and $g_j$, $j\ge 1$, be nonnegative measurable functionals on the space of multiset--valued paths. For $t>0$, denote by $(\Delta_j(t), j\ge 1)$ the sequence consisting of the value of $\Xcal_{\varnothing}(t)$, and all those jumps of $\Xcal_{\varnothing}$ that happened strictly before time $t$, ranked in descending order of their absolute value. Our goal is to show that
\[
\Ehat_1 \left(f(\Xcal_{\varnothing}(s), 0\le s\le b_{\Lcal(1)}) \prod_{j\ge 1} g_j(\Xbfhat_{0,j})\right)
=
\Ehat_1 \left(f(\Xcal_{\varnothing}(s), 0\le s\le b_{\Lcal(1)}) \prod_{j\ge 1} \Ebf_{\Delta_j(b_{\Lcal(1)})} \left[g_j(\Xbf)\right]\right).
\]
But,
\begin{multline*}
\Ehat_1 \left(f(\Xcal_{\varnothing}(s), 0\le s\le b_{\Lcal(1)}) \prod_{j\ge 1} g_j(\Xbfhat_{0,j})\right) \\
=
\Ecal_1 \left(\sum_{t>0} \frac{v_{\sgn{-\Delta \Xcal_{\varnothing}(t)}}(\omega)}{v_{+}(\omega)} |\Delta \Xcal_{\varnothing}(t)|^{\omega} f(\Xcal_{\varnothing}(s), 0\le s\le t) \prod_{j\ge 1} g_j(\Xbfhat_{0,j})\right),
\end{multline*}
and the definition of the $\Xbfhat_{0,j}$ together with the branching property give
\begin{multline*}
\Ehat_1 \left(f(\Xcal_{\varnothing}(s), 0\le s\le b_{\Lcal(1)}) \prod_{j\ge 1} g_j(\Xbfhat_{0,j})\right) \\
=
\Ecal_1 \left(\sum_{t>0} \frac{v_{\sgn{-\Delta \Xcal_{\varnothing}(t)}}(\omega)}{v_{+}(\omega)} |\Delta \Xcal_{\varnothing}(t)|^{\omega} f(\Xcal_{\varnothing}(s), 0\le s\le t) \prod_{j\ge 1} \Ebf_{\Delta_j(t)} \left[g_j(\Xbf)\right]\right).
\end{multline*}
Applying the change of measure backwards, we get the desired identity. Therefore Theorem \ref{thm:spine} is proved.

%
%
\section{A distinguished family of signed growth-fragmentations}
\label{sec:excursion}
Following \cite{AD}, we construct a particular family of signed growth-fragmentations. These can be seen in the upper half-plane by cutting at heights a path with real part given by a stable Lévy process, and imaginary part a positive Brownian excursion. \added{This can be done for any self-similarity index $\alpha$ in $(0,2)$, but for reasons that will be clarified later on, we take $\alpha$ to be in $(1,2)$}.

\subsection{Notation and setup}
We recall from \cite{AD} the following framework. All the definitions and results basically extend directly from the half-planar Brownian case.

\bigskip
\noindent \textbf{The excursion measure $\n^\alpha$.}
Let $(\Omega, \F, \Pb)$ be a complete probability space, on which $X^{\alpha}$ is an $\alpha$--stable Lévy process, and $Y$ an independent Brownian motion. Recall that the Laplace exponent of $X^{\alpha}$ is of the form
\begin{equation} \label{eq: Laplace X alpha}
    \psi_{\alpha}(q) := \frac{c_+-c_-}{1-\alpha}q + \int_{\R} (\mathrm{e}^{qy}-1-q\mathrm{e}^y\mathds{1}_{|y|<1}) \nu_{\alpha}(y)\mathrm{d}y,
\end{equation}
where
\[
\nu_{\alpha}(y):=c_{\sgn{y}} |y|^{-\alpha-1},
\]
is the L\'evy measure of $X^{\alpha}$, and $c_+,c_-$ are constants such that $c_++c_->0$. We will choose the value of $c_+$ and $c_-$ later on in \cref{sec: loc largest evolution}. Call $(\F_t)_{t\ge 0}$ the standard filtration associated with $(X^{\alpha},Y)$. Write $\Xscr$ for the space of càdlàg functions $x$ with finite duration $R(x)$, equipped with the standard $\sigma$-field generated by the coordinates. Let $\Xscr_0$ be the subset of such functions in $\Xscr$ that are \emph{continuous} and vanish at $R(x)$. Finally, let
\[
U := \left\{ u=(x,y) \in \Xscr\times \Xscr_0, \; u(0)=0 \; \text{and} \; R(x)=R(y) \right\} \quad \text{and} \quad U_{\partial} := U \cup \{\partial\},
\]
where $\partial$ is a cemetery state. For $u\in U$ we shall write $R(u):=R(x)=R(y)$. We endow this set with the product $\sigma$--field $\mathscr{U}_{\delta}$ and the filtration $(\mathcal{F}_t)_{t\ge 0}$ adapted to the coordinate process on $U$. Also, we write $(L_s,s\geq 0)$ for the local time at $0$ of the Brownian motion $Y$ and $\mathcal{T}_s$ its inverse.

We define on $(\Omega, \F, \Pb)$ the \emph{excursion process} $(\epsilon^{\alpha}_s, s>0)$ as in the case of planar Brownian motion in \cite{AD}, except that we take for the real part $X$ the $\alpha$--stable Lévy process $X^{\alpha}$ (which has discontinuities), namely:
\begin{itemize}
    \item[(i)] if $\mathcal{T}_s-\mathcal{T}_{s^-}>0$, then 
    \[\epsilon^{\alpha}_s : r\mapsto \left(X^{\alpha}_{r+\mathcal{T}_{s^-}}-X^{\alpha}_{\mathcal{T}_{s^-}}, Y_{r+\mathcal{T}_{s^-}}\right), \quad r\leq \mathcal{T}_s-\mathcal{T}_{s^-},\]
    \item[(ii)] if $\mathcal{T}_s-\mathcal{T}_{s^-}=0$, then $\epsilon^{\alpha}_s = \partial$.
\end{itemize}
Then it is not difficult to see that the excursion process $(\epsilon^{\alpha}_s)_{s>0}$ is a $(\F_{\mathcal{T}_s})_{s>0}$--Poisson point process (see \cite{RY}, Chap. XII, Theorem 2.4, for the one-dimensional case). We denote by $\n^\alpha$ its intensity, which is a measure on $U$, and we denote by $\n^\alpha_+$ and $\n^\alpha_-$ its restrictions to $U^+:=\{u=(x,y)\in U, \; y\ge 0\}$ and $U^-:=\{u=(x,y)\in U, \; y\le 0\}$. An easy calculation gives the following expression for $\n^\alpha$. 
\begin{Prop}
$\n^\alpha(\mathrm{d}x,\mathrm{d}y) = n(\mathrm{d}y)\Pb((X^{\alpha})^{R(y)} \in \mathrm{d}x)$, where $n$ denotes the one-dimensional (Brownian) Itô measure on $\mathscr{X}_0$, and $(X^{\alpha})^T:=(X_t, \, t\in [0,T])$.
\end{Prop}
\noindent Figure \ref{fig:excursion} shows a drawing of such an excursion.

\bigskip
\noindent\textbf{Descriptions of the excursion measure $\n^{\alpha}_+$.}
We first state for future reference the Markov property of $\n^\alpha_+$. For any $u\in U$ and any $a>0$, let $T_a := \inf\{0\le t\le R(u), \; y(t)=a\}$ be the hitting time of $a$ by $y$. Recall $\Fcal_t := \sigma(u(s),0\le s \le t)$.

\begin{Prop}{(Markov property under $\n^\alpha_+$)} \label{Markov under n}

\noindent Under $\n^\alpha_+$, on the event $\{T_a<\infty\}$, the process $\left(u(T_a+t)-u(T_a)\right)_{0\le t\leq R(u)-T_a}$ is independent of $\mathcal{F}_{T_a}$ and has the law of $(X^\alpha,Y)$ killed at the first hitting time of $\{\Im(z) = -a\}$. 
\end{Prop}

\noindent The proof of \cref{Markov under n} results from the Markov property under the Itô measure $n$ (cf. Theorem 4.1, Chap. XII in \cite{RY}), and the Markov property of $X^\alpha$.

We now recall Bismut's description of $\n^{\alpha}_+$, which is also a direct consequence of the analogous description of Itô's measure (see \cite[Theorem XII.4.7]{RY}).
\begin{Prop} \label{Bismut} (Bismut's description of $\n^{\alpha}_+$)

Let $\overline{\n^{\alpha}_+}$ be the measure defined on $\R_+\times U^+$ by
\[\overline{\n^{\alpha}_+}(\mathrm{d}t,\mathrm{d}u) = \mathds{1}_{\{0\leq t\leq R(u)\}} \mathrm{d}t \, \n^{\alpha}_+(\mathrm{d}u).\]
Then under $\overline{\n^{\alpha}_+}$ the "law" of $(t,(x,y))\mapsto y(t)$ is the Lebesgue measure $\mathrm{d}a$ and conditionally on $y(t)=a$, $u^{t, \leftarrow}=\left(u(t-s)-u(t)\right)_{0\leq s\leq t}$ and $u^{t,\rightarrow}=\left(u(t+s)-u(t)\right)_{0\leq s\leq R(u)-t}$ are independent with respective laws $(-X^{\alpha},Y)$ and $(X^{\alpha},Y)$ killed when reaching $\left\{\Im(z)=-a\right\}$.
\end{Prop}
\noindent Note that, unlike the planar Brownian case, there is a minus sign for the left part of the trajectory: this is because of the time-reversal, which involves the dual of the Lévy process $X^{\alpha}$. See Figure \ref{fig:Bismut} for an illustration. 
\begin{figure}[h] 
\begin{center}
\includegraphics[scale=0.75]{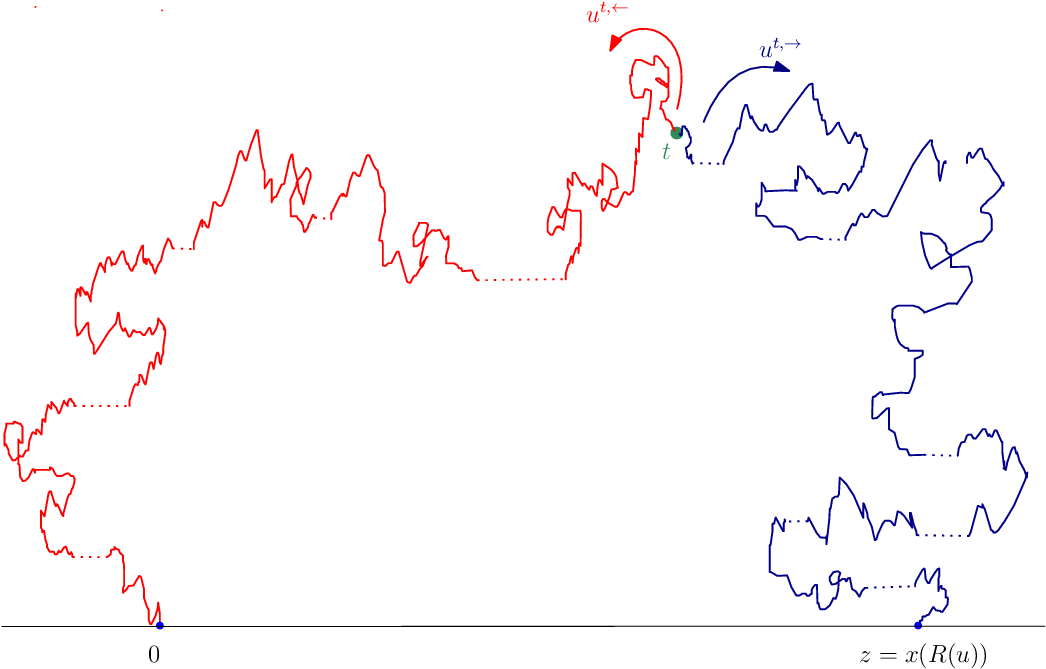}
\end{center}
\caption{An illustration of Bismut's description of $\n^{\alpha}_+$. Under $\overline{\n^{\alpha}_+}$,  conditionally on $y(t)$, the left part (in red) and right part (in blue) evolve independently and are identically distributed up to time-reversal.}
\label{fig:Bismut}
\end{figure}

\bigskip
\noindent\textbf{Disintegration of $\n^{\alpha}_+$}.
Finally, we disintegrate the infinite measure $\n^{\alpha}_+$ over the endpoint $z=x(R(u))$. This defines probability measures $\gtilde^{\alpha}_z$, $z\ne 0$, which are the laws of excursions $(X^{\alpha},Y)$ \emph{conditionally on} having displacement $z$. Introduce $\Pbb{\alpha,r}{a}{b}$ as the law of an $\alpha$--stable Lévy bridge of length $r$ between $a$ and $b$ (see \cite[Chapter VIII]{Ber}), and $\Pi_r$ as the law of a three-dimensional Bessel ($BES^3$) bridge of length $r$ from $0$ to $0$. Moreover, we denote by $(p^{\alpha}_t)_{t\ge 0}$ the transitions of the $\alpha$--stable Lévy process.

\begin{Prop} \label{disintegration prop}
We have the following disintegration formula
\begin{equation} \label{disintegration}
    \n^{\alpha}_+ = \int_{\R} \mathrm{d}z \, \frac{C_{\sgn{z}}}{|z|^{1+\alpha/2}} \, \gtilde^{\alpha}_z,
\end{equation}
where
\[C_{\pm}=\frac{\alpha}{2\sqrt{2\pi}} \int_0^{\infty} r^{\alpha/2} p^{\alpha}_1(\pm r) \mathrm{d}r,\]
and for $z\ne 0$, $\gtilde^{\alpha}_z$ is the probability measure defined by
\begin{equation} \label{gamma tilde}
    \gtilde^{\alpha}_z = \int_{\R_+} \mathrm{d}v \, \frac{p^{\alpha}_1(\sgn{z}v^{-1/\alpha})}{2\sqrt{2\pi}C_{\sgn{z}}v^{3/2+1/\alpha}} \,  \Pbb{\alpha, |z|^{\alpha}v}{0}{z}\otimes \Pi_{|z|^{\alpha}v}. 
\end{equation}
\end{Prop}
\begin{Rk}
The constants $C_{\pm}$ can be calculated (see \cite{KP-book}, Section 1). For example, if $X^{\alpha}$ is the so-called \emph{normalised} stable process of index $\alpha$, then
\[
C_{+} = \frac{\alpha}{2\sqrt{2}\pi} \Gamma\left(\frac{\alpha}{2}\right) \sin\left(\frac{\pi \alpha \rho}{2}\right),
\]
where $\rho=\Pb(X^{\alpha}(1)>0)$.
\end{Rk}
\begin{proof}
Although the proof follows exactly the same lines as in \cite{AD}, we include it here to highlight the importance of the sign of $z$, which does not show up in the Brownian case for symmetry reasons. Let $f$ and $g$ be two nonnegative measurable functions defined on $\mathscr{X}$ and $\mathscr{X}_0$ respectively. Applying It\^o's description of $n_+$ conditioned on its duration in terms of a Bessel bridge (see \cite{RY}, Chap. XII, Theorem 4.2), we get
\begin{align*}
    \int_{U} f(x)g(y) \, \n^{\alpha}_+(\mathrm{d}x,\mathrm{d}y) &= \int_{U} f(x)g(y) \, n_+(\mathrm{d}y) \Pb\left((X^{\alpha})^{R(y)}\in \mathrm{d}x\right) \\
    &=  \int_{\R_+} \frac{\mathrm{d}r}{2\sqrt{2\pi r^3}} \int_{\mathscr{X}} f(x) \, \Pi_r[g] \,\Pb\left((X^{\alpha})^{r}\in \mathrm{d}x\right).
\end{align*}
Now, decomposing on the value of $X^{\alpha}(r)$ yields 
\[\int_{U} f(x)g(y) \, \n^{\alpha}_+(\mathrm{d}x,\mathrm{d}y) = \int_{\R_+} \frac{\mathrm{d}r}{2\sqrt{2\pi r^3}} \int_{\R} \mathrm{d}z \, p^{\alpha}_r(z)  \Pi_r[g] \, \Ebb{\alpha, r}{0}{z} \left[f\right].\]
Using scale invariance, we have $p_r^{\alpha}(z)= r^{-1/\alpha}p_1^{\alpha}(r^{-1/\alpha}z)$. We finally perform the change of variables $v(r)=r/|z|^{\alpha}$ to get
\[\int_{U} f(x)g(y) \, \n^{\alpha}_+(\mathrm{d}x,\mathrm{d}y) = \int_{\R} \frac{\mathrm{d}z}{|z|^{1+\alpha/2}} \int_{\R_+} \mathrm{d}v \, \frac{p^{\alpha}_1(\sgn{z}v^{-1/\alpha})}{2\sqrt{2\pi}v^{3/2+1/\alpha}}  \Ebb{\alpha, v|z|^{\alpha}}{0}{z} \left[f\right] \Pi_{v|z|^{\alpha}}[g].\]
The constants $C_+$ and $C_-$ are then the normalising constants needed for $\gtilde_z^{\alpha}$ to be a probability measure.
\end{proof}

\subsection{Slicing excursions above levels}
We present the point of view that we will be interested in. We aim at describing a branching structure that shows up when slicing excursions at heights.

\bigskip
\noindent\textbf{Excursions above levels.}
We recall the following constructions from \cite{AD}. Let $u=(x,y)\in U^+$. For $a\ge0$, the set
\[\mathcal{I}(a) = \left\{s\in [0,R(u)], \; y(s)>a\right\},\] 
is a countable (possibly empty) union of disjoint open intervals, and for any such interval $I=(i_-,i_+)$, we write $u_{I}(s) := u(i_- +s)-u(i_-), 0\leq s\leq i_+ -i_-,$ for the restriction of $u$ to $I$, and $\Delta u_I = x(i_+)-x(i_-)$. We call $\Delta u_I$ the \emph{size} or \emph{length} of $u_I$, which may be negative. For $z=u(t)$, $0\leq t\leq R(u)$, and $0\leq a< \Im(z)$, we define $e_a^{(t)} = u_I$, where $I$ is the unique open interval in the above partition of $\mathcal{I}(a)$ containing $t$.

Moreover, set
\begin{align*}
u^{t, \leftarrow}
&:=
\left(u(t-s)-u(t)\right)_{0\leq s\leq t},
\\
u^{t,\rightarrow}
&:=
\left(u(t+s)-u(t)\right)_{0\leq s\leq R(u)-t}. 
\end{align*}
Define 
\[F^{(t)}:a\in[0, \Im(z)]\mapsto \Delta e_a^{(t)}=u^{t,\rightarrow}(T_a^{t,\rightarrow}) - u^{t,\leftarrow}(T_a^{t,\leftarrow}), \]
where
\[
T_a^{t,\leftarrow} 
:= 
\inf\left\{s\geq 0, \; y(t-s)=a\right\}
\quad \text{and} \quad
T_a^{t,\rightarrow} 
:= 
\inf\left\{s\geq 0, \; y(t+s)=a\right\}.
\]
Observe that for $\n_+^{\alpha}$--almost every excursion, $F^{(t)}$ is right-continuous on $(0,y(t)]$ for all $0\le t\le R(u)$ (use Lemma 4.8, Chapter 0 of \cite{RY}, and the fact that, under $\n_+^{\alpha}$, discontinuities of the real part and local minima of the imaginary part never occur at the same time).

\bigskip
\noindent \textbf{Loops above levels.}
As in \cite{AD}, \cref{Bismut} enables to prove that excursions under $\n^{\alpha}_+$ have no loop above any level.

\begin{Prop} \label{loop}
Let 
\[\mathscr{L} := \{u\in U^+, \; \exists 0\leq t\leq R(u), \; \exists 0\leq a<y(t), \; \Delta e_a^{(t)}(u) = 0\},\]
be the set of excursions $u$ having a loop remaining above some level $a$.
Then $\n^{\alpha}_+\left(\mathscr{L}\right) = 0$. 
\end{Prop}
\begin{proof}
We repeat the arguments of \cite[Proposition 2.7]{AD} for the sake of completeness. We first prove the result under $\overline{\n^{\alpha}_+}$, namely
\[\overline{\n^\alpha_+}\left(\{(t,u)\in \R_+\times U^+, \;  \; \exists 0\leq a<y(t), \; \Delta e_a^{(t)}(u) = 0\}\right) = 0.\]
Bismut's description of $\overline{\n^\alpha_+}$ gives
\begin{align*}
&\overline{\n^\alpha_+}\left(\{(t,u)\in \R_+\times U^+, \;  \; \exists 0\leq a<y(t), \; \Delta e_a^{(t)}(u) = 0\}\right) \\
&= \overline{\n^\alpha_+}\left(\{(t,u)\in \R_+\times U^+, \;  \; \exists 0\leq a<y(t), \; u^{t,\rightarrow}(T_a^{t,\rightarrow}) = u^{t,\leftarrow}(T_a^{t,\leftarrow})\}\right) \\
&= \int_0^{\infty} \mathrm{d}\alpha \, \Pb\left(\exists 0< a\le \alpha, \, X^1(T^1_{a}) = -X^2(T^2_{a})\right),
\end{align*}
where $X^1$ and $X^2$ are independent copies of $X^\alpha$, and $T^1_{a}$ and $T^2_{a}$ are hitting times of $a$ of independent Brownian motions. Using for example Section 4, Chap. III of \cite{RY}, $X^1(T^1_{a})$ and $X^2(T^2_{a})$ are independent $\theta$--stable Lévy processes with $\theta=\alpha/2$, and therefore $X^1(T^1_{a}) + X^2(T^2_{a})$ is again a $\theta$--stable process. Since $\theta<1$, points are polar for $X^1(T^1_{a}) + X^2(T^2_{a})$ (see \cite{Ber}, Chap. II, Section 5), so that $\Pb\left(\exists 0< a\le \alpha, \, X^1(T^1_{a}) = X^2(T^2_{a})\right)=0$. This proves our claim under $\overline{\n^\alpha_+}$. 

To prove the result under $\n^\alpha_+$, notice that if $u\in \mathscr{L}$, then the set of $t$'s satisfying the definition of $\mathscr{L}$ has positive Lebesgue measure (it contains all the times until the loop comes back to itself). This translates into 
\[\mathscr{L} \subset \left\{u\in U^+, \, \int_{0}^{R(u)} \mathds{1}_{\{\exists 0\leq a<y(t), \; \Delta e_a^{(t)}(u) = 0\}}\mathrm{d}t >0\right\}.\]
But, by the first step of the proof, 
\[\n^\alpha_+\left(\int_{0}^{R(u)} \mathds{1}_{\{\exists 0\leq a<y(t), \; \Delta e_a^{(t)}(u) = 0\}}\mathrm{d}t\right) = 0.\]
Thus $\displaystyle \int_{0}^{R(u)} \mathds{1}_{\{\exists 0\leq a<y(t), \; \Delta e_a^{(t)}(u) = 0\}}\mathrm{d}t=0$ for $\n^\alpha_+-$almost every excursion, and finally $\n^\alpha_+(\mathscr{L})=0$.
\end{proof}

\bigskip
\noindent \textbf{The locally largest excursion.}
Following the strategy of \cite{AD}, Proposition 2.8, one can establish the existence of a unique time on the excursion corresponding to the \emph{locally largest} excursion.
\begin{Prop} \label{locally largest prop}
For $u \in U^+$ and $0\le t \le R(u)$, let
\[
S(t) := \sup\left\{ a\in [0,y(t)], \quad \forall \, 0\leq a'\leq a, \; \big|F^{(t)}(a')\big| \geq \big|F^{(t)}(a'^-)-F^{(t)}(a')\big|\right\},
\]
and $S := \underset{0\leq t\leq R(u)}{\sup} S(t)$. For almost every $u$ under $\n^{\alpha}_+$, there exists a unique $0\leq \tb\leq R(u)$ such that $S(\tb) = S$. Moreover, $S=\Im(\zb)$ where $\zb=u(\tb)$.
\end{Prop}

\begin{figure}[h] 
\begin{center}
\includegraphics[scale=0.75]{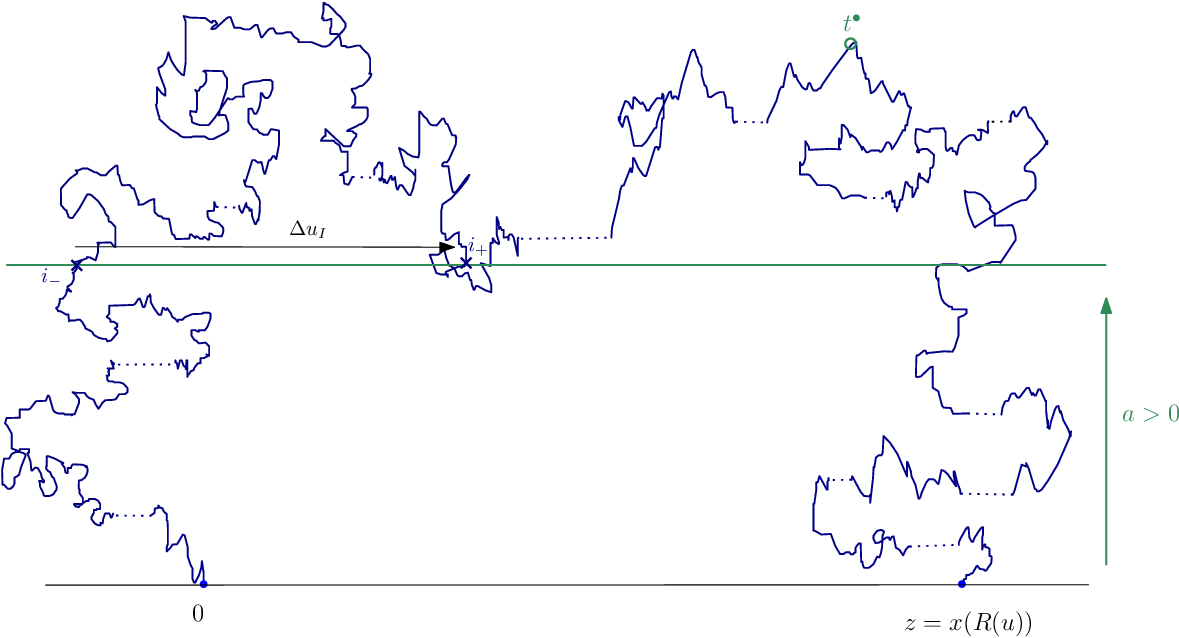}
\end{center}
\caption{Drawing of an excursion in the upper half-plane $\Hb$ and the locally largest excursion.}
\label{fig:excursion}
\end{figure}

By definition of $S$, $\big(e_a^{(\tb)}\big)_{0\leq a\leq \Im(\zb)}$ follows the largest excursion at each dislocation:  at any level $a$ where $F^{(\tb)}$ has a jump, the size $F^{(\tb)}(a)$ of the excursion $e_a^{(\tb)}$ is larger (in absolute value) than the length $F^{(\tb)}(a^-)-F^{(\tb)}(a)$ of the other excursion. For this reason we refer to $e_a^{(\tb)}$ as the \emph{locally largest excursion} at height $a$, and to $\left(\Xi(a) =\Delta e_a^{(\tb)}\right)_{0\leq a\leq \Im(\zb)}$ as the \emph{locally largest fragment}. See Figure \ref{fig:excursion}.

\begin{proof}
The proof being fairly technical, we just draw up a list with the main arguments. The reader interested in the details of the proof can get more information in \cite[Proposition 2.8]{AD}.

First, we prove that $S$ is attained. For this, we take a convergent sequence $(t_n,n\ge 1)$ of times such that $S(t_n)\rightarrow S$, and we denote by $t^\bullet$ the limit of $(t_n,n\ge 1)$. For any $a<S$, there exists $N\ge 1$ such that, for all $n\ge N$, $e_a^{(t^\bullet)} = e_a^{(t_n)}$. This implies that, up to $a$, $e^{(t^\bullet)}$ follows the locally largest excursion, \textit{i.e.} $S(t^\bullet)>a$. Hence $S(t^\bullet)\ge S$, and by maximality of $S$, $S(t^\bullet)=S$.

Now we want to prove that $S=y(t^\bullet)$. We claim that, for $\n^{\alpha}_+$--almost every excursion, and $0\le t\le R(u)$, the set
\begin{equation}\label{eq: A(t) open}
A(t) := \left\{ a\in [0,y(t)], \quad \forall \, 0\leq a'\leq a, \; \big|F^{(t)}(a')\big| \geq \big|F^{(t)}(a'^-)-F^{(t)}(a')\big|\right\},
\end{equation}
is open in $[0,y(t)]$. This follows from the right-continuity of $F^{(t)}$: if $a\in A(t)$, $a<y(t)$ and $F^{(t)}(a)\ne 0$ (which happens $\n^{\alpha}_+$--almost everywhere, see \cref{loop} above), then by right-continuity one can find a neighbourhood of $a$ where the inequality in \eqref{eq: A(t) open} holds. We then argue by contradiction: assume that $S<y(t^\bullet)$. Since $A(t^\bullet)$ is open, we have $A(t^\bullet) = [0,S)$. This means that, at level $S$, the reverse inequality to \eqref{eq: A(t) open} holds, that is:
\[
\big|F^{(t^\bullet)}(S)\big| < \big|F^{(t^\bullet)}(S^-)-F^{(t^\bullet)}(S)\big|.
\]
This implies that $F^{(t^\bullet)}$ has a jump at time $S$. Considering the excursion above $S$ which is detached from the one straddling $t^\bullet$, we get a set of times $t$ for which $e^{(t)}_a = e^{(t^\bullet)}_a$ for $a<S$, so that $S(t)\ge S$, but actually also $S\in A(t)$. Indeed, since $t^\bullet$ does not correspond to the locally largest fragment at height $S$, $t$ should (dislocation into two excursions with equal sizes is a $\n^{\alpha}_+$--negligible event). We conclude by the fact that $A(t)$ is open that $S(t)>S$ hence a contradiction.

Finally, the uniqueness statement can be proved also by contradiction. Assume the existence of two suitable times $t<t'$. Then, by considering the first height above which the excursions straddling $t$ and $t'$ differ, one sees that only one of them can correspond to the locally largest one (again owing to the fact that dislocation into two excursions with equal sizes is $\n^{\alpha}_+$--negligible). Hence the contradiction.
\end{proof}

\subsection{The branching property and a key many-to-one formula} \label{sec: key formula}
We here consider the path $u^{<a}$ obtained from $u$ under $\n^{\alpha}_+$ after removing the excursions above $a$, and closing up the time gaps. This can be defined formally as $u^{< a}_t:=u_{\tau^{< a}_t}$ if $t< A(R(u))$ and $u^{< a}_t:= u(R(u))$ if $t=A(R(u))$  where
\begin{equation} \label{eq: A_t time change}
A(t):= \int_0^t \mathds{1}_{\{y(s) \leq a \}}\mathrm{d}s \qquad \hbox{and} \qquad  \tau^{< a}_t:= \inf\{s> 0\,:\, A(s)>t\}.
\end{equation}
We call $\mathcal{G}_a$ the $\sigma$-field generated by $u^{< a}$, completed with the $\n^\alpha_+$--negligible sets, that is the $\sigma$--field carrying all the information about the trajectory \emph{below level $a$}. 

We now let $T_a:=\inf\{s>0, \; y(s)=a\}$ and we argue on the event that $T_a<\infty$. Let $(L_t^a)_{t\in [0,R(u)]}$ (resp.  $(\tau^{a}_s)_{s\in [0,L^a_{R(u)}]}$) be the (resp. inverse) local time process of $u$ at level $a$ and let $(e_s^a,\, s\in(0,L_{R(u)}^a))$ be the excursion process at level $a$ of $u$. It will be convenient to define $e_0^a$ and $e^a_{L_{R(u)}^a}$ respectively as the first and last parts of the trajectory $u$ between $\{\Im(z)=0\}$ and $\{\Im(z)=a\}$. Remark that, as a consequence of the strong Markov property under $\n^\alpha_+$ at time $T_a$ (cf. \cref{Markov under n}), on the event $T_a<\infty$ and conditionally on $\Fcal_{T_a}$, $(e_s^a,\, s\in(0,L_{R(u)}^a))$ forms a Poisson point process with intensity $\n^\alpha_+$ for the filtration $(\Fcal_{\tau_s^{a}}, 0\le s\le L^a_{R(u)})$, stopped at the first time when an excursion hits $\{\Im(z)=-a\}$.

For $a\ge 0$, we will write $(e_i^{a,+})_{i\ge 1}$ for the possible excursions that $u$ makes above $a$, ranked by descending order of their sizes $(z_i^{a,+})_{i\ge 1}$.

\begin{Prop} \label{branching gamma}(Branching property under $\gamma^\alpha_z$)

\noindent Let $z\in\R\setminus \{0\}$. For any $A\in \Gcal_a$, and for all nonnegative measurable functions $G_1,\ldots,G_k:U^+\rightarrow \R_+$, $k\ge 1$,
\[\gamma^\alpha_z\left(\mathds{1}_{\{T_a<\infty\}}\mathds{1}_A \prod_{i=1}^k G_i(e^{a,+}_i)\right) = \gamma^\alpha_z\left(\mathds{1}_{\{T_a<\infty\}}\mathds{1}_A \prod_{i=1}^k \gamma^\alpha_{z^{a,+}_i}[G_i]\right).\]
\end{Prop}

\begin{proof} 
We only sketch the proof under $\n_+^\alpha$: going from $\n_+^\alpha$ to $\gtilde^\alpha_z$ is then a technical step relying on disintegration over $z$ and some continuity argument (the reader can find more details for the Brownian case in \cite{AD}). We know from Lemma \ref{Markov under n} that on the event $\{T_a <\infty\}$, the trajectory $u$ after time $T_a$ is that of $(X^\alpha,Y)$ killed when hitting the line $\{\Im(z) = -a\}$. The Markov property at time $T_a$ and excursion theory entail that given the excursions below $a$, the excursions above $a$ form a Poisson point process on $U^+$ with intensity $L_{R(u)}^a \, \n^\alpha_+({\rm d} u)$. Conditioned on the sizes $(z_i^{a,+})_{i\ge 1}$, these excursions are independent with law $\gamma^\alpha_{z_i^{a,+}}$. The claim follows since $\Gcal_a$ is generated by the excursions below $a$ and the process of the sizes of the excursions above $a$.
\end{proof}

We call Bessel-stable (\emph{resp.} dual Bessel-stable) process a process in the upper half-plane whose real part is a copy of $X^\alpha$ (\emph{resp.} $-X^\alpha$) and whose imaginary part is an independent three-dimensional Bessel process starting at $0$. Under $\Pb$, let ${\mathfrak h}_1$ and ${\mathfrak h}_2={\mathfrak h}_2^z$ be respectively two independent Bessel-stable and dual Bessel-stable processes starting from ${\mathfrak h}_1(0)=0$ and ${\mathfrak h}_2(0)=z$. Define the analogues of \eqref{eq: A_t time change}, 
\begin{equation} \label{eq: A_i(t) time change}
A_i(t):= \int_0^t \mathds{1}_{\{{\Im(\mathfrak h}_i(s)) \leq a \}}\mathrm{d}s, \quad  \tau_i(t):= \inf\{s> 0\,:\, A_i(s)>t\}, \quad \text{for} \; i\in\{1,2\},
\end{equation}
and also $S_i^a := \sup\{t\ge 0\,:\, \added{\Im({\mathfrak h}_i(t))}\le a\}$ for the last passage time at $a$ of ${\mathfrak h}_i$, $i\in\{1,2\}$.

The following \emph{key formula} is a kind of \emph{many-to-one} formula for excursions cut at heights, and will be crucial in the rest of the paper. Write $V$ for the set of finite planar trajectories with càdlàg real part and continuous imaginary part. Recall that $(\tau^{a}_s)_{s\in [0,L^a_{R(u)}]}$ denotes the inverse local time at level $a$, and set
\begin{align*}
{\mathfrak{u}}_1^s &:= \Big( u(t), \, t\in [0, \tau^a_{s^-}]\Big),
 \\ 
{\mathfrak{u}}_2^s &:= \Big(u(R(u)-t),\, t \in [0, R(u)-\tau^a_s] \Big).
 \end{align*}
Then ${\mathfrak{u}}_1$ and ${\mathfrak{u}}_2$ are elements of $V$ which stand respectively for the trajectory of $u$ before the excursion $e^{a}_s$ and for the time-reversed trajectory of $u$ after the excursion $e^{a}_s$. We use the shorthand $s^+ \in [0, L_{R(u)}^a]$ to denote times $0\le s \le L_{R(u)}^a$ such that $e_s^{a} \in U^+$. Finally, let $\theta:=\alpha/2$.

\begin{Lem} \label{lem: key formula}
Let $F: V\times V\rightarrow \R_+$ be a nonnegative measurable function. Then 
\begin{multline} \label{eq: key formula excursions}
\gtilde^{\alpha}_z\left[\mathds{1}_{\{T_a<\infty\}}\sum_{ s^+ \in [0,L^a_{R(u)}]}  C^{-1}_{\sgn{\Delta e^{a}_s}} |\Delta e^{a}_s|^{1+\theta} F({\mathfrak u}_1^s, {\mathfrak u}_2^s) \right] \\
=
C_{\sgn{z}}^{-1}|z|^{1+\theta} \Eb\Big[F\left(({\mathfrak h}_1(t),t\in [0,S^a_1]\right),\left({\mathfrak h}_2^z(t),t\in [0,S^a_2])\right)\Big]
\end{multline}
\end{Lem}

\begin{proof} 
The proof follows the same lines as for Equation $(17)$ in \cite{AD}. We prove \eqref{eq: key formula excursions} for $F(u,v)=f(u)g(v)$, where $f,g: V\rightarrow \R_+$ are two nonnegative measurable functions. We first argue under the measure $\n^\alpha_+$. Recall from the discussion preceding \cref{branching gamma} that on the event $T_a<\infty$ and conditionally on $\Fcal_{T_a}$, the process $(e_s^a,\, s\in(0,L_{R(u)}^a))$ forms a Poisson point process with intensity $\n^\alpha_+$, stopped when an excursion hits $\{\Im(z)=-a\}$. Using the master formula \cite[Proposition XII.1.10]{RY} and the disintegration property (\cref{disintegration prop}), we therefore get
\begin{align}
&\n^\alpha_+\bigg( \mathds{1}_{\{T_a<\infty\}}\sum_{s^+ \in [0,L_{R(u)}^a]} C^{-1}_{\sgn{\Delta e^{a}_s}} |\Delta e^{a}_s|^{1+\theta} f(\mathfrak{u}_1^s) g(\mathfrak{u}_2^s)  \bigg)
 \label{beginning H excursion}\\
&= \n^\alpha_+\bigg(\mathds{1}_{\{T_a<\infty\}} \int_0^{R(u)} f\left(\mathfrak{u}_1^{L_r^a}\right)
 \mathrm{d}L_r^a \nonumber \\
 & \int_{-\infty}^{+\infty} \mathrm{d}x \, \Eb\left[ g(x+x'+X^\alpha(T^Y_{-a}-s),a+Y(T^Y_{-a}-s), 0\le s\le T^Y_{-a})\right]_{x'=X_r} \bigg), \nonumber 
\end{align}
where $T_{-a}^Y := \inf\{s>0, \; Y(s)=-a\}$ and $X_r = \Re(\mathfrak{u}_1^{L_r^a}(L_r^a))$.
The change of variables $x+X_r \mapsto x$ shows that it is also
\[
\n^\alpha_+\bigg(\mathds{1}_{\{T_a<\infty\}} \int_0^{R(u)} f\left(\mathfrak{u}_1^{L_r^a}\right)
 \mathrm{d}L_r^a \bigg)
 \int_{-\infty}^{+\infty} \mathrm{d}x \, \Eb\left[ g(x+X^\alpha(T^Y_{-a}-s),a+Y(T^Y_{-a}-s), 0\le s\le T^Y_{-a})\right].
\]
Conditionally on $Y$, $(X^\alpha_s,\, 0\le s\le T^Y_{-a})$ evolves as $X^\alpha$  stopped at time $T^Y_{-a}$. By duality for Lévy processes (see \cite[Section II.1]{Ber}), conditionally on $Y$, the "law" of $(x+X^\alpha(T^Y_{-a}-s),\, 0\le s \le T^Y_{-a})$ for $x$ sampled according to the Lebesgue measure is the "law" of $-X^\alpha$ with initial measure the Lebesgue measure, stopped at time $T^Y_{-a}$. On the other hand, the process $(a+Y(T^Y_{-a}-s), 0\le s\le T^Y_{-a})$ is a 3-dimensional Bessel process starting from $0$ and run until its last passage time at $a$, see Corollary 4.6, Chap. VII of \cite{RY}. In a nutshell, 
\[
\int_{-\infty}^{+\infty} \mathrm{d}x \, \Eb\left[ g(x+X(T^Y_{-a}-s),a+Y(T^Y_{-a}-s), 0\le s\le T^Y_{-a})\right]
=
\int_{-\infty}^{+\infty} \mathrm{d}z \, \Eb\Big[ g({\mathfrak h}_2^z(t),t\in [0,S_2^a])\Big].
\]
Moreover, by another application of the master formula,
\[
\n^\alpha_+\bigg( \mathds{1}_{\{T_a<\infty\}} f\left(\mathfrak{u}_1^{L_{R(u)}^a}\right)
  \bigg)  = \n^\alpha_+\bigg(\mathds{1}_{\{T_a<\infty\}} \int_0^{R(u)} f(\mathfrak{u}_1^{L_r^a}) \mathrm{d}L_r^a  \bigg)\n^\alpha_-(T_{-a}<\infty). 
\]

\noindent Since $\n^\alpha_-(T_{-a}<\infty)=\n^\alpha_+(T_{a}<\infty)$, this yields 
\[\n^\alpha_+\bigg(\mathds{1}_{\{T_a<\infty\}} \int_0^{R(u)} f(\mathfrak{u}_1^{L_r^a}) \mathrm{d}L_r^a  \bigg) = \n^\alpha_+\left(  f\left(\mathfrak{u}_1^{L_{R(u)}^a}\right) \mid T_a <\infty\right).\]
Now under $\n^\alpha_+(\cdot \mid T_a<\infty)$, $u$ up to its last passage time at $a$ has the law of ${\mathfrak h}_1$ up to $S_1^a$, whence
\begin{multline*}
 \n^\alpha_+\bigg( \mathds{1}_{\{T_a<\infty\}}\sum_{s^+ \in [0,L_{R(u)}^a]} C^{-1}_{\sgn{\Delta e^{a}_s}} |\Delta e^{a}_s|^{1+\theta} f(\mathfrak{u}_1^s) g(\mathfrak{u}_2^s) 
 \bigg) \\
  =
\Eb\Big[ f({\mathfrak h}_1(t),t\in [0,S_1^a]) \Big]
\int_{-\infty}^{+\infty} \mathrm{d}z\, \Eb\Big[ g({\mathfrak h}_2^z(t),t\in [0,S_2^a])\Big].
\end{multline*}

\noindent Finally, we disintegrate $\n^\alpha_+$ over $x(R(u))$ (cf. \cref{disintegration prop}) to get 
\begin{align*}
    \int_{-\infty}^{+\infty}& C_{\sgn{z}}\frac{\mathrm{d}z}{|z|^{1+\theta}} \, \gtilde^\alpha_z\left[\mathds{1}_{\{T_a<\infty\}} \sum_{s^+ \in [0,L_{R(u)}^a]} C^{-1}_{\sgn{\Delta e^{a}_s}} |\Delta e^{a}_s|^{1+\theta}f(\mathfrak{u}_1^s) g(\mathfrak{u}_2^s)
 \right] \\
 & =
\Eb\Big[  f({\mathfrak h}_1(t),t\in [0,S_1^a]) \Big]
\int_{-\infty}^{+\infty}\mathrm{d}z \, \Eb\Big[ g({\mathfrak h}_2^z(t),t\in [0,S_2^a])\Big].
    \end{align*}
By multiplying $g$ by any measurable function of $x(R(u))$, this entails that for Lebesgue-almost every $z\in \R$,
\begin{align*}
    & \gamma^\alpha_z\left[\mathds{1}_{\{T_a<\infty\}} \sum_{s^+ \in [0,L_{R(u)}^a]} C^{-1}_{\sgn{\Delta e^{a}_s}} |\Delta e^{a}_s|^{1+\theta} f(\mathfrak{u}_1^s) g(\mathfrak{u}_2^s) 
 \right] \notag \\
 & =
C_{\sgn{z}}^{-1}|z|^{1+\theta} \Eb\Big[  f({\mathfrak h}_1(t),t\in [0,S_1^a]) \Big]
 \Eb\Big[ g({\mathfrak h}_2^z(t),t\in [0,S_2^a])\Big]. 
\end{align*}
With some continuity argument, one can then prove that this holds for all $z$.
\end{proof}

\subsection{The law of the locally largest evolution} \label{sec: loc largest evolution}
Recall that $\theta=\frac{\alpha}{2}$, so that $\frac12<\theta<1$. We define $\eta^{\theta}$ under $P_z$ as the $\theta$--stable Lévy process starting at $z\in \R$. Recall that the Laplace exponent of $\eta^{\theta}$ is given by
\[
\psi_{\theta}(q) := \frac{c_+-c_-}{1-\theta}q + \int_{\R} (\mathrm{e}^{qy}-1-q\mathrm{e}^y\mathds{1}_{|y|<1}) \nu_{\theta}(y)\mathrm{d}y,
\]
where the density of the Lévy measure
\begin{equation} \label{eq:nu}
\nu_{\theta}(y):=c_{\sgn{y}} |y|^{-\theta-1},
\end{equation}
depends on constants $c_+,c_-$ such that $c_++c_->0$. An important feature is the positivity parameter $\rho:= P_0(\eta^{\theta}_1>0)$ which can be fixed by choosing $c_+,c_-$ to equal
\begin{equation} \label{eq:constants c_- and c_+}
c_-=\frac{\Gamma(1+\theta)}{\pi} \sin(\pi\theta(1-\rho)) \quad \text{and} \quad c_+=\frac{\Gamma(1+\theta)}{\pi} \sin(\pi\theta\rho).
\end{equation}
See \cite{CC} and \cite{KP}. Moreover, in order to retrieve the family \added{with no killing} introduced in \cite{BBCK} for $\theta<1$, we will in this subsection choose the following explicit constants. First, we fix $c_+, c_-$ so that $\theta\rho = 1/2$, which gives $c_+=\frac{\Gamma(1+\theta)}{\pi}$ and $c_-=-\frac{\Gamma(1+\theta)}{\pi} \cos(\pi\theta)$. \added{Notice that this implies $\alpha\in(1,2)$, which justifies our choice}. Finally, we take $X^{\alpha}$ to be an $\alpha$--stable Lévy process with positivity parameter $\rho$. It is important to note that, since $\alpha\rho=1$, $X^\alpha$ is \emph{spectrally negative}, meaning that it only has negative jumps, as can be seen from \eqref{eq:constants c_- and c_+} with $\alpha$ replacing $\theta$ (see \cite[Chapter VIII]{Ber} for more background). The process $X^\alpha$ being fixed, we now claim the following result.
\begin{Thm} \label{thm:loc largest}
Fix $z>0$. Let $\Xi = (\Xi(a), 0\le a\le \Im(\zb))$ denote the size of the locally largest fragment. Under $\gtilde^{\alpha}_z$, $(\Xi(a))_{0\le a< \Im(\zb)}$ is distributed as the positive self-similar Markov process $(Z_a)_{0\le a<\zeta}$ with index $\theta$ starting from $z$ whose Lamperti representation is 
\[Z_a = z\exp(\xi(\tau(z^{-\theta}a))), \]
where $\xi$ is the Lévy process with Laplace exponent
\begin{equation} \label{Lapl loc largest}
    \Psi(q) =  \int_{y>-\ln(2)} \left(\mathrm{e}^{q y}-1\right) \mathrm{e}^{-\theta y}\nu_{\theta}(-(e^y-1))\mathrm{d}y, \quad  q<2\theta+1,
\end{equation}
$\tau$ is the Lamperti time-change 
\[\tau(a) = \inf\left\{s\ge 0, \; \int_{0}^s \mathrm{e}^{\theta \xi(u)}\mathrm{d}u > a\right\},\]
and $\zeta = \inf\{a>0, Z_a=0\}$. 
\end{Thm}
\begin{Rk} \label{rk: loc largest negative}
One can give a similar description, starting from a negative $z<0$, for the locally largest evolution (which gives a negative cell process). In this case one would obtain a killing parameter in \eqref{Lapl loc largest}.
\end{Rk}
The remainder of this subsection is mostly devoted to the proof of Theorem \ref{thm:loc largest}. We start by recalling the main ingredients of the proof of Theorem 3.5 in \cite{AD}. Let $H$ be a bounded continuous nonnegative function defined on the space of finite càdlàg paths, and consider $a\ge 0$. On the event $\{a<\Im(\zb)\}$, we can write
\[
H(\Xi(b),\, b \in [0,a])
=
\sum_{s^+ \in [0,L_{R(u)}^a]} C^{-1}_{\sgn{\Delta e^{a}_s}}|\Delta e^{a}_s|^{1+\theta}F({\mathfrak u}_1^s, {\mathfrak u}_2^s), 
\]

\noindent where
\begin{equation}\label{eq:defF}
F({\mathfrak u}_1^s, {\mathfrak u}_2^s) = C_{\sgn{\Delta e^{a}_s}}|\Delta e^{a}_s|^{-1-\theta} H(\Xi(b),\, b \in [0,a]) \mathds{1}_{\{ e_a^{(\tb)} =e_s^a \}}.
\end{equation}

\noindent Note that the right-hand side is indeed a function of $({\mathfrak u}_1^s,{\mathfrak u}_2^s)$ since $(\Delta e_b^{(\tb)},\, b\in [0,a])$ is a measurable function of $u^{<a}$, and hence of $({\mathfrak u}_1^s,{\mathfrak u}_2^s)$. Apply the key formula \eqref{eq: key formula excursions}:
\begin{equation} \label{eq: key formula local largest}
 \gamma^\alpha_z\left[ H(\Xi(b),\, b \in [0,a])
 \mathds{1}_{\{a<\Im(\zb)\}} \right]  
=
C^{-1}_{\sgn{z}}|z|^{1+\theta} \Eb\Big[F\left(({\mathfrak h}_1(t),t\in [0,S^a_1]\right),\left({\mathfrak h}_2^z(t),t\in [0,S^a_2])\right)\Big].
\end{equation}

Now let $\eta$ be the process defined by $\eta_b:={\mathfrak h}_2(S_2^b) - {\mathfrak h}_1(S_1^b)$ for $b\ge 0$. For $i\in\{1,2\}$, the process $\Re( {\mathfrak h}_i(S_i^b)),\, b\ge 0$ is a $\theta$--stable process (using Corollary VII.4.6 and then Section III.4 of \cite{RY}). More precisely $\eta$ is a (c\`adl\`ag) $\theta$--stable process with positivity parameter $\rho'=1-\rho$, that is $\eta$ has law $\widetilde{P}_z$ started from $z$ given by the law of $-\eta^\theta$ under $P_{-z}$. Let $\Delta \eta_b := \eta_b- \eta_{b^-}$ stand for the jump of $\eta$ at time $b$. By definition of $F$ in \eqref{eq:defF}, we have
\[
F\left(({\mathfrak h}_1(t),t\in [0,S^a_1]\right),\left({\mathfrak h}_2^z(t),t\in [0,S^a_2])\right)
=
C_{\sgn{\eta_a}}|\eta_a|^{-1-\theta} H(\eta_b,\, b \in [0,a]) \mathds{1}_{ \forall \, b  \in [0,a], \; |{\eta}_b| \geq  |\Delta {\eta}_{b}| }.
\]
Finally, \eqref{eq: key formula local largest} rewrites
\[
\gamma^{\alpha}_z \left[H(\Xi(b),0\le b\le a)\mathds{1}_{a<\Im(z^{\bullet})}\right]
= 
\widetilde{E}_z\left(\frac{C_{\sgn{\eta_a}}}{C_{\sgn{z}}} \frac{|z|^{1+\theta}}{|\eta_a|^{1+\theta}}H(\eta_b, 0\le b\le a)\mathds{1}_{\forall 0\le b\le a, \; |\eta_b|\ge |\Delta \eta_b|}\right).
\]
Observe that under $\widetilde{P}_z$, $z>0$, on the event $E_a := \{\forall 0\le b\le a, \; |\eta_b|\ge |\Delta \eta_b|\}$, $\eta$ remains positive until time $a$. Therefore, the previous display simplifies to
\begin{equation} \label{eq: loc largest Radon-Nikodym}
\gamma^{\alpha}_z \left[H(\Xi(b),0\le b\le a)\mathds{1}_{a<\Im(z^{\bullet})}\right]
= 
\widetilde{E}_z\left(\frac{|z|^{1+\theta}}{|\eta_a|^{1+\theta}}H(\eta_b, 0\le b\le a)\mathds{1}_{\forall 0\le b\le a, \; |\eta_b|\ge |\Delta \eta_b|}\right).
\end{equation}
Furthermore, on the same event, $\eta$ can be written using the Lamperti representation of a $\theta$--stable Lévy process killed when entering the negative half-line, found in \cite{CC} (although we will take the form presented in \cite{KP}). More precisely, on this event, under $\widetilde{P}_z$ we can write $\eta_b = z\mathrm{e}^{\xi^0(\tau^0(b))}$, where 
\[
\tau^0(b):=\inf\left\{s\ge 0, \; \displaystyle{\int_0^s z^{\theta}\mathrm{e}^{\theta\xi^0(u)}\mathrm{d}u}\ge b\right\}=\displaystyle{\int_0^{b} \frac{\mathrm{d}s}{(\eta_s)^{\theta}}},
\]
and $\xi^0$ is a Lévy process with Laplace exponent
\begin{equation} \label{eq: Lapl exponent killed}
\Psi^0(q):= -\frac{c_+}{\theta} + \int_{\R^*} (\mathrm{e}^{qy}-1)\mathrm{e}^{y}\nu_{\theta}(-(\mathrm{e}^y-1))\mathrm{d}y, \quad -1<q< \theta.
\end{equation}
Note that compared to \cite{KP} we have inverted the role of the constants $c_+$ and $c_-$, because $\eta$ has the law of $-\eta^\theta$. Furthermore, observe that in this correspondence, the event $E_a$ is $\{\forall 0\le b\le \tau^0(a), \; \Delta \xi^0(b)>-\log(2)\}$. Thus Theorem \ref{thm:loc largest} is proved as soon as we have established the following lemma. 
\begin{Lem} \label{lem: martingale loc largest}
Under $\widetilde{P}_{z}$, the process
\[
M^{(\theta)}_a:= |z|^{1+\theta} \mathrm{e}^{-(1+\theta)\xi^0(a)} \mathds{1}_{\forall 0\le b\le a, \; \Delta \xi^0(b)>-\log(2)}, \quad a\ge 0,
\]
is a martingale with respect to the canonical filtration of $\xi^0$. Under the tilted probability measure $M^{(\theta)}_a\cdot \widetilde{P}_{z}$, the process $(\xi^0(b), 0\le b\le a)$ is a Lévy process with Laplace exponent $\Psi$ given by equation \eqref{Lapl loc largest}.
\end{Lem}

\noindent Indeed, the result then follows by a simple application of the optional stopping theorem. See also Lemma 17 in \cite{LGR} and Lemma 3.6 in \cite{AD}. We include the arguments for completeness. Rephrasing \eqref{eq: loc largest Radon-Nikodym}, we have obtained
\[
\gamma^\alpha_z[H(\Xi(b),\, 0\le b \le a)\mathds{1}_{\{a<\Im(\zb)\}}]
=
\widetilde{E}_z \left[  M^{(\theta)}_{\tau^0(a)} H\left(z \exp(\xi^0(\tau^0(b))),\, 0\le b\le a\right) \right].
\]

\noindent By the optional stopping theorem, for any $c>0$,
\begin{align*}
\widetilde{E}_z &\left[ M^{(\theta)}_{\tau^0(a)} H\left(z \exp(\xi^0(\tau^0(b))),\, b\in [0,a]\right) \mathds{1}_{\{c>\tau^0(a)\}}\right] \\
&=
\widetilde{E}_z \left[ M^{(\theta)}_{c} H\left(z \exp(\xi^0(\tau^0(b))),\, b\in [0,a]\right)\mathds{1}_{\{c>\tau^0(a)\}} \right].
\end{align*}

\noindent By \cref{lem: martingale loc largest}, the right-hand side is
\[
\widetilde{E}_z  \left[ H\left(z \exp(\xi(\tau(b))),\, b\in [0,a]\right)\mathds{1}_{\{c>\tau(a)\}} \right],
\]
where $\xi$ is the Lévy process with Laplace exponent $\Psi$ appearing in \cref{thm:loc largest}. It remains to take $c\to \infty$ and use dominated convergence to conclude the proof of \cref{thm:loc largest}. We now turn to proving \cref{lem: martingale loc largest}.

\begin{proof}[Proof of \cref{lem: martingale loc largest}]
By self-similarity, we may focus on the case $z=1$, in which case we write $P=\widetilde{P}_{1}$ for simplicity. We aim at computing the quantity
\[
E\left[\mathrm{e}^{(q-1-\theta)\xi^0(a)} \mathds{1}_{\forall 0\le b\le a, \; \Delta \xi^0(b)>-\log(2)} \right].
\]
To do this, we write $\xi^0(b)=\xi'(b)+\xi''(b)$, where $\xi''(b):=\sum_{0\le b\le a} \Delta \xi^0(b) \mathds{1}_{\Delta \xi^0(b)\le -\log(2)}$ is the Poisson point process of the small jumps of $\xi^0$. Then $\xi'$ and $\xi''$ are independent, and so the previous expectation is 
\[
E\left[\mathrm{e}^{(q-1-\theta)\xi^0(a)} \mathds{1}_{\forall 0\le b\le a, \; \Delta \xi^0(b)>-\log(2)} \right]
=
P(\xi''(a)=0)E[\mathrm{e}^{(q-1-\theta)\xi'(a)}].
\]
If we denote by $\Psi'$ and $\Psi''$ the Laplace exponents of $\xi'$ and $\xi''$, then we have 
\[
E\left[\mathrm{e}^{(q-1-\theta)\xi^0(a)} \mathds{1}_{\forall 0\le b\le a, \; \Delta \xi^0(b)>-\log(2)} \right]
=
\mathrm{e}^{a(\Psi''(\infty)+\Psi'(q-1-\theta))}.
\]
Therefore the calculation boils down to computing $\Psi''(\infty)+\Psi'(q-1-\theta)$. First of all, we know that the Lévy measure of $\xi''$ is the one of $\xi^0$ restricted to $(-\infty,-\log(2)]$, so that
\begin{equation} \label{eq: Lapl exponent small jumps}
\Psi''(q) 
=
\int_{y\le -\log(2)} (\mathrm{e}^{qy}-1) \mathrm{e}^y \nu_{\theta}(-(\mathrm{e}^y-1))\mathrm{d}y, \quad q>-1.
\end{equation}
Hence, by the expression of $\nu_{\theta}$ in \eqref{eq:nu},
\[
\Psi''(\infty) = - c_+ \int_{y\le -\log(2)} \frac{\mathrm{e}^y}{(1-\mathrm{e}^y)^{1+\theta}} \mathrm{d}y = \frac{c_+}{\theta}( 1-2^{\theta}).
\]
It remains to compute $q\mapsto \Psi'(q-1-\theta)$. By independence of $\xi'$ and $\xi''$, we have for all $-1<q<\theta$, $\Psi'(q)=\Psi^0(q)-\Psi''(q)$. Equations \eqref{eq: Lapl exponent killed} and \eqref{eq: Lapl exponent small jumps} provide
\begin{equation}
\Psi'(q) 
=
-\frac{c_+}{\theta} 
+ \int_{y>-\log(2)} (\mathrm{e}^{qy}-1)\mathrm{e}^{y}\nu_{\theta}(-(\mathrm{e}^y-1))\mathrm{d}y, \quad -1<q< \theta, \label{eq:psi''}
\end{equation}
This extends analytically to all $q<\theta$. We now fix $q<2\theta+1$, and we want to put $\Psi'(q-1-\theta)$ in a Lévy-Khintchine form. Replacing $q$ by $q-1-\theta$ in \eqref{eq:psi''}, we see that
\[
\Psi'(q-1-\theta) = -\mathbf{k}' + \int_{y>-\log(2)} (\mathrm{e}^{qy}-1)\mathrm{e}^{-\theta y}\nu_{\theta}(-(\mathrm{e}^y-1))\mathrm{d}y.
\]
with 
\begin{equation} \label{eq: k' def}
\mathbf{k}' := \frac{c_+}{\theta} - \int_{-\log(2)}^\infty (1-\mathrm{e}^{(1+\theta)y}) \mathrm{e}^{-\theta y} \nu_{\theta}(-(\mathrm{e}^y-1))\mathrm{d}y.
\end{equation}
In order to retrieve equation \eqref{Lapl loc largest}, it remains to prove that $\mathbf{k}'=\Psi''(\infty)$. The above integral can be computed as follows:
\begin{multline}
\int_{-\log(2)}^\infty (1-\mathrm{e}^{(1+\theta)y}) \mathrm{e}^{-\theta y} \nu_{\theta}(-(\mathrm{e}^y-1))\mathrm{d}y \\
=
c_+ \int_{-\log(2)}^0 (1-\mathrm{e}^{(1+\theta)y}) \frac{\mathrm{e}^{-\theta y}}{(1-\mathrm{e}^y)^{\theta+1}} \mathrm{d}y + c_-\int_{0}^\infty (1-\mathrm{e}^{(1+\theta)y}) \frac{\mathrm{e}^{-\theta y}}{(\mathrm{e}^y-1)^{\theta+1}}\mathrm{d}y.
\label{eq: k' splitting integrals}
\end{multline}
Start with the first integral:
\[
\int_{-\log(2)}^0 (1-\mathrm{e}^{(1+\theta)y}) \frac{\mathrm{e}^{-\theta y}}{(1-\mathrm{e}^y)^{\theta+1}} \mathrm{d}y
=
\int_{-\log(2)}^0 (\mathrm{e}^{-(1+\theta)y}-1) \frac{\mathrm{e}^{y}}{(1-\mathrm{e}^y)^{\theta+1}} \mathrm{d}y.
\]
By integration by parts (integrating $y\mapsto \frac{\mathrm{e}^{y}}{(1-\mathrm{e}^y)^{\theta+1}}$), this is 
\[
\int_{-\log(2)}^0 (1-\mathrm{e}^{(1+\theta)y}) \frac{\mathrm{e}^{-\theta y}}{(1-\mathrm{e}^y)^{\theta+1}} \mathrm{d}y
=
-(2^{1+\theta}-1)\frac{2^\theta}{\theta} + \frac{1+\theta}{\theta} \int_{-\log(2)}^0  \frac{\mathrm{e}^{-(1+\theta)y}}{(1-\mathrm{e}^y)^{\theta}} \mathrm{d}y.
\]
The change of variables $x=\mathrm{e}^y$ provides
\[
\int_{-\log(2)}^0 (1-\mathrm{e}^{(1+\theta)y}) \frac{\mathrm{e}^{-\theta y}}{(1-\mathrm{e}^y)^{\theta+1}} \mathrm{d}y
=
-(2^{1+\theta}-1)\frac{2^\theta}{\theta} + \frac{1+\theta}{\theta} B_{\frac12}(1-\theta, -(1+\theta)),
\]
where $B_{\frac{1}{2}}(x,y):=\displaystyle{\int_0^{1/2} t^{x-1}(1-t)^{y-1} \mathrm{d}t}$ is the incomplete beta function at $\frac12$. A similar calculation gives that the second integral in \eqref{eq: k' splitting integrals} is
\[
\int_{0}^\infty (1-\mathrm{e}^{(1+\theta)y}) \frac{\mathrm{e}^{-\theta y}}{(\mathrm{e}^y-1)^{\theta+1}}\mathrm{d}y
=
-\frac{1+\theta}{\theta} B(2\theta+1,1-\theta),
\]
where $B(x,y):=\displaystyle{\int_0^{1} t^{x-1}(1-t)^{y-1} \mathrm{d}t}$ is the beta function. It is well-known that $B(x,y) = \frac{\Gamma(x)\Gamma(y)}{\Gamma(x+y)}$, hence \eqref{eq: k' splitting integrals} boils down to 
\begin{multline} 
\int_{-\log(2)}^\infty (1-\mathrm{e}^{(1+\theta)y}) \mathrm{e}^{-\theta y} \nu_{\theta}(-(\mathrm{e}^y-1))\mathrm{d}y \\
=
-(2^{1+\theta}-1)\frac{2^\theta}{\theta}c_+ + \frac{1+\theta}{\theta} c_+ B_{\frac12}(1-\theta, -(1+\theta)) - c_- \frac{1+\theta}{\theta} \frac{\Gamma(2\theta+1)\Gamma(1-\theta)}{\Gamma(\theta+2)}
\label{eq: k' beta integrals}
\end{multline}
Now the two-variable function $B_{\frac12}(a,b)$ can be extended analytically to all $a,b\notin -\N$ \emph{via} the identity $aB_{1/2}(a,b+1) - bB_{1/2}(a+1,b) = 2^{-a-b}$ obtained by straightforward integration by parts. We then need to evaluate $B_{1/2}$ at $(-\theta,-\theta)$. But for $q>0$, $B_{1/2}(q,q) = \frac12 B(q,q) = \frac{\Gamma(q)^2}{2\Gamma(2q)}$ by symmetry. Uniqueness of analytic continuation implies that this must still hold for all $q\notin -\N$. This allows to write that $B_{1/2}(-\theta,-\theta) = \frac{\Gamma(-\theta)^2}{2\Gamma(-2\theta)}$, and therefore $(1+\theta) B_{\frac12}(1-\theta, -(1+\theta)) = 2^{1+2\theta} + \theta \frac{\Gamma(-\theta)^2}{2\Gamma(-2\theta)}$. In total, \eqref{eq: k' beta integrals} becomes
\[
\int_{-\log(2)}^\infty (1-\mathrm{e}^{(1+\theta)y}) \mathrm{e}^{-\theta y} \nu_{\theta}(-(\mathrm{e}^y-1))\mathrm{d}y 
=
\frac{2^{\theta}}{\theta}c_+ +c_+ \frac{\Gamma(-\theta)^2}{2\Gamma(-2\theta)} - c_- \frac{1+\theta}{\theta} \frac{\Gamma(2\theta+1)\Gamma(1-\theta)}{\Gamma(\theta+2)}.
\]
Recall that $c_+=\frac{\Gamma(1+\theta)}{\pi}$ and $c_-=-\frac{\Gamma(1+\theta)}{\pi} \cos(\pi\theta)$. We can then see after some calculations, using Euler's reflection formula, that the last two terms cancel out, leaving 
\[
\int_{-\log(2)}^\infty (1-\mathrm{e}^{(1+\theta)y}) \mathrm{e}^{-\theta y} \nu_{\theta}(-(\mathrm{e}^y-1))\mathrm{d}y 
=
\frac{2^{\theta}}{\theta}c_+.
\]
We finally come back to \eqref{eq: k' def} and obtain $\mathbf{k}' = \frac{c_+}{\theta}(1-2^{\theta})=\Psi''(\infty)$. This concludes the proof as we retrieve the Laplace exponent $\Psi$ of \eqref{Lapl loc largest}.
\end{proof}

\medskip
Our purpose is now to describe the law of the \emph{daughter} excursions of the locally largest one. We want to prove that the excursions which get detached along the way up to $z^\bullet$ are conditionally independent, and distributed as half-plane excursions with prescribed displacement. We rank these detached excursions $(\epsilon_i,i\ge 1)$ by descending order of their sizes $(z_i,i\ge 1)$ and we write $(a_i,i\ge 1)$ for the corresponding heights where they appear. We stress that both $(z_i, i\ge 1)$ and $(a_i, i\ge 1)$ are measurable with respect to $\Xi$, as they correspond to (opposite) jump sizes and jump times of $\Xi$. Let $z\in \R\backslash\{0\}$. 
\begin{Prop} \label{prop: offspring loc largest}
Under $\gamma^\alpha_z$, conditionally on $(z_i,a_i)_{i\ge 1}$, the excursions $(\epsilon_i)_{i\ge 1}$, are independent and each $\epsilon_i$ has law $\gamma^\alpha_{z_i}$.
\end{Prop}
\begin{proof}
We repeat the main ideas of \cite[Theorem 3.7]{AD}. Fix $n\in\N$ and take some measurable functions $f_1,\ldots, f_n: U \to \R_+$ and $g_1,\ldots, g_n: \R\times\R_+ \to \R_+$. Let $(\epsilon^{(a)}_i)_{i\ge 1}$ the sequence of sub-excursions detached from $\Xi$ below $a$, ranked by descending order of the absolute value of their sizes $z^{(a)}_i$, and let $b_i$ the corresponding jump time. It is enough to prove the claim for the $n$ first largest excursions below level $a$, namely:
\begin{equation} \label{eq: offspring n largest}
\gamma^\alpha_z\Big[ \mathds{1}_{\{a<\Im(\zb)\}}\prod_{i=1}^n f_i(\epsilon^{(a)}_i)g_i(z^{(a)}_i,b_i) \Big]
=
\gamma^\alpha_z\Big[ \mathds{1}_{a<\Im(\zb)}\prod_{i=1}^n \gamma^\alpha_{z_i^{(a)}}(f_i(\epsilon^{(a)}_i))g_i(z_i^{(a)},b_i) \Big].
\end{equation}
In view of applying the key formula, write 
\[
\prod_{i=1}^n f_i(\epsilon^{(a)}_i)g_i(z_i^{(a)},b_i)
=
\sum_{0<s^+ < L_{R(u)}^a} C_{\sgn{\Delta e^{a}_s}}^{-1} |\Delta e^{a}_s|^{1+\theta} F({\mathfrak u}_1^s, {\mathfrak u}_2^s),
\]

\noindent where
\[
F({\mathfrak u}_1^s, {\mathfrak u}_2^s) = C_{\sgn{\Delta e^{a}_s}} |\Delta e^{a}_s|^{-1-\theta}  \prod_{i=1}^nf_i(\epsilon^{(a)}_i)g_i(z^{(a)}_i,b_i) \mathds{1}_{\{ e_a^{(\tb)} =e_s^a \}}.
\]
Hence by \eqref{eq: key formula excursions},
\begin{multline*}
\gamma^\alpha_z\Big[\mathds{1}_{\{a<\Im(\zb)\}} \prod_{i=1}^n f_i(\epsilon^{(a)}_i)g_i(z_i^{(a)},b_i) \Big] \\
=
\Eb\Big[ \frac{C_{\sgn{\eta_a}}}{C_{\sgn{z}}}\frac{z^{1+\theta}}{\eta_a^{1+\theta}}\prod_{i=1}^n f_i(\varepsilon_i)g_i(z(\varepsilon_i),b(\varepsilon_i)) \mathds{1}_{\forall \, b  \in [0,a], \; |{\eta}_b| \geq  |\Delta {\eta}_{b}|}  \Big],
\end{multline*}
where the $\varepsilon_i,i\ge 1,$ form the ranked excursions of ${\mathfrak h}_1$ and ${\mathfrak h}_2$ above the future infimum $b(\varepsilon_i)$ of their imaginary parts before leaving $\{\Im\le a\}$ forever and $z(\varepsilon_i)$ is the size of the excursion  $\varepsilon_i$. But if we call $(b,\efrak_b)$ the process of excursions of $\hfrak_1$ or $\hfrak_2$ above the future infimum of their imaginary parts, then upon time-reversal, Lévy's theorem (Theorem VI.2.3 in \cite{RY}) and the Lévy property of $X^\alpha$ imply that it is a Poisson point process with intensity $2\mathds{1}_{\R_+}\mathrm{d}b\,  \n^\alpha_+(\mathrm{d}u)$. Conditionally on the heights and sizes $\{(b,z({\mathfrak e}_b)),\, b\ge 0 \}$, the excursions ${\mathfrak e}_b$ are independent with law $\gamma^\alpha_{z({\mathfrak e}_b)}$, whence
\begin{multline*}
\gamma^\alpha_z\Big[ \mathds{1}_{\{a<\Im(\zb)\}}\prod_{i=1}^n f_i(e^{(a)}_i)g_i(z_i^{(a)},b_i) \Big]
\\=
\Eb\Big[ \frac{C_{\sgn{\eta_a}}}{C_{\sgn{z}}}\frac{z^{1+\theta}}{\eta_a^{1+\theta}}\prod_{i=1}^n \gamma^\alpha_{z(\varepsilon_i)}(f_i(\varepsilon_i))g_i(z(\varepsilon_i),b(\varepsilon_i)) \mathds{1}_{\forall \, b  \in [0,a], \; |{\eta}_b| \geq  |\Delta {\eta}_{b}|}  \Big].
\end{multline*}
A backwards application of the key formula yields the claim \eqref{eq: offspring n largest}.

\end{proof}

\subsection{The temporal martingale} \label{sec: temporal mart excursions}
We first point out a temporal martingale for excursions cut at heights. For $a\ge 0$, recall that $(e_i^{a,+})_{i\ge 1}$ stands for the possible excursions that $u$ makes above $a$, ranked by descending sizes. Recall the notation $(\Gcal_a)_{a\ge 0}$ for filtration of events occurring below level $a$, and the definition of the constants $C_{\pm}$ introduced in the disintegration property \ref{disintegration prop}.
\begin{Prop} \label{prop:temporal mart}
The process
\[\Mcal^{\alpha,+}_a = \sum_{i\ge 1} C^{-1}_{\sgn{\Delta e_i^{a,+}}} \cdot |\Delta e^{a,+}_i|^{1+\theta}, \quad a\ge 0,\]
is a $(\Gcal_a)$-martingale.
\end{Prop}
\noindent This is a direct corollary of the key formula (taking $F=1$ in \cref{lem: key formula}), and the branching property of excursions above levels in \cref{branching gamma}. We note that -- once we establish that the excursion sizes form a growth-fragmentation -- \cref{prop:temporal mart} gives an example where the supermartingale in \cref{cor:supermartingale} is a martingale. However, temporal analogues of the genealogical martingales of \cref{sec:genealogical martingale} are in general \emph{not} martingales (see \cref{Rk: spine thm}(v)). 

The martingale $\Mcal^{\alpha,+}$ points at a natural change of measure. Recall the definition of $u^{<a}$ in Section \ref{sec: key formula}, and take $z\ne 0$. By Kolmogorov extension theorem, we may define on the same probability space a process $({\mathfrak U}_a^z,\, a> 0)$ such that for any $a> 0$, the law of ${\mathfrak U}_a^z$ is that of $u^{<a}$ under the probability measure $C_{\sgn{z}}z^{-1-\theta} {\mathcal M}^\alpha_a \mathrm{d}\gamma_z$. Our goal is to describe the law of $({\mathfrak U}_a^z, a\ge 0)$.

Our description involves the processes $\hfrak_1$ and $\hfrak_2^z$ of \cref{sec: key formula}. Let $a>0$ and  recall from \eqref{eq: A_i(t) time change} the definition of the time-changes $A_1$ and $A_2$ with respect to level $a$ (and their inverses $\tau_1$ and $\tau_2$). Set also $A_i(\infty) = \lim_{t\to\infty} A_i(t)$, $i=1,2$. Under $\Pb$, we define $\widetilde {\mathfrak U}_a^z$ as the process obtained by concatenating $\hfrak_1$ and $\hfrak^z_2$ when they leave $\{\Im(z)\le a\}$ forever, and removing everything above level $a$. More precisely, 
$$
\widetilde {\mathfrak U}_a^z(t)
:=
\left\{
\begin{array}{ll}
{\mathfrak h}_1(\tau_1(t)) &\hbox{ if } t \in \added{[0,A_1(\infty))}, \\
{\mathfrak h}_2^z(\tau_2(A_1(\infty)+A_2(\infty)-t)) &\hbox{ if } t \in [A_1(\infty),A_1(\infty)+A_2(\infty)],
\end{array}
\right. 
$$
\added{with the convention that ${\mathfrak h}_2^z(\tau_2(A_2(\infty))) := {\mathfrak h}_2^z(\tau_2(A_2(\infty))^-)$}. Thus $\widetilde {\mathfrak U}^z_a$ follows the trajectory of ${\mathfrak h}_1$ below level $a$ until it leaves \added{$\{\Im \le a\}$} forever, makes a jump to the last passage point at $a$ of ${\mathfrak h}_2^z$, then follows the time-reversed trajectory of ${\mathfrak h}_2^z$ below $a$ and finally ends up at $z$.

\begin{Thm} \label{thm:mu_z}
For any $z\ne 0$, the process $({\mathfrak U}_a^z,a>0)$ is distributed as $(\widetilde {\mathfrak U}_a^z,a>0)$.  
\end{Thm}

\begin{figure}[h] 
\begin{center}
\includegraphics[scale=0.8]{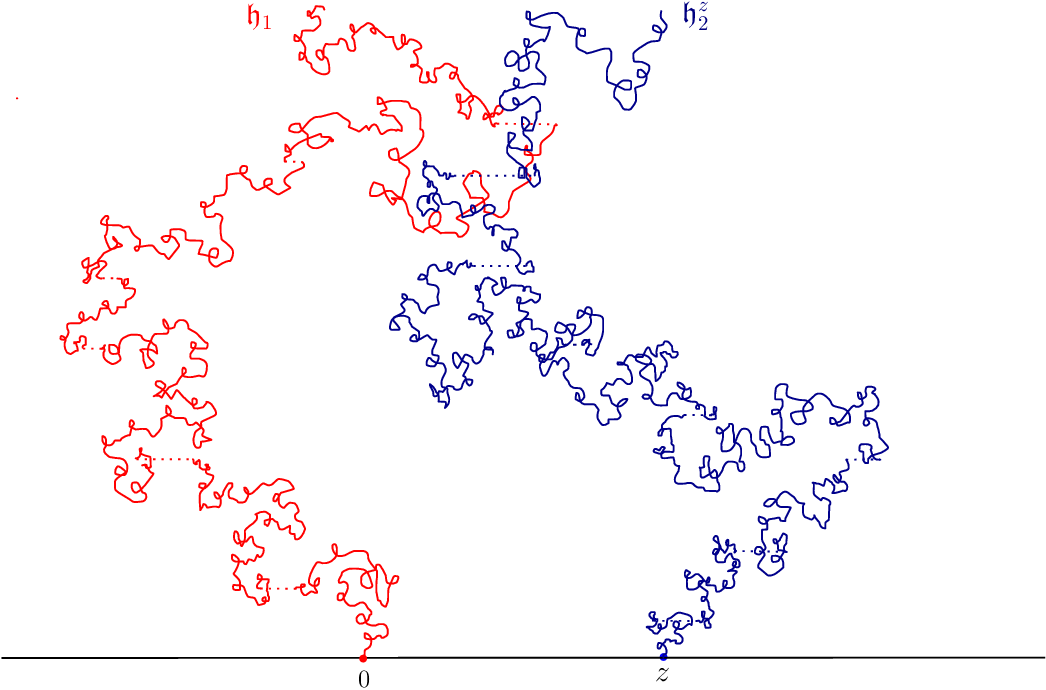}
\end{center}
\caption{Splitting the excursion according to the change of measure with respect to the martingale $\Mcal^{\alpha,+}$ (in the case when $X^{\alpha}$ is spectrally negative). The red and blue trajectories are independent and evolve as $\hfrak_1$ and $\hfrak^z_2$ respectively.}
\label{fig:change-measure}
\end{figure}

\noindent Under the change of measure, the path $u$ therefore splits into two infinite trajectories from $0$ and $z$ to $\infty$. See Figure \ref{fig:change-measure} (the picture shows the spectrally negative case).

\begin{proof}
The claim is included in the key formula \eqref{eq: key formula excursions}, by taking for $F({\mathfrak u}_1^s, {\mathfrak u}_2^s)$ and all $s$ some measurable function of $u^{<a}$.
\end{proof}

\begin{Rk} \label{rk: spine excursions}
\begin{enumerate}
\item Assuming that the \emph{sizes} of the excursions cut at heights form a signed growth-fragmentation process $\Xbf$ (this will be stated in the following section),  Theorem \ref{thm:mu_z} describes the law of the \emph{spine} defined in section \ref{sec:change measure}. Indeed, specifying the key formula \eqref{eq: key formula excursions} in the case when $F({\mathfrak u}_1^s, {\mathfrak u}_2^s)$ is a function of the size $\Delta e^a_s$, we get that the value of the spine at height $a$ for $\Xbf$ is given by looking at the size ${\mathfrak h}_2^z(\tau_2(A_2(\infty)))-{\mathfrak h}_1(\tau_1(A_1(\infty)))$ of the unbounded excursion of $\widetilde{\mathfrak U}_a^z$ above $a$. As the height $a$ increases, the spine is therefore given by the time-reversal of the difference of (the real part of) trajectories $R^{\alpha}$ and $L^{\alpha}$ coming down from infinity, taken at a Brownian hitting time. For this reason, we say that the spine amounts to \emph{targeting a point at infinity} (in the picture given by $\hfrak_1$, $\hfrak^z_2$). Finally, since the latter hitting times are subordinators of index $\frac12$ and $X^{\alpha}$ is an $\alpha$--stable Lévy process, this yields a stable Lévy process of index $\theta=\frac{\alpha}{2}$, and with positivity parameter $\rho':=1-\rho$.
\item We remark that, in the spectrally negative case, this is consistent with the martingales appearing in \cite{BBCK}. Indeed, they have the same form with a power given by $\omega_+ = \theta+\frac32 = (\theta+1)+\frac12$. Hence $\omega_+$ is the power appearing in $M_a$ plus one half. We can naturally retrieve this extra $1/2$ in our setting as follows. First, under $\gamma^{\alpha}_z$ (with $z>0$), let $\Xbf^+$ denote the family of \emph{positive} excursions obtained by removing from $\Xbf$ the negative sizes (together with their progeny). Then for any nonnegative measurable function $f$, one can leverage the \emph{many-to-one} formula given in \cref{lem: key formula} and the description of the spine in \cref{rk: spine excursions}(i) to express
\begin{equation} \label{eq: many-to-one h-transform}
\gamma^{\alpha}_z\left(\sum_{e\in \Xbf^+(a)} |\Delta e|^{\theta+3/2}f(\Delta e) \right)
=
E_z\left(f(Y_0^{\theta}(a))\sqrt{|Y_0^{\theta}(a)|}\right), 
\end{equation}
where under $P_z$, $Y_0^{\theta}$ is the $\theta$--stable Lévy process killed below $0$. On the other hand, since we have chosen $c_+$ and $c_-$ so that $\theta(1-\rho') = 1/2$, the $h$-transform used to condition the latter $\theta$-stable process to remain positive is given by $x\mapsto x^{\theta(1-\rho')}=\sqrt{x}$ (see \cite{CC}). 
Hence \eqref{eq: many-to-one h-transform} rewrites
\[
\gamma^{\alpha}_z\left(\sum_{e\in \Xbf^+(a)} |\Delta e|^{\theta+3/2}f(\Delta e) \right)
=
E_z\left(f(Y^{\theta}_+(a))\right),
\]
where under $P_z$, $Y_0^{\theta}$ is the $\theta$--stable Lévy process conditioned to remain positive.
We thus recovered (embedded in $\Xbf^+$) both the martingale exponent and the spine in the positive growth-fragmentation introduced in \cite[Proposition 5.2]{BBCK} (in the case $1/2<\theta<1$).
\end{enumerate}
\end{Rk}

\subsection{The growth-fragmentation embedded in half-planar excursions} \label{sec: GF excursions}
We now turn to the description of the cell system in terms of a growth-fragmentation process. The main results hold in general, but in order to retrieve the growth-fragmentation processes with no killing introduced by \cite{BBCK} for $\frac12<\theta<1$, we focus on the case when $X^\alpha$ is \emph{spectrally negative}, where the law of the locally largest fragment was explicited in \cref{thm:loc largest}. Recall that this amounts to set the positivity parameter $\rho$ of $X^{\alpha}$ so that $\theta\rho=\frac12$. On the other hand, let $(Z_a)_{0\le a<\zeta}$ the $\theta$--self-similar Markov process started from $z>0$ defined in \cref{thm:loc largest}, and extend the construction to $z<0$ by \cref{rk: loc largest negative}. Finally, construct the signed growth-fragmentation $\Zbf$ driven by $Z$.
Building on results from previous sections, we observe the following growth-fragmentation. 
\begin{Thm}
Under $\gamma^\alpha_z$, 
\[
\left(\left\{\left\{ \Delta e_i^{a,+}, i\ge 1\right\} \right\}, a\ge 0\right) \overset{\Lcal}{=}  (\Zbf(a), a\ge 0) .
\]
In particular, the process of the sizes of excursions cut at heights is a signed growth-fragmentation process.
\end{Thm}
\begin{proof}
The claim is essentially included in our work from the previous sections. We argue on the event that there is no loop above any level (\cref{loop}), that local minima of $y$ are distinct, and that there is no dislocation into two excursions with the same size, which has full $\gamma^\alpha_z$--probability. First, \cref{thm:loc largest} gives that, under $\gamma^\alpha_z$ the law of the locally largest fragment $\Xi$ in the stable-Brownian excursion is that of $Z$. Secondly, the conditional independence and conditional law of the \emph{offspring} of $\Xi$ was established in \cref{prop: offspring loc largest}. The only non-trivial statement is that we indeed recover all the excursions in the genealogy of $Z$. This statement does not bring any new idea as it is merely technical, so that we refer to the Brownian case \cite[Theorem 4.1]{AD}. Note that the discontinuities of $u$ do not conflict with conservativity at times when the growth-fragmentation cells divide: indeed, by independence, they almost surely occur at levels which are not local minima for the Brownian motion. 
\end{proof}
\noindent We now determine the spine under the change of measure $\Phat_z$.
\begin{Thm}
The vector $(C_+^{-1},C_-^{-1},\theta+1)$ is admissible for the locally largest evolution $\Xi$. Under $\Phat_z$, the spine $\Xhat$ defined in section \ref{sec:change measure} evolves as a $\theta$--stable Lévy process with Lévy measure $2\nu_{\theta}(-y)\mathrm{d}y$ and hence positivity parameter $\rho'=1-\rho$.
\end{Thm}
In particular, the positive growth-fragmentation $\Xbf^+$ obtained by removing from $\Xbf$ all the negative cells and their descendants is the same as that of \cite{BBCK}, for the appropriate self-similarity index $\frac12<\theta<1$. Indeed, by the many-to-one lemma, the law of the spine $\Xhat^+$ for the cell system of \emph{positive} masses is that of $\Xhat$ conditioned to stay positive, hence is distributed as the spine appearing in \cite{BBCK} for $\frac12<\theta<1$ (see \cref{rk: spine excursions} (ii)). Yet \cite[Theorem 5.1]{BBCK} entails that the spine characterises the law of the growth-fragmentation, and thus $\Xbf^+$ has the law of the growth-fragmentation process described in \cite{BBCK} for $\frac12<\theta<1$.

\begin{proof}
There are several ways to prove admissibility. For example, we use that $(\Mcal^{\alpha,+}_a, a\ge 0)$ is a martingale (Proposition \ref{prop:temporal mart}), and we condition on the first generation (the offspring of $\Xi$) to obtain
\[
C_{\sgn{z}}^{-1}|z|^{1+\theta}
=
\gamma^\alpha_z[\Mcal^{\alpha,+}_a]
=
C_+^{-1} \gamma^\alpha_z[|\Xi(a)|^{1+\theta}\mathds{1}_{a<\Im(\zb)}] + \gamma^\alpha_z\left[\sum_{s<a} C_{\sgn{-\Delta \Xi(s)}}^{-1} |\Delta \Xi(s)|^{1+\theta} \right].
\]
We then let $a$ tend to infinity and get that 
\[
\gamma^\alpha_z\left[\sum_{s>0} C_{\sgn{-\Delta \Xi(s)}}^{-1} |\Delta \Xi(s)|^{1+\theta} \right]
=
C_{\sgn{z}}^{-1}|z|^{1+\theta}.
\]
The Lamperti-Kiu representation of stable processes was established in \cite{CPR} (see Corollary 11): it is then a simple check to see that Theorem \ref{thm:spine} gives the same matrix exponent. Alternatively, we can use the description of Theorem \ref{thm:mu_z} of the spine as the difference of two $\theta$--stable processes, together with a version of Proposition \ref{prop:spine-temporal}.
\end{proof}

\subsection{Conditioning to continuously absorb at the origin} \label{sec: martingale absorption excursions}
We conclude by revealing another martingale for the growth-fragmentation cell system. Unlike in \cref{sec: temporal mart excursions}, this martingale will only be genealogical (in the form of \cref{thm:frag martingale}). We will see that the martingale converges $\gamma^\alpha_z$--almost surely towards the duration of the excursion, and describe the spine defined with respect to this change of measure.

We start by observing from \eqref{gamma tilde} that the law of the duration $R$ under $\gamma^\alpha_1$ is
\[
\gamma^\alpha_1(R\in\mathrm{d}r) = \frac{p^{\alpha}_1(r^{-1/\alpha})}{2\sqrt{2\pi}C_{+}r^{3/2+1/\alpha}} \mathrm{d}r,
\]
and likewise
\[
\gamma^\alpha_{-1}(R\in\mathrm{d}r) = \frac{p^{\alpha}_1(-r^{-1/\alpha})}{2\sqrt{2\pi}C_{-}r^{3/2+1/\alpha}} \mathrm{d}r.
\]
Moreover, the same equation shows that the duration $R$ under $\gamma^\alpha_z$, $z\ne0$, has the law of $|z|^\alpha R_{\sgn{z}}$, where $R_+$ and $R_-$ denote the law of $R$ under $\gamma^\alpha_1$ and $\gamma^\alpha_{-1}$ respectively. Observe that, because $\alpha<2$, $R_+$ and $R_-$ both have finite expectations that we denote by $w_+$ and $w_-$. 

We take the point of view of \cref{sec: GF excursions}, and we index the genealogy of the locally largest fragment $\Xi$ by the tree $\Ub$. In accordance with \cref{sec:genealogical martingale}, the collection of sizes at generation $n$ will be written $(\Xcal_u,|u|=n)$, and the $\sigma$--field generated by $(\Xcal_u,|u|\le n)$ denoted $\Gscr_n$. We can now claim:
\begin{Prop} Under $\gamma^\alpha_z$, the process
\[
\Mcal^{\alpha,-}(n) 
=
\sum_{|u|=n} w_{\sgn{\Xcal_u(0)}} |\Xcal_u(0)|^{\alpha},
\]
is a $(\mathscr{G}_n)$--martingale. Furthermore, it is uniformly integrable and converges $\gamma^\alpha_z$--almost surely and in $L^1$ to the duration $R$ of the excursion.
\end{Prop}
\begin{proof}
All the claims are a consequence of Lévy's theorem and the following identity:
\begin{equation} \label{eq: martingale M^- Lévy}
\Mcal^{\alpha,-}(n) := \gamma^\alpha_z(R \, | \, \Gscr_n), \quad \gamma^\alpha_z-\text{a.s.}
\end{equation}
Indeed, Lévy's theorem then implies that $\Mcal^{\alpha,-}(n) \to \gamma^\alpha_z(R \, | \, \Gscr_\infty)$ (a.s. and in $L^1$), where $\Gscr_\infty = \bigcup_{n\ge 0} \Gscr_n$. But since $R$ is $\Gscr_{\infty}$--measurable, this means $\Mcal^{\alpha,-}(n) \to R$. It remains to prove \eqref{eq: martingale M^- Lévy}. We restrict to $n=0$ (the general case follows by the branching property). We split $R$ over the daughter excursions $e_i$, $i\ge 1$, of $\Xi$. Since the set of times $0\le s\le R$ which are not straddled by such an excursion is Lebesgue-negligible, we have $R = \sum_{i\ge 1} R_i$, where $R_i$ is the duration of the sub-excursion $e_i$. We now use the conditional law of the offspring (\cref{prop: offspring loc largest}) to get
\[
\gamma^\alpha_z (R \, | \, \Gscr_0) = \sum_{i\ge 1} \gamma^\alpha_{\Delta e_i}(R).
\]
The self-similarity property discussed in the paragraph preceding the proposition entails 
\[
\gamma^\alpha_z (R \, | \, \Gscr_0) = \sum_{i\ge 1} w_{\sgn{\Delta e_i}} |\Delta e_i|^{\alpha},
\]
which is precisely \eqref{eq: martingale M^- Lévy} for $n=0$. 
\end{proof}
Now we turn to the description of the spine $\Xhat^-$ with respect to $(\Mcal^{\alpha,-}(n),n\ge 0)$. Recall the general change of measures in \cref{sec:change measure}, from which we borrow the notation, and write $\Phat_z^-$ for the change of measures started from $z$ relative to $\Mcal^{\alpha,-}$. The many-to-one formula in \cref{prop:spine-temporal} gives, for all $t\ge 0$, all nonnegative measurable function $f$ vanishing at $\partial$, and all $\overline{\Fcal}_t$--measurable nonnegative random variable $B_t$: 
\[
w_{\sgn{z}} |z|^{\alpha}\Ehat^-_z\left(f(\Xhat^-(t))B_t\right)
=
\Ecal_z\left( \sum_{i\ge 1} w_{\sgn{X_i(t)}}|X_i(t)|^{\alpha}f(X_i(t))B_t\right).
\]
Denote by $\Phat^+_z$ the change of measures relative to the other martingale $\Mcal^{\alpha,+}$, presented in \cref{prop:temporal mart}, and $\Xhat$ the corresponding spine. Then the many-to-one formula for $\Phat^+_z$ brings to
\begin{equation} \label{eq: law spine Mcal^-}
w_{\sgn{z}}C_{\sgn{z}} |z|^{\theta-1}\Ehat^-_z\big(f(\Xhat^-(t))B_t\big)
=
\Ehat^+_z\Big( w_{\sgn{\Xhat(t)}}C_{\sgn{\Xhat(t)}}|\Xhat(t)|^{\theta-1}f(\Xhat(t))B_t\Big).
\end{equation}
The previous formula extends to functionals of $(\Xhat^-(s),s\le t)$. Hence $\Xhat^-$ is a Doob $h$-transform of $\Xhat$ killed at the origin. Now recall from our description in \cref{thm:mu_z} that $\Xhat$ is a $\theta$--stable process. Thanks to \cite{KRS-conditioned} (in the case $\theta<1$), we may deduce that the law of $\Xhat^-$ is that of the stable process conditioned to absorb continuously at $0$.

We saw that the change of measures relative to $\Mcal^{\alpha,+}$ has a nice interpretation in terms of targeting a point \emph{at infinity} (see \cref{thm:mu_z} and \cref{rk: spine excursions}(i)). This is reminiscent of the situation obtained in the scaling limit from the peeling exploration of large Boltzmann planar maps. In \cite{BBCK}, the authors point out the existence of two martingales $\Mcal^+$ and $\Mcal^-$, for which they give the following descriptions. On the one hand, the spine relative to the martingale $\Mcal^+$ corresponds to targeting a point at infinity in the infinite-volume version of the planar map. On the other hand, the spine relative to $\Mcal^-$ corresponds to targeting a uniform point in the \emph{size-biased} planar map. As it turns out, this image persists in our excursion setting.

\begin{Thm}
Let $z\in\R^*$ and define the probability measure
\[
\overline{\gamma}^\alpha_z(\mathrm{d}t, \mathrm{d}u)
:=
(w_{\sgn{z}}|z|^\alpha)^{-1}\mathds{1}_{0\le t\le R(u)} \mathrm{d}t \gamma^\alpha_z(\mathrm{d}u).
\]
Under $\overline{\gamma}^\alpha_z$, the size $(X^{(t)}(a) := \Delta e^{(t)}_a, a>0)$, of the uniform point $t$ has the law of $\Xhat^-$ under $\gamma^\alpha_z$, \textit{i.e.} that of the $\theta$--stable process conditioned to be absorbed continuously at the origin.
\end{Thm}
\begin{proof}
Let $f$ a nonnegative measurable function defined on the set of finite càdlàg trajectories. From Bismut's description of $\n_+^{\alpha}$ (\cref{Bismut}), we see that 
\[
\overline{\n_+^\alpha}(f(X^{(t)}(b), 0\le b\le a)) \mathds{1}_{y(t)>a})
=
\int_0^\infty \mathrm{d}A \Eb(f(X^1(T^1_{-(A-b)})+X^2(T^2_{-(A-b)}), 0\le b\le a)) \mathds{1}_{A>a}, 
\]
where $X^1$ and $X^2$ are independent copies of $X^\alpha$, and $T^1_{-b}$ and $T^2_{-b}$ are hitting times of $-b$ of independent Brownian motions. With the same notation as in \cref{sec: loc largest evolution}, call $\eta^\theta_b:= X^1(T^1_{-b})+X^2(T^2_{-b})$, which is a $\theta$--stable Lévy process. By the Markov property, and a change of variables, this rewrites
\[
\overline{\n_+^\alpha}(f(X^{(t)}(b), 0\le b\le a)) \mathds{1}_{y(t)>a})
=
\int_a^\infty \mathrm{d}A E_0(h_a(\eta^\theta_{A-a}))
=
\int_0^\infty \mathrm{d}A E_0(h_a(\eta^\theta_A)), 
\]
where $h_a(x) := E_x(f(\eta^\theta_{a-b}, 0\le b\le a))$. Again, applying Bismut's description of $\n^\alpha_+$ backwards, one obtains
\[
\overline{\n_+^\alpha}(f(X^{(t)}(b), 0\le b\le a)) \mathds{1}_{y(t)>a})
=
\n_+^\alpha(h_a(x(R(u))) R(u)).
\]
We then disintegrate $\n_+^\alpha$ over the endpoint:
\[
\overline{\n_+^\alpha}(f(X^{(t)}(b), 0\le b\le a)) \mathds{1}_{y(t)>a})
=
\int_{\R^*} \frac{\mathrm{d}z}{|z|^{1+\theta}} C_{\sgn{z}}\gamma^\alpha_z(R) h_a(z).
\]
Now recall that in our notation, $\gamma^\alpha_z(R) = w_{\sgn{z}} |z|^{\alpha} = w_{\sgn{z}} |z|^{2\theta}$ for all $z\in\R^*$. Hence
\[
\overline{\n_+^\alpha}(f(X^{(t)}(b), 0\le b\le a)) \mathds{1}_{y(t)>a})
=
\int_{\R^*} \frac{\mathrm{d}z}{|z|^{1-\theta}}C_{\sgn{z}} w_{\sgn{z}} E_z(f(\eta^\theta_{a-b}, 0\le b\le a)).
\]
By duality for the Lévy process $\eta^\theta$ (cf. \cite[Section II.1]{Ber}), we can reverse the previous equation in time:
\begin{multline*}
\overline{\n_+^\alpha}(f(X^{(t)}(b), 0\le b\le a)) \mathds{1}_{y(t)>a})
= \\
\int_{\R^*} \frac{\mathrm{d}x}{|x|^{1-\theta}} w_{\sgn{x}} C_{\sgn{x}}\widetilde{E}_x\bigg(\frac{w_{\sgn{\eta_a}}C_{\sgn{\eta_a}}}{w_{\sgn{x}}C_{\sgn{x}}} \frac{|\eta_a|^{\theta-1}}{|x|^{\theta-1}}f(\eta_{b}, 0\le b\le a)\bigg),
\end{multline*}
where under $\widetilde{P}_x$, $\eta$ has the law of $-\eta^\theta$ under $P_{-x}$. On the other hand, disintegrating the left-hand side over the endpoint, we find  
\[
\overline{\n_+^\alpha}(f(X^{(t)}(b), 0\le b\le a)) \mathds{1}_{y(t)>a})
=
\int_{\R^*} \frac{\mathrm{d}x}{|x|^{1-\theta}} C_{\sgn{x}} w_{\sgn{x}} \overline{\gamma}^\alpha_x(f(X^{(t)}(b), 0\le b\le a)) \mathds{1}_{y(t)>a}).
\]
Putting all the pieces together, we end up with 
\begin{multline*}
\int_{\R^*} \frac{\mathrm{d}x}{|x|^{1-\theta}} w_{\sgn{x}} C_{\sgn{x}} \overline{\gamma}^\alpha_x(f(X^{(t)}(b), 0\le b\le a)) \mathds{1}_{y(t)>a})
= \\
\int_{\R^*} \frac{\mathrm{d}x}{|x|^{1-\theta}} w_{\sgn{x}} C_{\sgn{x}}\widetilde{E}_x\bigg(\frac{w_{\sgn{\eta_a}}C_{\sgn{\eta_a}}}{w_{\sgn{x}}C_{\sgn{x}}} \frac{|\eta_a|^{\theta-1}}{|x|^{\theta-1}}f(\eta_{b}, 0\le b\le a)\bigg).
\end{multline*}
Hence for Lebesgue-almost every $x$, 
\begin{equation} \label{eq: law of uniform point}
\overline{\gamma}^\alpha_x(f(X^{(t)}(b), 0\le b\le a)) \mathds{1}_{y(t)>a})
=
\widetilde{E}_x\bigg(\frac{w_{\sgn{\eta_a}}C_{\sgn{\eta_a}}}{w_{\sgn{x}}C_{\sgn{z}}} \frac{|\eta_a|^{\theta-1}}{|x|^{\theta-1}}f(\eta_{b}, 0\le b\le a)\bigg).
\end{equation}
A continuity argument that we feel free to skip ensures that \eqref{eq: law of uniform point} actually holds for all $x\in\R^*$. It remains to notice that the right-hand side of \eqref{eq: law of uniform point} is the same as in the description \eqref{eq: law spine Mcal^-} of the law of $\Xhat^-$ under $\gamma^\alpha_x$. We conclude that under $\overline{\gamma}^\alpha_x$, $X^{(t)}$ evolves as the $\theta$--stable Lévy process $\eta$ conditioned to absorb continuously at the origin.
\end{proof}

\addcontentsline{toc}{chapter}{References}
\bibliographystyle{alpha}
\bibliography{biblio}

\end{document}